\documentclass[11pt]{amsart}
\usepackage{amssymb,amsmath,amsthm,amsfonts,amsopn,url,hyperref,enumerate,mathtools,microtype,MnSymbol}
\usepackage[dvipsnames]{xcolor}
\usepackage[mathscr]{euscript}
\usepackage[normalem]{ulem}
\usepackage[all]{xy}
\usepackage{amscd}
\usepackage{tikz,tikz-cd}
\usepackage{comment}
\usepackage{mabliautoref}
\usepackage{xspace}
\usepackage{fullpage}
\usepackage{stmaryrd}

\newcommand{\DuBois}[1]{{\underline \Omega {}^0_{#1}}}
\newcommand{\field}[1]{\mathbb{#1}}
\newcommand{\N}{\field{N}}

\newcommand{\Q}{\field{Q}}
\newcommand{\bQ}{\field{Q}}
\newcommand{\bZ}{\field{Z}}

\newcommand{\bL}{\mathbb{L}}
\newcommand{\fm}{\mathfrak{m}}
\newcommand{\fn}{\mathfrak{n}}
\newcommand{\ideal}[1]{\mathfrak{#1}}
\newcommand{\m}{\ideal{m}}

\newcommand{\func}[1]{\mathrm{#1} \,}
\newcommand{\rank}{\func{rank}}
\newcommand{\Spec}{\func{Spec}}
\newcommand{\Proj}{\func{Proj}}
\newcommand{\depth}{\func{depth}}
\newcommand{\height}{\func{ht}}

\newcommand{\cf}{{\itshape{cf.} }}
\newcommand{\mJ}{{\mathcal{J}}}

\newcommand{\coker}{\func{coker}}
\newcommand{\im}{\func{im}}

\newcommand{\arrow}[1]{\stackrel{#1}{\rightarrow}}

\newcommand{\ra}{\rightarrow}

\DeclareMathOperator{\Image}{Image}
\DeclareMathOperator{\myR}{{\mathbb R}}
\DeclareMathOperator{\myL}{{\mathbb L}}
\DeclareMathOperator{\ann}{ann}

\DeclareMathOperator{\Hom}{Hom}
\DeclareMathOperator{\Ann}{Ann}

\newcommand{\be}{\begin{enumerate}}
\newcommand{\ee}{\end{enumerate}}

\newcommand{\li}
 {\leftfootline}

\newcommand{\into}{\hookrightarrow}

\newcommand{\cO}{\mathcal{O}}

\renewcommand{\phi}{\varphi}

\newcommand{\Cech}{{\v{C}ech} }

\newcommand{\fg}{finitely generated}

\newcommand{\KosClos}{{KH closure}\xspace}
\newcommand{\KosClosed}{{KH closed}\xspace}
\newcommand{\KosTestIdeal}{{KH test ideal}\xspace}
\newcommand{\CMPhantom}{{finite CM-ppd}\xspace}
\newcommand{\locCMPhantom}{{locally finite CM-ppd}\xspace}
\DeclareMathOperator{\colim}{colim}
\DeclareMathOperator{\Kos}{KH}
\DeclareMathOperator{\KH}{KH}
\DeclareMathOperator{\Hir}{Hir}
\newcommand{\mydot}{{{\,\begin{picture}(1,1)(-1,-2)\circle*{2}\end{picture}\,}}}
\newcommand{\traceIdealComplex}{degree zero trace\xspace}
\def\trcomp#1{{\mathrm{tr}}^0_{#1}}

\newcommand{\cl}{{\mathrm{cl}}}
\let\int\relax
\DeclareMathOperator{\int}{i}

\DeclareMathOperator{\tr}{tr}

\DeclareMathOperator{\img}{im}
\DeclareMathOperator{\coh}{coh}

\newcommand{\crd}{\color{red}} 

\def\KarlTodo#1{\textcolor{blue}%
{\footnotesize\newline{\color{blue}\fbox{\parbox{\textwidth-15pt}{\textbf{todo-Karl: } #1}}}\newline}}

\newcommand{\pitoR}{\pi:Y \to \Spec(R)}
\newcommand{\canon}{\Gamma(\omega_Y)}
\newcommand{\alt}{\mathrm{calt}}
\newcommand{\palt}{\mathrm{palt}}

\author{Neil Epstein}
\address{Department of Mathematical Sciences \\ George Mason University \\ Fairfax, VA  22030}
\email{nepstei2@gmu.edu}

\author{Peter M. McDonald}
\address{Department of Mathematics\\ Simon Fraser University \\ Burnaby, BC V5A 1S6}
\email{pma94@sfu.ca}

\author{Rebecca R.G.}
\address{Department of Mathematical Sciences \\ George Mason University \\ Fairfax, VA  22030}
\email{rrebhuhn@gmu.edu}

\author{Karl Schwede}
\address{Department of Mathematics \\ University of Utah \\ Salt Lake City, UT 84112}
\email{schwede@math.utah.edu}

\title[Closure operations via resolutions]{Closure operations induced via resolutions of singularities in characteristic zero}
\subjclass[2020]{Primary: 13A99.   
Secondary: 14F18, 13A35, 13D09, 13B22, 13C14, 13D45, 14B05, 14B15, 14E15
}
\keywords{
closure operation, resolution of singularities, alteration, test ideal, Cohen-Macaulay complex, multiplier ideal, interior operation, trace, tight closure, integral closure
}

\numberwithin{equation}{subsection}
\date{\today}
\begin{document}
\begin{abstract}
    Using the fact that the structure sheaf of a resolution of singularities, or regular alteration, pushes forward to a Cohen-Macaulay complex in equal characteristic zero with a differential graded algebra structure, we introduce a tight-closure-like operation on ideals in equal characteristic zero using the Koszul complex, which we call KH (Koszul-Hironaka) closure.  We prove it satisfies various strong colon capturing properties, a substantial case of the Brian\c{c}on-Skoda theorem, and it behaves well under finite extensions.  It detects rational singularities and is tighter than tight closure in equal characteristic zero.  Furthermore, its formation commutes with localization and it can be computed effectively.  On the other hand, the product of the KH closures of ideals is not always contained in the KH of the product, as one might expect.  
    
    We also explore a related closure operation (canonical alteration closure), induced by canonical modules of regular alterations, which detects KLT-type singularities in equal characteristic zero and which is closely related to tight closure in characteristic $p > 0$.  For parameter ideals we show both these closure operations coincide and reduce modulo $p \gg 0$ to tight closure.  
    We also study an intermediate operation, Hironaka preclosure, which satisfies numerous desired properties.  
    Finally, in the appendix we show that most of these operations coincide, and coincide with tight closure under reduction mod $p \gg 0$, for ideals that satisfy a characteristic-free analog of finite phantom projective dimension (for example, for powers of parameter ideals).  In particular, the tight closure of these ideals is also constant for $p \gg 0$.
\end{abstract}

\maketitle
\setcounter{tocdepth}{1} 
\tableofcontents

\section{Introduction}

It is well known that the singularities associated to Frobenius and tight closure theory are closely related to the singularities of the minimal model program defined by resolution of singularities, for instance see \cite{Fedder.FPurityAndRational,MehtaRamanathanFrobeniusSplittingAndCohomologyVanishing,HHmain,SmithFRatImpliesRat,HaraRatImpliesFRat,MehtaSrinivasRatImpliesFRat,HaraWatanabeFRegFPure,TakagiInterpretationOfMultiplierIdeals,HaraYoshida,MustataTakagiWatanabeFThresholdsAndBernsteinSato,BlickleMustataSmithDiscretenessAndRationalityOfFThresholds}.  One large omission was that while the characteristic $p > 0$ picture is closely tied to the theory of Hochster and Huneke's \emph{tight closure}, there doesn't seem to be a corresponding closure operation in {equal} characteristic zero that is induced by resolutions of singularities and their associated vanishing theorems.  

Of course, we have closure operations obtained via reduction mod $p$, see \cite{HochsterHunekeTightClosureInEqualCharactersticZero} ({but beware of \cite{BrennerKatzman.ArithmeticOfTightClosure})}, and closure operations coming from ultraproducts and ultra Frobenius \cite{Schoutens.NonStandardTightClosureForAffineCAlgebras,aschenbrennerschoutens}.  At the same time, we have Brenner's characteristic-free parasolid closure \cite{BrennerHowToRescueSolidClosure} (\cf solid closure \cite{HochsterSolidClosure}) which also satisfies numerous desired properties (although again, some are proven in {equal} characteristic zero by reduction to characteristic $p > 0$).  Several other interesting closure operations which apply in {equal} characteristic zero can be found in \cite{Brenner.ContinuousSolutionsAlgebraicForcing,EpsteinHochster.ContinuousAxesNatural,BrennerSteinbuch.TightClosureAndContinuousClosure}.
Another way to produce tight closure-like operations in any characteristic is to use extension and contraction from (weakly functorially assigned) big Cohen-Macaulay algebras.  In characteristic $p > 0$, such closures essentially agree with tight closure if the big Cohen-Macaulay algebra is large enough \cite{HochsterSolidClosure} and satisfy many of the same properties in any characteristic (see for instance \cite{dietz,RodriguezVillalobosSchwede.BrianconSkodaProperty} as well as \cite{heitmannepf} in view of \cite{BhattAbsoluteIntegralClosure}).  However, the only way we know to construct big Cohen-Macaulay algebras in characteristic $0$ uses characteristic $p > 0$ (or at least reduction to mixed characteristic).

In {equal} characteristic zero, the Matlis dual version of Grauert-Riemenschneider vanishing \cite{GRVanishing}, which has been generalized from varieties to $\bQ$-schemes under mild hypotheses in \cite{Murayama.VanishingForQSchemes}, guarantees that if $\pi : Y \to \Spec R$ is a resolution of singularities (or a regular alteration), then 
\[
    \myR\Gamma(Y, \cO_Y)
\]
is a Cohen-Macaulay complex \cite{Roberts.HomologicalInvariantsOfModulesOverCommutative} (\cf \cite{IMSW:2021}) under mild hypotheses on $R$.  In fact, it even has algebra-like properties as it can be viewed as a cosimplicial algebra or differential graded algebra.  This leads to the following question.

\begin{question*}
    How can you extend and contract an ideal from a differential graded $R$-algebra that is Cohen-Macaulay as an $R$-complex?
\end{question*}

While we explored several different approaches, we found the strongest results by using the Koszul complex on an ideal.

Suppose now that $R$ is a reduced excellent ring containing $\bQ$ with a dualizing complex, $J = (f_1, \dots, f_n) \subseteq R$ is an ideal, and $\pi : Y \to \Spec R$ is a resolution or regular alteration.  Then we define the \emph{Koszul-Hironaka (KH) closure of $J$} to be 
\[
    \begin{array}{rl}
        J^{\Kos} := & \ker\Big(R \to H_0 \big(K_{\mydot}({\bf f}; R) \otimes^{\myL} \myR \Gamma(Y, \cO_Y)\big) \Big)
        \\
        = & \big\{ x \in R \; \big|\; x \text{ annihilates $K_{\mydot}({\bf f}; R) \otimes^{\myL} \myR \Gamma(Y, \cO_Y)$ viewed as an object of $D(R)$}\big\}
    \end{array}
\]
where $K_{\mydot}({\bf f}; R)$ is a Koszul complex on a set of generators of $J$.  It turns out that $J^{\Kos}$ is independent of the choice of regular alteration or resolution and independent of the choice of generators of $J$.  It is also idempotent, meaning that $(J^{\Kos})^{\Kos} = J^{\Kos}$.  For all this and more, see \autoref{prop.kosIndependence} (whose proof crucially uses the differential graded algebra structure of $\myR \Gamma(Y, \cO_Y)$). The equivalence of the two characterizations is found in \autoref{rem.DerivedCatAnnihilatorThanksToBriggs}.

\KosClos also satisfies some of the usual accoutrements of tight closure such as persistence (\autoref{prop.PersistenceOfKos}), $J^{\Kos} = (JS)^{\Kos} \cap R$ for finite extensions $R \subseteq S$ (\autoref{prop.FiniteExtensionComputationKHClosure}), and more.  Notably, it satisfies strong forms of colon capturing:

\begin{theoremA*}[Colon capturing: \autoref{thm.cc}]
    Suppose $R$ is an excellent reduced {equidimensional} local ring of {equal} characteristic zero with a dualizing complex.
    Suppose $x_1, \dots, x_n \in R$ is a system of parameters.  Then for $t > a$ and $1 \leq k \leq n$
    \[
        (x_1^t, x_2, \dots, x_k)^{\Kos} : x_1^a \subseteq (x_1^{t-a}, x_2, \dots, x_k)^{\KH}.
    \]
    Furthermore,  
    \[
        (x_1, \dots, x_{k-1})^{\KH} : x_k \subseteq (x_1, \dots, x_{k-1})^{\KH}.
    \]
\end{theoremA*}

We also obtain a special version of the Brian\c{c}on-Skoda theorem, although some natural generalizations do not hold (see below):

\begin{theoremB*}[\autoref{thm.SkodaForPowerKos}]
    Suppose $R$ is an excellent domain of {equal} characteristic zero with a dualizing complex.  Then {for any ideal $J$ of $R$ which can be generated by $n$ elements,}
    \[ 
        \overline{J^n} \subseteq J^{\Kos}.
    \]
\end{theoremB*}

We have since built upon this idea.   In \cite{MaMcDonaldRGSchwede.BSForPseudorRational}, we utilized the perspective of the current paper in the context of the Brian\c{c}on-Skoda theorem.  Specifically, we generalized the Koszul complex to the Buchsbaum-Eisenbud or Eagon-Northcott complex to obtain a result that applies to powers of an ideal on both sides of the containment relation. This led directly to a proof of a conjecture of Huneke, namely uniform Artin-Rees and Brian\c con-Skoda theorems (see \cite{Huneke.UniformBoundsInNoetherianRings}).  This point of view also allowed us to improve and unify many previous Brian\c con-Skoda type theorems. 

It is also interesting to observe the ways that the \KosClos operation \emph{differs from other related closures} (like tight closure and plus closure in characteristic $p > 0$).

\begin{enumerate}
    \item \KosClos is strictly {\textbf{\emph{smaller}}} (that is, ``tighter'') than the reduction-modulo-$p$ version of tight closure in {equal} characteristic zero for finite type algebras over a field (or even the reduction-modulo-$p$ version of plus closure) \cite{HochsterHunekeTightClosureInEqualCharactersticZero}.  See \autoref{prop.KHContainedInPlusModP} and \autoref{sec.ComputationsInM2}.  This partially strengthens the above colon-capturing and Brian\c{c}on-Skoda theorems as the ideal on the right can be smaller.
    \item \KosClos{} {\textbf{\emph{measures rational singularities}}}.  Indeed, the following are equivalent thanks to \autoref{cor.RatSingsVsKoszulClosedParameterIdeal}.
    \begin{enumerate}
        \item $(R,\m)$ has rational singularities.
        \item $J = J^{\Kos}$ for all ideals.
        \item $J = J^{\Kos}$ for a single ideal generated by a full system of parameters.
    \end{enumerate}
    For tight closure in characteristic $p > 0$, having all ideals be tightly closed implies KLT singularities in the $\bQ$-Gorenstein setting \cite{HaraWatanabeFRegFPure}, which is strictly stronger than pseudo-rational singularities. 
    \item \KosClos{} {\textbf{\emph{can be computed}}}, {by} a computer, if one knows a resolution of singularities or if one knows the module $\Gamma(Y, \omega_Y)$ for $Y \to \Spec R$ a resolution of singularities (or regular alteration).  We include a simple Macaulay2 package and it is with this that we verify that \KosClos is strictly tighter than tight closure even for some diagonal hypersurfaces.  See \autoref{sec.ComputationsInM2}.  Tight closure and plus closure tend to be difficult to compute, although see \cite{McDermottTightClosurePlusClosureCubicalCones,Singh.TightClosureDiagonalHypersurfaces,Brenner.TightClosureAndProjectiveBundles,Brenner.SlopesOfVBOnProjectiveCurvesAndApplicationsToTC,KatzmanParameterTestIdeals}.  We believe the fact that \KosClos can be computed may make the Brian\c{c}on-Skoda and colon capturing results above more effective.
    \item The formation of \KosClos{} {\textbf{\emph{commutes with completion and localization}}}, and in fact commutes with any flat map $R \to S$ whose fibers have rational singularities, see \autoref{prop.FlatMapsWithRationalSingularitiesFibers}.  While plus closure commutes with localization, tight closure does not \cite{BrennerMonsky.TCLocalization, BorevitzNaderSandstromShaprioSimpsonZomback.LocalizationOfTightClosureS4Line}.  Also see \cite{Lyu.PermanencePropertiesOfSplintersViaUltra}.
\end{enumerate}

{
We would like to take a moment to emphasize the importance of having a computable closure operation in {equal} characteristic zero. When tight closure was initially introduced, the difficulty of computing it was immediately clear. Huneke noted in \cite[Example~1.2]{Huneke.TCParamAndGeometry} (originally written in 1997) that even studying the tight closure of monomial ideals in $k[x,y,z]/(x^3+y^3+z^3)$, where $k$ is a field of prime characteristic $p\neq3$, could prove quite difficult. At the time, McDermott had shown that $xyz\in (x^2,y^2,z^2)^*$ when $p<200$ \cite{McDermottTightClosurePlusClosureCubicalCones} before Singh was able to show the result for all $p\neq3$ (and a more general result for diagonal hypersurfaces) \cite{Singh.TightClosureDiagonalHypersurfaces}. It's worth pointing out that these arguments only apply to specific ideals in the ring and rely on the fact that the ring is a diagonal hypersurface to reduce the problem to a question about a polynomial ring in two variables.

Breakthrough work of Brenner \cite{Brenner.TightClosureAndProjectiveBundles, Brenner.SlopesOfVBOnProjectiveCurvesAndApplicationsToTC} expanded upon this, studying the tight closure in normal standard graded $k$-algebras (frequently of dimension 2) by investigating corresponding vector bundles and their subbundles on the corresponding projective variety. {Some of his work was in service of computing tight closure, and it was most effective in the 2-dimensional case.  Thus,} it is again worth noting the effort required just to study tight closure in two-dimensional rings.

The fact that KH-closure can (oftentimes) be easily computed, combined with the fact that it is smaller than characteristic zero tight closure, allows us to obtain tighter Brian\c{c}on-Skoda bounds among other things. To return to the example of the cubical hypersurface, working now over $\bQ$, consider $J:=(x^2,y^2,z^2)\subseteq \bQ[x,y,z]/(x^3+y^3+z^3)$. \cite{Singh.TightClosureDiagonalHypersurfaces} tells us the characteristic zero tight closure contains $xyz$, whereas in \autoref{exam.KHClosureExampleComputations} we show that $J^{\KH}=J$, giving a tighter bound on $\overline{J^3}$.
}

It is not completely surprising that a closure operation based on the object $\myR\Gamma(Y, \cO_Y)$ should detect rational singularities.  Indeed by \cite{KovacsRat,BhattDerivedDirectSummandCharP,Murayama.VanishingForQSchemes,Lyu.PropertiesBirationalDerivedSplinters}, $R$ has rational singularities if and only if $R \to \myR\Gamma(Y, \cO_Y)$ splits for one or equivalently any regular alteration $\pi : Y \to \Spec R$ (or even for any proper surjective map from a nonsingular scheme).

It is also worth noting that all the properties we show about \KosClos are consequences of resolution of singularities and the associated vanishing theorems, or are characteristic independent.  Our results about Koszul-Hironaka closure do not utilize reduction to characteristic $p > 0$.

However, the \KosClos operation does not behave as well as other common closure operations when it comes to products or powers of ideals.  
This is perhaps not surprising as we are not aware of a clean way to compare Koszul complexes of $I$ and $IJ$ or $I^n$.  Specifically, in \autoref{exam.KosClosureNoGeneralizedBS} we show that for an $n$-generated ideal $I$, it can happen that $\overline{I^{n+k-1}} \nsubseteq (I^k)^{\KH}$ for $k \geq 2$.  In other words, the generalized Brian\c{c}on-Skoda theorem fails for \KosClos.  Furthermore, it can happen that $I^{\Kos} I^{\Kos} \nsubseteq (I^2)^{\Kos}$ and that $x I^{\Kos} \neq (xI)^{\Kos}$ for $x$ a {nonzerodivisor}.  That is, \KosClos is not a semi-prime operation or a star operation, see \autoref{subsec.NotASemiprimeOrStarOperation}.  All of these examples were verified using the Macaulay2 package we created to compute \KosClos.

\subsection{Other resolution-based {equal} characteristic zero closures}

There is another (larger) closure operation in {equal} characteristic zero we study that is probably closer to tight closure, in that it detects KLT-type singularities (that is, that there exists a $\Delta \geq 0$ such that $(R, \Delta)$ is KLT).  We call this closure \emph{canonical alteration closure} and define it as 
\[L_M^\alt\colon=\bigcap_{\pitoR} L_M^{\cl_{\canon}}\]
where $\pi$ varies over regular alterations $\pi : Y \to \Spec R$ and we set $L_M^{\cl_{\canon}}$
to be the module closure associated to the $R$-module $\Gamma(Y,\omega_Y)$.  Note the modules $\Gamma(Y,\omega_Y)$ 
    are the multiplier submodules / Grauert-Riemenschneider submodules of $\omega_S$ where $R \subseteq S$ is a finite extension of $R$.  We show that if one uses the parameter test modules $\tau(\omega_S)$ in characteristic $p > 0$, instead of $\Gamma(Y, \omega_Y)$, then this closure operation coincides with tight closure, see \autoref{thm.TightClosureEqualsModuleClosureParameterTestModules}.  This gives some evidence already that this closure operation is closely related to tight closure.

    Back in {equal} characteristic zero, we observe, for any ideal $I \subseteq R$, that 
    \[ 
        I^{\KH} \subseteq I^{\alt} := I_R^{\alt} 
    \]
    {%
    and that the reduction to characteristic $p > 0$ tight closure (and plus closure) sits in between these two operations for finite type algebras over a field, see \autoref{prop.KHContainedInPlusModP} and \autoref{prop.TCIsContainedInCaltChar0}.
    Furthermore, we are able to show that all these closures are equal for parameter ideals, compare with \cite[Corollary 4.2]{Huneke.TCParamAndGeometry},  \cite[Theorem 1]{Vraciu.TightClosureAndLinkageGorenstein}, \cite[Proposition 6.2]{Hara.GeometricInterpretationOfTightClosureAndTestIdeals}, and \cite[Theorem 5.24]{Yamaguchi.CharacterizationOfMultiplierIdealsViaUltra} (\cf \autoref{appendix}).
    \begin{theoremC*}[{\autoref{cor.ClosuresCoincideInChar0ForParameterIdeals}, \autoref{thm.AllOurClosuresAgreeWithTightClosureForParameter}}]
        Suppose $R$ is an excellent domain {of equal characteristic zero}, with a dualizing complex, and $J = (f_1, \dots, f_t) \subseteq R$ is an ideal such that $f_1, \dots, f_t$ is part of a system of parameters in every localization $R_Q$ where $J \subseteq Q \in \Spec R$.  Then 
        \[
            J^{\KH} = J^{\alt} = J^{\cl_{\Gamma(\omega_Y)}} := (J \Gamma(Y, \omega_Y)) : \Gamma(Y, \omega_Y)
        \]
        where $\pi : Y \to \Spec R$ is any regular alteration.  Furthermore, if $R$ is of finite type over a field of characteristic zero, this ideal agrees with the tight closure $(J_p)^*$ after reduction to any characteristic $p \gg 0$.  As an application, $J_p^*$ is constant on a Zariski open set of characteristic $p > 0$ models (i.e., for all $p \gg 0$).
    \end{theoremC*}
    As an immediate consequence, canonical alteration closure also satisfies strong colon capturing properties for parameter ideals.}

    In fact, in the appendix (\autoref{prop.PhantomHirCalt}) we partially generalize the above result to ideals $J$ such that $R/J = H_0(P_{\mydot})$ where $P_{\mydot}$ is a bounded complex of finite free modules, in homological degrees $\geq 0$, satisfying the standard conditions on rank and height\footnote{Roughly speaking, the existence of such a $P_{\mydot}$ should be thought of as requiring that $J$ has an analog of finite phantom projective dimension.  For example, this always holds for ideals $J^k$ where $J$ is a parameter ideal and where $P$ is an Eagon-Northcott or Buchsbaum-Eisenbud complex which would produce a free resolution of $R/J^k$ if $J$ was generated by a regular sequence.}.  Specifically we show that 
    \begin{equation}
        \label{eq.AppendixComparison}
        \ker\big( R \to H_0(P_{\mydot} \otimes^{\myL} \myR\Gamma(Y, \cO_Y) \big) = J^{\alt} = J^{\cl_{\Gamma(\omega_Y)}} := (J \Gamma(Y, \omega_Y)) : \Gamma(Y, \omega_Y).
    \end{equation}
    as well as even a generalization to modules.  Again, these also coincide with tight and plus closure after reduction to characteristic $p \gg 0$, see \autoref{cor.AppendixReductionModP}.
    

    We conclude by identifying the test ideals associated to both \KosClos and canonical alteration closure.

    \begin{theoremD*}[\autoref{thm:alttestidealisJ}, \autoref{prop.TestIdealForKosClosure}]
        The test ideal associated to the canonical alteration closure is the de Fernex-Hacon multiplier ideal, \cite{DeFernexHaconSingsOnNormal}.
        \[
            \tau_{\alt}(R) = \mJ(R).
        \]

        If $R$ is Cohen-Macaulay, the test ideal associated to the Koszul-Hironaka (KH) closure is
        \[
            \tau_{\KH}(R) = \Ann_R\big(\omega_R / \Gamma(Y, \omega_Y)\big)
        \] where  {$Y \to \Spec R$} is a resolution of singularities. Recall that  $\Gamma(Y, \omega_Y) = \mJ(\omega_R)$ is the multiplier module (aka, the Grauert-Riemenschneider sheaf).  In particular, $\tau_{\KH}(R)$ agrees with the multiplier ideal if $R$ is Gorenstein.
    \end{theoremD*}
    \noindent
    It easily follows that if $R$ is Cohen-Macaulay, then $\tau_{\Kos}(R)$ reduces modulo $p$ to the parameter test ideal, see \autoref{cor.KHTestIdealReducesToParameterTestIdeal}.

    \subsection{A resolution-based preclosure operation}

    Finally, we also consider one other operation which we call \emph{Hironaka preclosure}, which sits in between \KosClos and canonical alteration closure, see \autoref{subsec.HironakaPreclosure}.  Hironaka preclosure is not idempotent and hence it is not a closure operation, see \autoref{exam.HironakaPreclosureNotIdempotent}.   However, one can fix this simply by iterating the operation and noting it must stabilize since our ring is Noetherian. We call this iterated version \emph{$\text{Hironaka}^\infty$ closure}.  Regardless, 
    both Hironaka preclosure and $\text{Hironaka}^\infty$ closure agree with KH-closure and canonical alteration closure for parameter ideals and even for ideals $J$ such that $R/J = H_0(P_{\mydot})$ where $P_{\mydot}$ is as in \autoref{eq.AppendixComparison}, see also \autoref{cor.HirIdempotentForPPD}.  In particular, all the closure operations we study satisfy strong versions of colon capturing.  Interestingly, the Hironaka$^{\infty}$ closure operation is better behaved with respect to ideal powers and products, see \autoref{prop.HirClosureIdealProducts}.  Finally, we note that Hironaka preclosure provides the full version of the Brian\c{c}on-Skoda theorem thanks to \cite{MaMcDonaldRGSchwede.BSForPseudorRational} (indeed, the work on Hironaka preclosure in this paper led directly to that result).



\begin{table}[h]
    {\fbox{
\begin{tabular}{ c | c | c | c | c }
name & Koszul-Hironaka & Hironaka & $\text{Hironaka}^\infty$& Canonical alteration\\[3pt]
\hline \hline
\rule{0pt}{2.5ex} notation & $I^{\KH}$ & $I^{\Hir}$ or  $L^{\Hir}_M$ & $I^{\Hir^{\infty}}$ or $L^{\Hir^{\infty}}_M$  & $I^{\alt}$ or $L^{\alt}_M$ \\[2pt] \hline 
idempotent & Yes & No & Yes & Yes \\ \hline
colon-capturing & Yes & Yes & Yes & Yes\\ \hline
Brian\c{c}on-Skoda & Partial & Yes & Yes & Yes \\ \hline
Macaulay2 & Yes & Yes & Yes & No \\\hline
detects & Rational sing. & Rational sing. & Rational sing. & KLT sing.\\ \hline
localization & Commutes & Commutes & Commutes & ? \\ \hline
ideal products & Poorly behaved & Well behaved & Well behaved & Well behaved \\ \hline
definition & \autoref{def.KHClosure} & \autoref{def.Hironaka} & \autoref{def.HironakaIterated} & \autoref{def.CanonicalAlterationClosure}
\end{tabular}
}}
    \caption{Closure operation properties}
    \label{tab:dual}
\end{table}

\subsection*{Acknowledgements}  The authors thank Holger Brenner, Ben Briggs, Hanlin Cai, Daniel Erman, Nobuo Hara, Srikanth Iyengar, Haydee Lindo, Linquan Ma, Kyle Maddox, Daniel McCormick, Shunsuke Takagi, and Mark Walker for valuable conversations.  We also thank Julian Benali, Rankeya Datta, Anne Fayolle, Srikanth Iyengar, Kyle Maddox, and Sandra Rodr\'iguez Villalobos for useful comments on a previous draft.

This material is partly based upon work supported by the National Science Foundation under Grant No. DMS-1928930 and by the Alfred P. Sloan Foundation under grant G-2021-16778, while some of the authors were visiting the Simons Laufer Mathematical Sciences Institute (formerly MSRI) in Berkeley, California, during the Spring 2024 semester.    McDonald was partially supported by NSF RTG grant DMS-1840190.
Schwede was partially supported by NSF Grant \#2101800, by NSF FRG Grant \#1952522 and NSF Grant \#2501903.  He also received support from the Simons Foundation SFI-MPS-TSM-00013051.

\subsection*{AI Statement}  ChatGPT 5.5 was used for grammar and proofreading in the main body of the text.  ChatGPT 6.0 Astra and Claude Opus 5.5 were used similarly on the Appendix.  ChatGPT 5.6 Sol suggested the example found in \autoref{exam.HironakaPreclosureNotIdempotent} showing that Hironaka preclosure is not idempotent (our verification of the example is via Macaulay2).

A substantial part of the proof of the main result of the Appendix was identified by ChatGPT 5.6 Sol.  What was prompted in what context is detailed in \autoref{subsec.app.AIDisclosure}.  

All writing in the document was done by the human authors.


\section{Background}

We begin with a quick review of the formalities of closure operations.
\subsection{Closure operations}
\label{subsec.ClosureOperations}
Let $R$ be a ring. A \emph{closure operation} (resp. a \emph{preclosure operation}) $\cl$ on $R$ is an operation on any pair of $R$-modules $L\subseteq M$ that returns an $R$-module $L^\cl_M\subseteq M$ that satisfies the following properties (resp. the first two of the following properties):
\begin{enumerate}
    \item[(1)] \emph{Extension}: $L\subseteq L^\cl_M$
    \item[(2)] \emph{Order-Preservation}: $L\subseteq L'\subseteq M$ implies $L^\cl_M\subseteq (L')^\cl_M$
    \item[(3)] \emph{Idempotence}: $(L^\cl_M)^\cl_M=L^\cl_M$
\end{enumerate}

Sometimes closure operations only apply to certain modules.  For instance, only submodules of $R$ (that is, ideals), in which case we say it is a \emph{closure operation on ideals}.  For ease of notation, when $M = R$ and $L = I$ is an ideal, we write 
\[
    I^{\cl} := I^{\cl}_M
\]
as the ambient module is clear.

There are several other properties of closure operations that are desirable:
\begin{enumerate}
    \item[(4)] \emph{Functoriality}: If $f\colon M\to N$ is a homomorphism, then $f(L^\cl_M)\subseteq f(L)^\cl_N$
    \item[(5a)] \emph{Semi-Residuality}: If $L^\cl_M=L$, then $0^\cl_{M/L}=0$
    \item[(5b)] \emph{Residuality}: $L^\cl_M=q^{-1}\left(0^\cl_{M/L}\right)$ where $q\colon M\to M/L$ is the quotient map
    \item[(6)] \emph{Faithfulness}: If $R$ is local, the maximal ideal is closed in $R$.    
    \item[(7)]  \emph{Persistence}:  If $R \to S$ is a map of rings for which the closure is defined, then the image of $S \otimes_R L^{\cl}_M$ in $S \otimes_R M$ is contained in $\Image(S \otimes L \to S \otimes M)^{\cl}_{S \otimes M}$.  For ideals $I \subseteq R$, this simply says that $I^{\cl}S \subseteq (IS)^{\cl}$.
\end{enumerate}
Perhaps the most important properties that a closure operation can satisfy, though, are those related to colon-capturing. In the following, let $(R,\fm)$ be a local ring and let $x_1,\dots,x_d$ be a system of parameters for $R$
\begin{enumerate}
    \item[(8a)] \emph{Colon-capturing}: $(x_1,\dots,x_k):x_{k+1}\subseteq(x_1,\dots,x_k)^\cl$.
    \item[(8b)] \emph{Strong colon-capturing, version A}: $(x_1^t,\dots,x_k):x_1^a\subseteq(x_1^{t-a},\dots,x_k)^\cl$.  We say it is \emph{improved} if $(x_1^t,\dots,x_k)^{\cl}:x_1^a\subseteq(x_1^{t-a},\dots,x_k)^\cl$.
    \item[(8c)] \emph{Strong colon-capturing, version B}: $(x_1,\dots,x_k)^\cl:x_{k+1}\subseteq(x_1,\dots,x_k)^\cl$.
    \item[(8d)] \emph{Generalized colon-capturing}: Suppose $R$ is a complete domain and $M$ is an $R$-module with $f\colon M\to R/(x_1,\dots,x_k)$ a surjective map such that $f(v)=x_{k+1}$. Then
    \[(Rv)^\cl_M\cap\ker(f)\subseteq((x_1,\dots,x_k)v)^\cl_M.\]    
\end{enumerate}
{%
Notice that strong colon capturing version $B$ can also be written as 
\[
    (x_1,\dots,x_k)^\cl:x_{k+1} = (x_1,\dots,x_k)^\cl
\]
since the containment $(x_1,\dots,x_k)^\cl \subseteq (x_1,\dots,x_k)^\cl:x_{k+1}$ always holds.
}

In \cite{dietz}, Dietz showed that a closure operation on $R$ satisfying (1-4), (5a), (6) and (8d) is equivalent to the existence of a big Cohen-Macaulay module over $R$. This was done by studying the properties of module closures.

\begin{defn}[Module closures]
\label{def.ModuleClosure}
    Let $B$ be an $R$-module. Then we define $\cl_B$ to be the closure operation
    \begin{align*}
        L^{\cl_B}_M:&=\big\{m\in M~\big|~m\otimes b\in\img(L\otimes_RB\to M\otimes_RB) ~\forall ~b\in B\big\}\\
        &=\bigcap_{b\in B}\big\{\ker (M\to M/L\otimes_R B)\big\}
    \end{align*}
    where the map on the second line is the composition of the map $M\to\ M\otimes_RB$ sending $m\mapsto m\otimes b$ with the natural map $M\otimes_RB\to M/L\otimes_RB$.
%
\end{defn}
{%
Furthermore, for any $R$-module $B$, if $J \subseteq R$ is an ideal, then 
\begin{equation}
\label{eq.ModuleClosureForIdeals}
    J^{\cl_B} = JB :_R B.
\end{equation}
Indeed, $x \in \bigcap_{b \in B} \ker(R \xrightarrow{1 \mapsto b} R/J \otimes B \cong B/JB)$ if and only if $bx \in JB$ for all $b \in B$.  But that is exactly the same as $x \in JB :_R B$.
}

In the case where $B$ is a big Cohen-Macaulay module, $\cl_B$ satisfies (1-4), (5b), (6) and (8(a-d)), see \cite{dietz}.

An important invariant of a closure operation is the associated test ideal.

\begin{defn}
    \label{def:testideal}
    Let $R$ be a ring and $\cl$ be a closure operation on a class of $R$-modules. Then the test ideal of $R$ is defined as
    \[\tau_\cl(R):=\bigcap_{L\subseteq M}(L:L^\cl_M)\]
    where the intersection ranges over all $R$-module pairs $L\subseteq M$ to which the closure operation applies. We can also define the finitistic test ideal as 
    \[\tau^{\mathrm{fg}}_\cl(R):=\bigcap_{L\subseteq M}(L:L^\cl_M)\]
    where the intersection ranges over all  $R$-module pairs $L\subseteq M$ such that $M/L$ is finitely generated.
\end{defn}
Test ideals have alternate characterizations in certain settings.  First we recall the notion of a trace ideal as well as a certain generalization to complexes.
\begin{defn}
    Suppose $R$ is a ring and $M$ is an $R$-module.  Then  the \emph{trace ideal of $M$ in $R$} is $\tr_M(R) := \sum_{\phi} \phi(M)$ where $\phi$ runs over elements of $\Hom_R(M, R)$.
\end{defn}

See \autoref{defn.TraceOfCopmlex} below for a variant of trace for a complex.

\begin{thm}[\cite{perezrg}~Theorem~1.1]
    \label{thm:testistrace}
    Let $(R,\fm,k)$ be a Noetherian local ring and $E:=E_R(k)$ the injective hull of the residue field. 
    \begin{enumerate}
        \item Let $\cl$ be a residual closure operation. Then
        \[\tau_\cl(R)=\ann0^\cl_E.\]
        \item Let $\cl=\cl_B$ be a module closure. If $R$ is complete or $B$ is finitely-presented, then
        \[\tau_\cl(R)=\sum_{f\in\Hom_R(B,R)} f(B)=\tr_B(R).\]
    \end{enumerate}
\end{thm}

\subsection{Regular alterations and singularities}

In order to define another one of our closure operations, we will need some definitions from algebraic geometry.  A much more complete reference to most of what we discuss is \cite{KollarKovacsSingularitiesBook}.


\begin{defn}
Let $X$ be a reduced Noetherian scheme. For us, a \emph{regular alteration} of $X$ is a proper, surjective, generically finite map $\pi\colon Y\to X$ from a scheme $Y$ such that $Y$ is regular and such that every irreducible component of $Y$ dominates an irreducible component of $X$.  A \emph{resolution of singularities} is a birational regular alteration (in particular, there is a bijection between irreducible components of $X$ and $Y$ in this case).  If $\Delta$ is a $\bQ$-divisor on $X$, we say a resolution is a \emph{log resolution of $(X, \Delta)$} if the union of the exceptional set of $\pi$, with $\pi_*^{-1} \Delta$, (the strict transform of $\Delta$) is a simple normal crossings divisor\footnote{This means that this support is a union of smooth prime divisors that locally analytically look like coordinate hyperplanes}.
\end{defn}

(Log) Resolutions of singularities, and hence regular alterations, exist for varieties thanks to \cite{HironakaResolution}.  This was generalized to excellent schemes of {equal} characteristic zero in \cite{Tempkin.DesingularizationQuasiExcellentChar0}.  Regular alterations exist for positive characteristic varieties (and some excellent schemes even in mixed characteristic) thanks to \cite{deJongAlterations, deJongAlterations2}.

\begin{notation}
    We will frequently consider $\pi : Y \to \Spec R$ a resolution of singularities or a (regular) alteration.  In that case we write $\Gamma(\cO_Y)$ (or $\Gamma(\omega_Y)$) instead of $\Gamma(Y, \cO_Y)$ (respectively $\Gamma(Y, \omega_Y)$) as no confusion seems likely.
\end{notation}

\begin{definition}
    Suppose $X$ is an excellent normal scheme in {equal} characteristic zero with a dualizing complex.  We say that $X$ has \emph{rational singularities} if for one, or equivalently any, resolution of singularities $\pi : Y \to X$ we have that $\cO_X \to \myR \pi_* \cO_Y$ is an isomorphism in $D(X)$.  Equivalently, if $X = \Spec R$, we say that $R$ has \emph{rational singularities} if $X$ does.  This is equivalent to requiring that $R \to \myR \Gamma(\cO_Y)$ is an isomorphism in $D(R)$.
\end{definition}

Because of the above, given a proper morphism of schemes $\pi\colon Y\to X$, we are interested in the derived pushforward of the structure sheaf $\myR\pi_*\cO_Y$. Because $\pi$ is proper, this is an object of $D^b(\coh X)$ and can be computed by taking an injective resolution of $\cO_Y$ and applying $\pi_*$. As an object, this has been long known to be key to understanding rational singularities in {equal} characteristic zero.  For varieties, Kov\'acs showed in \cite{KovacsRat} that if $Y$ has rational singularities and $\cO_X\to \myR\pi_*\cO_Y$ splits then $X$ has rational singularities, and work of Kov\'acs \cite{KovacsRat} and Bhatt \cite{BhattDerivedDirectSummandCharP} showed that $X$ has rational singularities if and only if $\cO_X\to \myR\pi_*\cO_Y$ splits for all $\pi\colon Y\to X$ proper surjective (i.e., $X$ is a derived splinter).  These results generalize to excellent schemes with dualizing complexes by \cite{Murayama.VanishingForQSchemes}.

The other class of singularities we will be interested in is \emph{KLT singularities}, an {equal} characteristic zero analog of strongly $F$-regular singularities from positive characteristic. 

\begin{definition}
    Suppose $X$ is a normal excellent integral scheme in {equal} characteristic zero with a dualizing complex.  We say that $X$ has \emph{Kawamata log terminal type} (or \emph{KLT-type}) singularities if there exists a $\bQ$-divisor $\Delta \geq 0$ with $K_X + \Delta$ $\bQ$-Cartier, such that for $\pi : Y \to X$ a log resolution of $(X, \Delta)$, we have that $\pi_* \cO_Y(\lceil K_Y - \pi^* (K_X + \Delta) \rceil) = \cO_X$.  Equivalently, if $X = \Spec R$, this means that $\myR \Gamma(\cO_Y(\lceil K_Y - \pi^* (K_X + \Delta) \rceil) ) \to R$ is an isomorphism where the higher direct images vanish by relative Kawamata-Viehweg vanishing.

    More generally, the \emph{(de Fernex-Hacon) multiplier ideal sheaf of $X$} is $\mJ(X) := \sum_{\Delta} \pi_* \cO_Y(\lceil K_Y - \pi^* (K_X + \Delta) \rceil) \subseteq \cO_X$ where $\Delta$ is as above. In fact, that sum has a maximal element \cite{DeFernexHaconSingsOnNormal}.  Likewise if $X = \Spec R$, we define the \emph{(de Fernex-Hacon) multiplier ideal} $\mJ(R)$ to be $\Gamma(X,\mJ(X)) = \sum_{\Delta} \Gamma(\cO_Y(\lceil K_Y - \pi^* (K_X + \Delta) \rceil))$.
\end{definition}

Similar to how strongly $F$-regular singularities are measured using the test ideal in positive characteristic, we immediately see that $R$ has KLT singularities if and only if the \emph{multiplier ideal} equals $R$.  We will need the following characterization of the multiplier ideal given by the second author.  We note that KLT singularities are rational thanks to \cite{ElkikRationalityOfCanonicalSings,KovacsRat,Murayama.VanishingForQSchemes}.

It is worth remarking that the higher direct images vanish by the relative Kawamata-Viehweg vanishing theorem \cite{KawamataVanishing,ViehwegVanishingTheorems} as generalized by Murayama \cite{Murayama.VanishingForQSchemes}.  That is, for $i > 0$, 
\[
    0 = \myR^i \pi_* \cO_Y(\lceil K_Y - \pi^* (K_X + \Delta) \rceil)
\]
A special case of this is $\myR^i \pi_* \omega_Y = 0$ for $i > 0$ -- the Grauert-Riemenschneider vanishing theorem \cite{GRVanishing} (again, generalized to the non-variety case in \cite{Murayama.VanishingForQSchemes}). A recent result of the second author allows us to understand the multiplier ideal as a sum of trace ideals.

\begin{theorem}[\cite{Mcd-multklt}]
    \label{thm.Mcd-multklt}
Let $R$ be a normal, excellent domain containing $\Q$ with a dualizing complex. Then the \emph{multiplier ideal} of $R$ is the ideal:
\[\mJ(R):=\sum_{\pitoR} \Image(\Hom(\canon,R) \otimes_R \canon \to R),\]
where $\pitoR$ ranges over all regular alterations and the map $\Hom(\canon,R) \otimes_R \canon \to R$ is the evaluation map. 
\end{theorem}

This leads to a notion of a multiplier ideal of a module.

\begin{defn}
\label{def:multsubmod}
 Let $M$ be a \fg\ $R$-module, where $R$ is as above. We define the \emph{multiplier submodule of $M$} to be
 \begin{align*}
 	\mJ(M)&:=\sum_{\pitoR} \Image(\Hom(\canon,M) \otimes_R \canon \to M)\\
   	& =\sum_{\pitoR} \tr_{\canon}(M).
\end{align*}
For $\pitoR$ a particular regular alteration, we define
    \[\mJ_\pi(M):=\Image(\Hom(\canon,M) \otimes_R \canon \to M)= \tr_{\canon}(M).\]
\end{defn}

Finally, we recall a criterion for rational singularities and a multiplier-ideal-like object.
\begin{definition}
    Suppose $R$ is an excellent reduced locally equidimensional ring of {equal} characteristic zero with a dualizing complex and $\pi : Y \to \Spec R$ is a resolution of singularities.  Then the submodule $\Gamma(\omega_Y) \subseteq \omega_R$ is called the \emph{multiplier module} or \emph{Grauert-Riemenschneider sheaf}, denoted $\mJ(\omega_R)$.  It is a straightforward application of Grothendieck duality and the vanishing theorems mentioned above that $R$ has rational singularities if and only if the following two conditions hold.
    \begin{enumerate}
        \item $R$ is Cohen-Macaulay.
        \item $\Gamma(\omega_Y) =: \mJ(\omega_R) = \omega_R$.
    \end{enumerate}
    This criterion is sometimes called Kempf's criterion for rational singularities \cite{KempfToroidalEmbeddings}.  See \cite{Murayama.VanishingForQSchemes} for generalizations. 
\end{definition}

The following theorem will be useful to us.

\begin{proposition}[{\cite[Theorem 5.2]{Hara.GeometricInterpretationOfTightClosureAndTestIdeals}, see also \cite{SmithMultiplierTestIdeals} and \cite[Chapter 6, Theorem 3.7]{SchwedeSmith.FBook}}]
    \label{prop.JOmegaReducesToTauOmega}
    Suppose $R$ is a reduced locally equidimensional ring essentially of finite type over a field $k$ of characteristic zero.  Then $\mJ(\omega_R) \subseteq \omega_R$ reduces modulo $p \gg 0$ to $\tau(\omega_{R_p}) \subseteq \omega_{R_p}$.  Here $\tau(\omega_{R_p})$ is the parameter test module \cite{SmithTestIdealsInLocalRings,Hara.GeometricInterpretationOfTightClosureAndTestIdeals,SchwedeTuckerTestIdealsOfNonprincipalIdeals}.  
\end{proposition}
Recall that locally, $\tau(\omega_{R_p})$ is Matlis dual to $H^d_{\fm}(R_p)/0^*_{H^d_{\fm}(R_p)}$ where $d = \dim R_{\fm}$ for $\fm \subseteq R$ maximal.


\subsection{The Koszul complex}

Let $x$ be an element of $R$ and denote by $K_{\mydot}(x;R)$ the Koszul complex on $x$, that is the mapping cone of the map $R\to R$ given by multiplication by $x$. More generally, given a sequence $\mathbf{x}=(x_1,\dots,x_n)$ we denote by $K_{\mydot}(\mathbf{x};R)=K_{\mydot}(x_1,\dots,x_n;R)\colon= K_{\mydot}(x_1;R)\otimes_R\cdots\otimes_R K_{\mydot}(x_n;R)$. Sometimes it will be helpful to work dually, so we denote by $K^\mydot(\mathbf{x};R)$ the dual Koszul complex on $\mathbf{x}$ and note that
\[K^{\mydot}(\mathbf{x};R)\colon=\Hom_R(K_{\mydot}(\mathbf{x};R),R)\cong\Sigma^{-n}K_{\mydot}(\mathbf{x};R).\]
Given an $R$-complex $X$, we denote the Koszul complex on $\mathbf{x}$ with coefficients in $X$ by $K_{\mydot}(\mathbf{x};X)\colon=K_{\mydot}(\mathbf{x};R)\otimes_RX$ and do the same for the dual Koszul complex. We will use the following notation for the (co)homology.
\begin{align*}
    H_*(\mathbf{x};X)&\colon=H_*(K_{\mydot}(\mathbf{x};X))\\
    H^*(\mathbf{x};X)&\colon=H^*(K^\mydot(\mathbf{x};X))
\end{align*}
Note that $H_i(\mathbf{x};X)\cong H^{n-i}(\mathbf{x};X)$.  In particular, if $M$ is an $R$-module, then $H_0({\bf x}; M) \cong M/{\bf x}M \cong H^{n}({\bf x}; M)$.

Note that as a chain complex $K_{\mydot}({\bf x};R)$ lives in homological degrees $0$ to $n$, while as a co-chain complex it lives in cohomological degrees $-n$ to $0$.

We will be most interested in the depth-sensitivity of the Koszul complex.

\begin{defn}
    Let $I=(f_1,\dots,f_n)$ be an ideal of $R$ and $X$ an $R$-complex. The $I$-depth of $X$ is 
    \[\depth_R(I,X):=n-\sup\{i~|~H_i(f_1,\dots,f_n;X)\neq0\}.\]
    When $R$ is local with maximal ideal $\fm$, we will often write $\depth_R(X)$ for the $\fm$-depth of $X$.
\end{defn}

We are particularly interested in complexes of maximal depth and maximal Cohen-Macaulay complexes. For an in-depth survey of this circle of ideas, see \cite{IMSW:2021}.
{
    
\begin{defn}[\cf \cite{Roberts.CMComplexesAndAnalyticProofOfNewIntersection,BhattAbsoluteIntegralClosure,IMSW:2021}]
    Let $(R,\fm,k)$ be a local Noetherian ring with maximal ideal $\fm$ and residue field $k$. We say that an $R$-complex $X$ has \emph{maximal depth} if
    \begin{enumerate}
        \item $H(X)$ is bounded;
        \item $H_0(X)\to H_0(k\otimes_R^\bL X)$ is nonzero; and
        \item $\depth_R(X)=\dim(R)$, (equivalently $H^i_\fm(X)=0$ for $i< d= \dim(R)$ and $H^d_{\m}(X) \neq 0$, see \cite[(2.1.1)]{IMSW:2021}).
    \end{enumerate}
    If additionally, $H^i_\fm(X)=0$ for $i\neq\dim(R)$ we say $X$ is \emph{big Cohen-Macaulay}.

    Suppose now that $R$ is not necessarily local.  We say that a bounded complex $X$ is \emph{locally Cohen-Macaulay} if $H(X)$ is finitely generated, and if for each prime ideal $Q \in \Spec R$, we have that $X_Q$ is big Cohen-Macaulay over $R_Q$.
\end{defn}

Over a local ring $R$, $X$ can be big Cohen-Macaulay but not 
locally Cohen-Macaulay even if $X$ has finitely generated cohomology.  The point is that the support of $H(X)$ need not agree with $\Spec R$.  This issue even appears for modules.  For instance, if $R = k\llbracket x,y\rrbracket/(xy)$ and $M = R/(x)$,  then $M$ is a big Cohen-Macaulay $R$-module, but it is not locally Cohen-Macaulay in our sense (since it is zero after localizing at $Q = (y)$).

If $R$ is an excellent { locally equidimensional} ring containing $\Q$ with a dualizing complex, an important example of a locally Cohen-Macaulay complex comes from a resolution of singularities or regular alteration.

\begin{lemma}[{\cite{Roberts.CMComplexesAndAnalyticProofOfNewIntersection,Roberts.HomologicalInvariantsOfModulesOverCommutative,GRVanishing,Murayama.VanishingForQSchemes}}]  
    \label{lem.CMComplexFromResolution}
    Suppose $R$ is a reduced excellent locally equidimensional ring over $\bQ$ containing a dualizing complex, and suppose that $\pi : Y \to \Spec R$ is a regular alteration.  Then $\myR \Gamma(\cO_Y)$ is a locally Cohen-Macaulay $R$-complex.
\end{lemma}

This is well known to experts but the proof is short, and usually asserted only in the domain case, so we sketch the proof.

\begin{proof}
    We may assume that $(R,\fm, k)$ is local of dimension $d$.  Then $\myR\Gamma(\cO_Y)$ is a bounded complex with finitely generated cohomology since $\pi$ is proper.  Then the vanishing statement follows from \cite{Murayama.VanishingForQSchemes, GRVanishing} and Matlis, local and Grothendieck duality (we crucially use equidimensionality here).  For the non-triviality statement, this follows as the fiber $Y_k$ over the closed point $\m \in \Spec R$ is nonempty.
\end{proof}
Suppose additionally that $R$ is local but not equidimensional, and $I \subseteq R$ is the intersection of minimal primes $Q$ such that $\dim R/Q = \dim R$.  Set $R' = R/I$.  If $Y \to \Spec R'$ is a regular alteration, then $\myR\Gamma(\cO_Y)$ is an $R$-complex of maximal depth, and even a big Cohen-Macaulay complex.  However, it is not locally Cohen-Macaulay in our sense as it becomes zero after localizing at minimal primes $P$ with $\dim R/P < \dim R$.  Some of the results we prove in this paper could be generalized outside of the equidimensional case by using such a partial resolution of singularities (for instance, colon capturing), however, then we would lose the property that the closures commute with localization (see e.g. Proposition \ref{prop.FlatMapsWithRationalSingularitiesFibers}).  
}

\begin{defn}
Let $I\subseteq R$ be an ideal and $X$ an $R$-complex. The $I$-adic completion of $X$, denoted $\Lambda^I X$, is defined as
\[\Lambda^IX:=\lim_{n\geq1}X/I^nX=\lim\left(\cdots\to X/I^3X\to X/I^2X\to X/IX\right).\]
The derived $I$-adic completion of $X$, denoted $L\Lambda^I(X)$, is computed by taking $P\to X$ a projective resolution and defining
\[L\Lambda^I(X):=\Lambda^IP.\]
This complex is well-defined in $D(R)$ and there is a natural map $X\to\Lambda^I X$. We say $X$ is derived $I$-complete if this map is a quasi-isomorphism.
\end{defn}

If $(R,\fm)$ is a local ring, the Koszul complex of a locally Cohen-Macaulay complex (or, more generally, of a derived $\fm$-complete complex of maximal depth) has nice vanishing properties.

\begin{lemma}[\cf \cite{IMSW:2021} Lemma~2.6]\label{depth}
Let $(R,\fm)$ be a local ring with $x_1,\dots,x_d$ a system of parameters and $X$ an $R$-complex of maximal depth. If $X$ is derived $\fm$-complete or if $H(X)$ is finitely generated, then
\[H_i(x_1,\dots,x_k;X)=0\]
for all $i\geq1$ and all $1\leq k\leq d$. Equivalently,
\[H^i(x_1,\dots,x_k;X)=0\]
for all $i\leq k$ and all $1\leq k\leq d$.
\end{lemma}

\begin{rem}\label{fgdepth}
    Note that the original lemma in \cite{IMSW:2021} only considered $X$ derived $\fm$-complete. The requirement that $X$ be derived $\fm$-complete allows for the possibility that $X$ does not have finitely generated homology, which is of particular relevance to the study of big Cohen-Macaulay modules. Key to the proof of \cite[Lemma~2.6]{IMSW:2021} is \cite[Lemma~1.3]{IMSW:2021}, which says that the homology of derived $\fm$-complete complexes satisfy Nakayama's Lemma; that is to say, if $X$ is derived $\fm$-complete, then $\fm H_i(X)=H_i(X)$ implies $H_i(X)=0$. Thus, the proof of the original lemma applies to complexes with finitely generated homology as well.  In this paper, we will essentially only consider the case when $X$ has finitely generated homology.
\end{rem}

\subsection{Local cohomology via the Koszul complex}
The discussion below will be relevant to several results in later sections. The basic ideas are contained in \cite{Dwyer-Greenlees:2002}. First, let $f\in R$ and consider the following commutative diagram
\[\begin{tikzcd}
    R\arrow[r]\arrow[d,swap,"f^k"]&R\arrow[d,"f^{k+1}"]\\
    R\arrow[r,"f"]&R
\end{tikzcd}
\]
which induces a map $K^{\mydot}(f^k)\to K^{\mydot}(f^{k+1})$. Set $K^{\mydot}(f^\infty):=\colim_k K^{\mydot}(f^k)$. Given a collection of elements of $R$, $\mathbf{f}=f_1,\dots,f_n$, set $\mathbf{f}^k=(f_1^k,\dots,f_n^k)$. We can tensor these natural maps $K^{\mydot}(f_i^k)\to K^{\mydot}(f_i^{k+1})$ to get a natural map $K^{\mydot}(\mathbf{f}^k)\to K^{\mydot}(\mathbf{f}^{k+1})$. Set $K^{\mydot}(\mathbf{f}^\infty):=\colim_k K^{\mydot}(\mathbf{f}^k)$ and note that
\[K^{\mydot}(\mathbf{f}^\infty)=K^{\mydot}(f_1^\infty)\otimes_R\cdots\otimes_R K^{\mydot}(f_n^\infty).\]
By \cite[Lemma~6.9]{Dwyer-Greenlees:2002} this complex is isomorphic to a free $R$-complex concentrated in cohomological degrees $0$ to $n$. It also computes the local cohomology of an $R$-complex $M$ via the following formula:
\[H^i_{(\mathbf{f})}(M)=H^i(\mathbf{f}^\infty;M)=H^i(K^{\mydot}(\mathbf{f}^\infty)\otimes_RM).\]

\begin{rem}
Note that the directed system computing $K^{\mydot}(\mathbf{f^{\infty}})$ can be refined to include all maps $K^{\mydot}(f_1^{a_1},\dots,f_n^{a_n})\to K^{\mydot}(f_1^{a_1},\dots,f_i^{a_i+1},\dots,f_n^{a_n})$.
\end{rem}

The following result connecting depth and local cohomology will also be useful.

\begin{lemma}
    \label{lem.VanishingLocalCohomologyForIdealsContainingARegSeq}
    Suppose $(R, \m)$ is a Noetherian local ring and $X \in D^{\geq 0}(R)$ is a bounded complex with $H(X)$ finitely generated.  Suppose $I \subseteq \m$ contains elements $x_1, \dots, x_r$ and $X$ has $(x_1, \dots, x_r)$-depth $r$.  Then $H^i_I(X) = 0$ for $i  < r$.
\end{lemma}
We will be primarily interested in the case that $x_1, \dots, x_r$ is a (partial) system of parameters and $X$ has maximal depth, in which case it has $(x_1, \dots, x_r)$-depth $r$ by \cite[Lemma 2.6]{IMSW:2021}.
\begin{proof}
    We proceed by induction on $r$, the base case $r = 0$ holding vacuously.  We have the following long exact sequence
    \[ 
        \dots \to H^{i-1}_{I}(K_{\mydot}(x_1; X)) \to H^i_{I}(X) \xrightarrow{\cdot x_1} H^i_{I}(X) \to \dots
    \]
    By the depth hypothesis, we see that $K_{\mydot}(x_1; X)$ lives in $D^{\geq 0}$ and in fact, $K_{\mydot}(x_1; X)$ has $(x_2, \dots, x_r)$-depth $r-1$.  Suppose $i < r$.  Then by induction $H^{i-1}_{I}(K_{\mydot}(x_1; X)) = 0$ and so the multiplication by $x_1$-map is injective.  But $H^i_{I}(X)$ is made up of $I$-power-torsion elements, and so it must vanish as claimed.
\end{proof}

\subsection{The trace ideal associated to a complex}

We will need the following notion of a trace ideal associated to a complex.  In our case, the complex will typically be $\myR\Gamma(\cO_Y)$ for an alteration $Y \to \Spec R$.  See also \cite[Subsection 2.4]{Lyu.PropertiesBirationalDerivedSplinters}.

\begin{defn}
    \label{defn.TraceOfCopmlex}
    Suppose we have a cochain complex $M \in D(R)$, then we define the \emph{\traceIdealComplex} of $M$ in $R$ to be $\trcomp{M}(R) := \sum_{\phi} \Image(H^0(M) \xrightarrow{H^0 \phi} R)$ where $\phi$ runs over $\Hom_{D(R)}(M, R)$. 
\end{defn}
\begin{lemma}
\label{lem.TraceOfComplexDGAVersion}
    With notation as above, suppose $M$ is also a dg $R$-algebra.  Then
    \[
        \trcomp{M}(R) = \{f \in R \;|\; \text{ multiplication by $f$ can be factored as $\times f : R \to M \xrightarrow{\rho} R$ in $D(R)$} \}.
    \]
    Furthermore, we also have that $\trcomp{M}(R)$ is identified with the image 
    \[
        \Image\big(\Hom_{D(R)}(M, R) \to \Hom_{D(R)}(R, R) \cong R\big).
    \]
    where the map is induced by applying $\Hom_{D(R)}(-, R)$ to $R \to M$.
\end{lemma}

\begin{proof}
    Suppose $f$ is in the right side.  Then $f$ is in the image of $H^0 {\rho}$, $f = (H^0 \rho)(g)$ for some $g \in H^0(M)$.  
    The containment $\supseteq$ follows.
    {%
    Conversely, if $f \in \trcomp{M}$, then we can write $f = \phi_1(g_1) + \dots + \phi_n(g_n)$ for some $g_i \in H^0(M)$ and $\phi_i = H^0(\psi_i)$ for some $\psi_i : M \to R$ in $D(R)$. Let $G_i\in M^0$ represent the cohomology class for each $g_i$ and let $\psi_{G_i}$ be the multiplication by $G_i$ map on $M$. Note that $\psi_{G_i}$ is independent of choice of representative, for if $G_i'\in M^0$ also represents the cohomology class $g_i$, then $G_i-G_i'=d(h)$ for some $h\in M^{-1}$. Then multiplication by $h$ defines a nullhomotopy, so $\psi_{G_i}$ is homotopic to $\psi_{G_i'}$ and thus they define the same map in $D(R)$.
    Setting $\psi_i' = \psi_i \circ \psi_{G_i}$
    and $\phi_i' = H^0(\psi_i')$ we see that 
    \[
        f = \sum_i \phi_i'(1) = \big(\sum_i \phi_i'\big)(1).
    \]
    Setting $\psi' = \sum_i \psi_i'$ we see that the composition 
    \[
        R \to M \xrightarrow{\psi'} R
    \]
    is multiplication by $f$.
    }

    For the second statement, we notice that the image is exactly those $f \in R$ such that $\times f$ factors as $R \to M \xrightarrow{\rho} R$. 
\end{proof}

{%
\subsection{Reduction modulo $p \gg 0$}
\label{subsec.ReductionModP}
In various theorems below, we will compare our closure operations to characteristic $p > 0$ operations via reduction modulo $p$.  We will be somewhat informal in our notation around these arguments and give a quick summary below of what the precise setup entails.  For a more detailed explanation, we refer the reader to \cite{HochsterHunekeTightClosureInEqualCharactersticZero} (see also \cite[Chapter 6]{SchwedeSmith.FBook}).

Suppose $k$ is a field of characteristic zero, $R$ is a finite type $k$-algebra, $J \subseteq R$ is an ideal and $M$ is a finitely generated $R$-module.  One can fix $A \subseteq k$ a finitely generated $\bZ$-algebra, as well as an $A$-algebra $R_A$, an ideal $J_A$ and a module $M_A$ whose base change $- \otimes_A k$ recovers $R$, $J$, and $M$.  Indeed, this can even be done for finitely many ideals, modules, projective schemes over $R$, and various maps between the modules and between the rings schemes.

For any maximal ideal $\mathfrak{t} \in A$, we can form the base changes $R_{\mathfrak{t}} := R_A \otimes_A A/\mathfrak{t}$, $J_{\mathfrak{t}} := J_A \otimes_A A/\mathfrak{t}$, $M_{\mathfrak{t}} :=M_A \otimes_A A/\mathfrak{t}$, etc.  Most properties of $R$, $J$, $M$, or maps between them can be preserved under this process, at least for all $\mathfrak{t}$ in a dense open subset $U \subseteq \m\text{-}\Spec A$, see for instance \cite{HochsterHunekeTightClosureInEqualCharactersticZero}.  Notably, exactness of sequences of maps can be preserved.

When we talk about reduction modulo $p \gg 0$ of an ideal $J \subseteq R$,  we are implicitly fixing such an $A$, as well as a relevant $U \subseteq \m\text{-}\Spec A$ which preserve our desired properties (for instance, surjectivity or injectivity of a map between modules, or a containment of ideals).  Indeed, when we write $J_{p}$ for $p \gg 0$, we mean $J_{\mathfrak{t}}$ for $\mathfrak{t} \in U$.  
}
\section{The \KosClos operation in {equal} characteristic zero}

In this section, we define the \KosClos on ideals and prove basic properties about it.  We work in the following setting.

\begin{setting}
    \label{set.NiceSetup}
    Throughout this section, $R \supseteq \bQ$ is a reduced Noetherian excellent ring with a dualizing complex (in particular, $R$ has finite Krull dimension). 
\end{setting}

\begin{defn}
\label{def.KHClosure}
    We suppose $R$ is as in \autoref{set.NiceSetup}.
    Let $I=(\mathbf{f})$ be an ideal and let $\pi:Y\to\Spec R$ be a regular alteration. Then define the \emph{Koszul-Hironaka closure} (the \emph{\KosClos}) of $I$ in $R$ to be
        \[I^{\Kos}_R:=\ker\big(R\to H_0(\mathbf{f};\myR \Gamma(\cO_Y))\big)\]
    where the map in question comes from following either path around the following commutative square
    \[\begin{tikzcd}
        R\arrow[r]\arrow[d]&\myR \Gamma(\cO_Y)\arrow[d]\\
        K_{\mydot}(\mathbf{f};R)\arrow[r]&K_{\mydot}(\mathbf{f};\myR \Gamma(\cO_Y))
    \end{tikzcd}\]
    Note that when $R$ is local, $(x_1,\dots,x_d)$ is a system of parameters for $R${, and $R$ is equidimensional}, then
        \[H_i(x_1,\dots,x_k;\myR \Gamma(\cO_Y))=0\]
    for $i\geq1$ and all $k$ by  {\autoref{lem.CMComplexFromResolution} and} \autoref{depth}.
\end{defn}

Given that this is independent of choices (see the result immediately below), we see that if $R$ has rational singularities, then $R \cong \myR \Gamma(\cO_Y)$ if $Y$ is a resolution of singularities.  As an immediate consequence we see that $J^{\Kos} = J$ for \emph{every} ideal $J \subseteq R$.  For converse statements, see \autoref{lem.ParameterIdealObservationsFollowingHochsterHuneke} below.

\begin{prop}\label{prop.kosIndependence}
    With notation as in \autoref{set.NiceSetup}, the \KosClos of an ideal $I$ in $R$ is a well-defined closure operation. Furthermore, when $R$ is local, it is faithful.
\end{prop}
\begin{proof}
We need to show that \KosClos is well-defined, which requires showing that it is independent of both the regular alteration chosen as well as the choice of generating set. 

For the alteration, let $\pi_i\colon Y_i\to\Spec R$ for $i=1,2$ be regular alterations and let $\pi\colon Y\to \Spec R$ be a regular alteration dominating the two. Then we get that the natural maps $\myR\Gamma(\cO_{Y_i})\to \myR\Gamma(\cO_Y)$ split, as the $Y_i$ are smooth and thus have rational singularities and so are derived splinters. Then the following maps split for any $\mathbf{f}\in R$
\[H_0(\mathbf{f};\myR\Gamma(\cO_{Y_i}))\to H_0(\mathbf{f};\myR \Gamma(\cO_Y)).\]
This implies that
\[\ker(R\to H_0(\mathbf{f};\myR\Gamma(\cO_{Y_i})))\cong\ker(R\to H_0(\mathbf{f};\myR \Gamma(\cO_Y)))\]
and thus \KosClos is independent of the choice of regular alteration. 

To see that \KosClos is extensive, let $I=(\mathbf{f})$ and note that the map $R\to H_0(\mathbf{f};\myR \Gamma(\cO_Y))$ factors through $H_0(f_1,\dots,f_n;R)=R/I$ 
and so $I\subseteq I^{\Kos}$. Similarly, \KosClos is order preserving because if $I\subseteq J$, we can extend a generating set for $I$ to a generating set for $J$ and thus get that $I^{\Kos}\subseteq J^{\Kos}$.

We will now prove that \KosClos is independent of the generating set ${\bf f}$ of $I$.  Since two choices for a generating set can each be extended to a third which contains both of them, we just need to show that adding a single redundant generator leads to the same closure.  More generally, since clearly $I \subseteq I^{\Kos}$, it suffices to show that for any $g \in I^{\Kos}$, we have that $(I + (g))^{\Kos} = I^{\Kos}$ and this will prove that \KosClos is idempotent as well.

Hence pick $g \in I^{\Kos}$.  We want to show
\[
    \ker(R\to H_0(\mathbf{f};\myR \Gamma(\cO_Y)))=\ker(R\to H_0(\mathbf{f},g;\myR \Gamma(\cO_Y))).
\]
To see this, we note that the (dg) $R$-algebra map $R\to K_{\mydot}(\mathbf{f};\myR \Gamma(\cO_Y))$ induces a map of $R$-algebras $R\to H_0(\mathbf{f};\myR\Gamma(\cO_Y))$ and thus $H_0(\mathbf{f};\myR \Gamma(\cO_Y))$ is also an $R/I^{\Kos}$-{module}. Consider now the map on triangles
\[\begin{tikzcd}
    R\arrow[r,"\cdot g"]\arrow[d]&R\arrow[r]\arrow[d]&K_{\mydot}(g;R)\arrow[r]\arrow[d]&\hspace{1mm}\\
    K_{\mydot}(\mathbf{f};\myR \Gamma(\cO_Y))\arrow[r,"\cdot g"]&K_{\mydot}(\mathbf{f};\myR \Gamma(\cO_Y))\arrow[r]&K_{\mydot}(\mathbf{f},g;\myR \Gamma(\cO_Y))\arrow[r]&\hspace{1mm}
\end{tikzcd}\]
which induces the following map on long exact sequences in homology
\[\begin{tikzcd}
    R\arrow[r,"\cdot g"]\arrow[d]&R\arrow[r]\arrow[d]&R/(g)\arrow[d]\\
   H_0(\mathbf{f};\myR \Gamma(\cO_Y))\arrow[r,"\cdot g"]&H_0(\mathbf{f};\myR \Gamma(\cO_Y))\arrow[r]&H_0(\mathbf{f},g;\myR \Gamma(\cO_Y))
\end{tikzcd}\]
We claim the lower left horizontal map is actually the zero map.  The key observation is that $H_0(\mathbf{f};\myR \Gamma(\cO_Y))$ is an $R$-algebra because $K_{\mydot}(\mathbf{f};\myR \Gamma(\cO_Y))$ may be viewed as a differential graded $R$-algebra (as it is a tensor product of differential graded algebras).  Hence, since $g\in I^{\Kos}$, we see that $g$ annihilates $1 \in H_0(\mathbf{f};\myR \Gamma(\cO_Y))$ and so the lower left multiplication-by-g-map is zero as claimed. 
Then $H_0(\mathbf{f};\myR\Gamma(\cO_Y))\hookrightarrow H_0(\mathbf{f},g;\myR\Gamma(\cO_Y))$ which implies 
\[\ker\big(R\to H_0(\mathbf{f},g;\myR \Gamma(\cO_Y))\big)=\ker\big(R\to H_0(\mathbf{f};\myR \Gamma(\cO_Y))\big).\]
Thus, the operation \KosClos is well-defined and idempotent, and hence a closure operation.


Finally, if $R$ is local, \KosClos is faithful because $\myR \Gamma(\cO_Y)$ is locally Cohen-Macaulay and thus the natural map $R\to \myR \Gamma(\cO_Y)\to H_0(k\otimes_R^\bL \myR \Gamma(\cO_Y))$ is nonzero.  As this map factors through $\myR \Gamma(\cO_Y)\to H_0(\mathbf{x};\myR \Gamma(\cO_Y))$ where $\mathbf{x}=x_1,\dots,x_n$ is a generating set for $\fm$, we are done.
\end{proof}

{%
\begin{remark}
\label{rem.DerivedCatAnnihilatorThanksToBriggs}
    As pointed out to us by Benjamin Briggs, $I^{\Kos}$ can also be viewed as the set of elements of $R$ that annihilate the object 
    $K_{\mydot}({\bf f}; \myR \Gamma(\cO_Y) )$ in $D(R)$.  More generally, as pointed out to us by Briggs, if an element of $R$ annihilates $H_0(C_{\mydot})$ where $C_{\mydot}$ is a differential graded $R$-algebra, then that element annihilates $C_{\mydot}$ in $D(R)$. 
    
    We briefly sketch the argument that Briggs explained to us.
    Let $C_{\mydot} := K_{\mydot}({\bf f}; \myR \Gamma(\cO_Y) )$ and let $\phi_g\in\Hom_R(C_{\mydot},C_{\mydot})$ be the multiplication-by-$g$ map for $g\in I^{\KH}$.  Consider $1_C \in C_0$ which becomes the unit $1 \in H^0(C_{\mydot})$.  We know that $g 1_C = d(h)$ for some $h \in C_1$.  Then $\phi_g=d\circ\phi_h+\phi_h\circ d$ where $\phi_h\colon C_{\mydot}\to\Sigma C_{\mydot}$ is multiplication by $h$. This gives a nullhomotopy and thus $\phi_g=0\in\Hom_{D(R)}(C_{\mydot},C_{\mydot})$.  As pointed out by Srikanth Iyengar, this is essentially the same as the corresponding argument for the Koszul complex, see \cite[Proposition 1.6.5]{BH}.
\end{remark}
}

The formation of the \KosClos operation commutes with localization, completion, and in fact any flat map with rational singularity fibers.  Famously, the formation of tight closure in characteristic $p > 0$ does not commute with localization \cite{BrennerMonsky.TCLocalization}, although plus closure does.  Compare the following with \cite[Section 7]{HochsterHunekeFRegularityTestElementsBaseChange}.  

\begin{proposition}
    \label{prop.FlatMapsWithRationalSingularitiesFibers}
    Suppose that $R \to S$ is a flat map of rings both satisfying the conditions of \autoref{set.NiceSetup}.  Suppose further that the fibers of $R \to S$ have rational singularities (for instance, if the fibers are nonsingular).  Then 
    \[  
        I^{\Kos}S = (IS)^{\Kos}
    \]
    As a consequence, the formation of \KosClos commutes with localization and completion along any ideal: $I^{\Kos} W^{-1}R = (W^{-1}I)^{\Kos}$ and $I^{\Kos} \widehat{R} = (I \widehat{R})^{\Kos}$.
\end{proposition}
\begin{proof}
    Let $\pi_R : W \to \Spec R$ be a projective resolution of singularities and consider the base change $\pi_S : W \times_R S \to \Spec S$.  While $\pi_S$ is not a resolution of singularities, it is proper and $W \times_R S$ does have rational singularities by \cite[Footnote {\bf Q} to Table 2]{Murayama.UniformTreatmentOfGrothendieckLocalization} (for the variety case this is \cite[Th\'eor\`em 5]{ElkikDeformationsOfRational}, also see \cite[Theorem 9.3]{Murayama.VanishingForQSchemes}).   

    Since $R \to S$ is flat, we see that $\pi_S : W \times_R S \to \Spec S$ is birational in the weaker sense that there exists an element $f \in R$, not in any minimal prime of $R$ (and hence not in a minimal prime of $S$) such that $\pi_S$ is an isomorphism after inverting $f$.  In fact, we claim that components of $W \times_R S$ dominate components of $\Spec S$ and so $W \times_R S \to \Spec S$ is birational in the stronger sense (that there is a bijection between irreducible components) as well.
    Working locally on a chart $\Spec T$ of $W$ intersecting each component of $W$ nontrivially, we have that $R[f^{-1}] \to T[f^{-1}]$ is an isomorphism, hence so is $S[f^{-1}] \to (S \otimes_R T)[f^{-1}]$.  Thus if $S \otimes_R T$ has a minimal prime not dominating a minimal prime of $S$, it must contain $f$.  But since $R \to T$ is birational, $f$ is not in a minimal prime of $T$, and so $T \xrightarrow{\times f} T$ is injective since $T$ is reduced.  But then $S \otimes_R T \xrightarrow{\times f} S \otimes_R T$ is also injective.  So $f$ cannot be in a minimal prime.  Thus $W \times_R S \to \Spec S$ really is birational in this stronger sense as claimed.
    
    It follows that if $\kappa : Y \to W \times_R S \to \Spec S$ is a resolution of 
    $W \times_R S$, and thus also of $\Spec S$, that 
    \[ 
        \myR \Gamma(\cO_Y) \cong \myR \Gamma( \cO_{W \times_R S}) \cong (\myR \Gamma( \cO_W)) \otimes_R S
    \] 
    where the second isomorphism is the derived projection formula (\cite[II, Proposition 5.6]{HartshorneResidues}) and the first follows as $W \times_R S$ has rational singularities.  We do not need to derive the tensor product as $R \to S$ is flat.

    Hence, 
    \begin{equation}
        \label{eq.prop.FlatMapsWithRationalSingularitiesFibers}
        R \to H_0(K_{\mydot}({\bf f}; \myR \Gamma(\cO_W)))
    \end{equation}
    base changes to 
    \[
        S \to H_0(K_{\mydot}({\bf f}; \myR \Gamma(\cO_W))) \otimes_R S = H_0(K_{\mydot}({\bf f}; \myR \Gamma(\cO_W) \otimes_R S)) = H_0\big(K_{\mydot}({\bf f}; \myR \Gamma( \cO_Y))\big) 
    \]
    Thus $I^{\Kos}$, the kernel of \autoref{eq.prop.FlatMapsWithRationalSingularitiesFibers}, base changes to the kernel of $S \to H_0(K_{\mydot}({\bf f}; \myR \Gamma(\cO_W)))$, which is what we wanted to show.
\end{proof}

We next show persistence.

\begin{proposition}
\label{prop.PersistenceOfKos}
    Suppose $R \to S$ is a map of rings satisfying \autoref{set.NiceSetup}.  Then $I^{\Kos} S \subseteq (IS)^{\Kos}$. That is, \KosClos is persistent.
\end{proposition}
\begin{proof}
    Let $\pi : W \to \Spec R$ be a projective resolution of singularities.  While the base change $\pi_S : W \times_R S \to \Spec S$ may not be birational, it is still projective and surjective.  By restricting to irreducible components and taking hyperplane sections, there exists a closed subscheme $W' \subseteq W \times_R S$ such that the induced $W' \to \Spec S$ is an alteration.  Taking a further resolution of singularities of $W'$, we can construct a regular alteration $\kappa : Y \to \Spec S$.  In particular we have the following commutative diagram of maps of schemes:
    \[
        \xymatrix{
            Y \ar[d]_{\kappa} \ar[r] & W \ar[d]^{\pi} \\
            \Spec S \ar[r] & \Spec R
        }
    \]
    It follows that we have a diagram 
    \[
        \xymatrix{
            R \ar[d] \ar[r] & S \ar[d] \\
            \myR\Gamma(\cO_W) \ar[r] & \myR \Gamma( \cO_Y).
        }
    \]
    Writing $I = ({\bf f})$, tensoring with $\otimes_R^{\myL} K_{\mydot}({\bf f}; R)$, and taking $0$th homology, we obtain the diagram:
    \[
        \xymatrix{
            R \ar[d] \ar[r] & S \ar[d]\\
            H_0 \big(K_{\mydot}({\bf f}; R)\big) \ar[d] \ar[r] & H_0 \big(K_{\mydot}({\bf f}; S)\big) \ar[d] \\
            H_0 \big(K_{\mydot}({\bf f}; \myR\Gamma(\cO_W))\big) \ar[r] & H_0 \big(K_{\mydot}({\bf f}; \myR \Gamma( \cO_Y))\big).
        }
    \]
    The kernel of the left column maps into the kernel of the right column, and the result follows.
\end{proof}

We observe that the \KosClos behaves well under finite extensions.

\begin{prop}
\label{prop.FiniteExtensionComputationKHClosure}
    Let $R\to S$ be a finite extension of rings satisfying \autoref{set.NiceSetup} and where each minimal prime of $S$ lies over a minimal prime of $R$.  Suppose further that $I\subseteq R$ is an ideal. Then
    \[I^{\Kos}=(IS)^{\Kos}\cap R.\]
\end{prop}
\begin{proof}
    Let $I=(\mathbf{f})$. Let $\pi:Y\to\Spec S$ be a resolution of singularities so that the composition $\tau\colon Y\to\Spec R$ is a regular alteration. Then because $K_{\mydot}(\mathbf{f};R)\otimes_RS\cong K_{\mydot}(\mathbf{f};S)$ we have the following diagram where the bottom map is an isomorphism of $R$-modules.
    \[\begin{tikzcd}
        R\arrow[r]\arrow[d,labels=left,"\alpha"]&S\arrow[d, "\beta"]\\
        K_{\mydot}(\mathbf{f};\myR\Gamma(\cO_Y))\arrow[r,labels=below,"="]&K_{\mydot}(\mathbf{f};\myR \Gamma(\cO_Y)).
    \end{tikzcd}\]
    We see that $\ker \alpha = (\ker \beta) \cap R$ as desired.
\end{proof}

We give a direct proof that the \KosClos is contained in the integral closure.  This can also be deduced from results in later sections or from reduction modulo $p \gg 0$, but we give a proof using our resolution of singularities.

\begin{proposition}
    \label{prop.KHClosureInIntegralClosure}
    Suppose $R$ is as in \autoref{set.NiceSetup} and $({\bf f}) = I \subseteq R$ is an ideal.  Then $I^{\Kos} \subseteq \overline{I}$.
\end{proposition}
\begin{proof}
    We first notice that there is a map $K_{\mydot}({\bf f}) \to R/I$.  Hence
    \[
        I^{\Kos} \subseteq \ker\big(R \to H_0 ((R/I) \otimes^{\myL} \myR \Gamma(\cO_Y))\big)
    \]
    for $\pi : Y \to \Spec R$ a resolution of singularities.  Without loss of generality, we may assume that $I \cO_Y = \cO_Y(-G)$ is a line bundle.  
    
    The derived projection formula gives us a map $I \otimes^{\myL} \myR \Gamma(\cO_Y) \to \myR \Gamma(\cO_Y \otimes^{\myL} \myL \pi^* I) = \myR \Gamma(\myL \pi^* I)$, in fact, an isomorphism thanks to \cite[II, Proposition 5.6]{HartshorneResidues}.  Now we have a map $\myL \pi^* I \to I \cO_Y = \cO_Y(-G)$ and so composing, we obtain 
    \[
        I \otimes^{\myL} \myR \Gamma(\cO_Y) \to \myR \Gamma(\cO_Y(-G)).
    \]
    We thus obtain the diagram:
    \[
        \xymatrix{
            I \otimes^{\myL} \myR \Gamma(\cO_Y) \ar[r] \ar[d] & \myR \Gamma(\cO_Y) \ar[r] \ar[d]^= & R/I \otimes^{\myL} \myR \Gamma(\cO_Y) \ar[r]\ar[d] & + 1\\
            \myR \Gamma(\cO_Y(-G)) \ar[r] & \myR \Gamma(\cO_Y) \ar[r] & \myR \Gamma(\cO_G)\ar[r] & + 1
        }
    \]
    It immediately follows that $\ker\big(R \to H_0 ((R/I) \otimes^{\myL} \myR \Gamma(\cO_Y))\big)$ is contained in $\ker\big( R \to H_0(\myR \Gamma(\cO_G))\big)$.   But that equals
    \[ 
        \Image \Big(\Gamma(\cO_Y(-G)) \to \Gamma(\cO_Y)\Big) \cap R = \overline{IR^{\mathrm N}} \cap R = \overline{I}.
    \]
    where $R^{\mathrm{N}}\simeq\Gamma(\cO_Y)$ is the normalization of $R$.
\end{proof}

{%
Finally, we consider a technique that can be used to reduce to the $\m$-primary case, and which will be useful in other contexts as well.

\begin{proposition}
\label{prop.IntersectionContainments}
    Suppose $(R,\m)$ is a local ring satisfying \autoref{set.NiceSetup}.  Suppose $J = ({\bf f})$ and $g \in R$.  Then 
    \[
        J^{\Kos} = \bigcap_{t>0} \big( J + (g^t) \big)^{\Kos}.
    \]
    As a consequence,  
    \[
        J^{\Kos} = \bigcap_{t>0} \big( J + \m^t \ \big)^{\Kos}
    \]
    and so \KosClos can be computed from the \KosClos of $\m$-primary ideals.  
\end{proposition}
\begin{proof}
The $\subseteq$ containments are clear.
    For the first statement, as in the proof of \autoref{prop.kosIndependence} we have the following diagram with exact rows
    \[\begin{tikzcd}
    R\arrow[r,"\cdot g^t"]\arrow[d]&R\arrow[r]\arrow[d, "\alpha"]\arrow[dr]&R/(g^t)\arrow[d]\\
    H_0(\mathbf{f};\myR \Gamma(\cO_Y))\arrow[r,"\cdot g^t"]&H_0(\mathbf{f};\myR \Gamma(\cO_Y))\arrow[r]&H_0(\mathbf{f},g^t;\myR \Gamma(\cO_Y)).
\end{tikzcd}\]
The kernel of the diagonal arrow is $(J + (g^t))^{\Kos}$.
Suppose then that $x \in \bigcap_t \big( J + (g^t) \big)^{\Kos}$.  It follows that $\alpha(x) \in g^t(H_0(\mathbf{f};\myR \Gamma(\cO_Y)))$ for every $t>0$.  As $R$ is local and $H_0(\mathbf{f};\myR \Gamma(\cO_Y))$ is finitely generated, Nakayama's lemma implies that $\alpha(x) = 0$ and hence that $x \in J^{\Kos}$.

For the second statement, if $\m = (x_1, \dots, x_d)$, then observe that it suffices to show that $ J^{\Kos} = \bigcap_{t} \big( J + (x_1^t, \dots, x_d^t) \, \big)^{\Kos}$.  We sketch an argument to prove this below.
Observe that
\[
    \bigcap_{t} \big( J + (x_1^t, \dots, x_d^t)\,\big)^{\Kos} = \bigcap_{t_1, \dots, t_d} \big( J + (x_1^{t_1}, \dots, x_d^{t_d})\,\big)^{\Kos} = \bigcap_{t_1} \dots \bigcap_{t_d} \big( J + (x_1^{t_1}, \dots, x_d^{t_d})\,\big)^{\Kos}
\]
where the first equality comes from cofinality of the sequences of ideals.  But now, using the first statement of the proposition, we can work one intersection at a time, so we obtain
\[
    \bigcap_{t_1} \dots \bigcap_{t_{d-1}} \bigcap_{t_d} \big( J + (x_1^{t_1}, \dots, x_{d-1}^{t_{d-1}}, x_d^{t_d}) \,\big)^{\Kos} = \bigcap_{t_1} \dots \bigcap_{t_{d-1}} \big( J + (x_1^{t_1}, \dots, x_{d-1}^{t_{d-1}})\, \big)^{\Kos} = \dots = J^{\Kos}.
\]
\end{proof}
}
\section{Colon-capturing, rational singularities, and the  Brian\c{c}on-Skoda theorem }

We are ultimately interested in showing that the above closure operation satisfies generalized colon capturing, as this would imply the existence of big Cohen-Macaulay modules in {equal} characteristic zero without requiring reduction to positive characteristic.
We use this section to present results in this direction.

\begin{theorem}\label{thm.cc}
    Let $(x_1,\dots,x_d)$ be a system of parameters for a{n equidimensional} local ring $R$ satisfying \autoref{set.NiceSetup}. Then \KosClos satisfies the following:
    \begin{enumerate}
        \item improved strong colon-capturing, version A, \label{thm.cc.versionA}
        \[
            (x_1^{t}, \dots, x_k)^{\Kos} : x_1^a \subseteq (x_1^{t-a}, \dots, x_k)^{\Kos},
        \]
        \item strong colon-capturing, version B, \label{thm.cc.versionB}
        \[
            (x_1, \dots, x_k)^{\Kos} : x_{k+1} \subseteq (x_1, \dots, x_{k})^{\Kos}, 
        \]
        and as a consequence
        \item colon-capturing.
    \end{enumerate}
\end{theorem}

\begin{proof}
We first prove that \KosClos satisfies improved strong colon-capturing, version A. Fix a regular alteration $\pi\colon Y\to\Spec R$. Consider the following diagram
\[
\begin{tikzcd}
    R\arrow[r,"x_1^a"]\arrow[d]&R\arrow[d]\\
    H_0(x_1^{t-a},x_2,\dots,x_k;\myR \Gamma(\cO_Y))\arrow[r,"x_1^a"]&H_0(x_1^t,x_2,\dots,x_k;\myR \Gamma(\cO_Y)).
\end{tikzcd}
\]
Note that if we can show the bottom map, whose construction we explain below, is injective, we are done. To see this, take $r\in(x_1^t,x_2,\dots,x_k)^{\Kos}\colon x_1^a$. Then $rx_1^a=0\in H_0(x_1^t,x_2,\dots,x_k;\myR \Gamma(\cO_Y))$, so if the bottom map is injective we must have $r\in(x_1^{t-a},x_2,\dots,x_k)^{\Kos}$.

Now, to see that the bottom map is injective, we consider the following related diagram
\[
\begin{tikzcd}
    R\arrow[r]\arrow[d,swap,"x_1^{t-a}"]&R\arrow[d,"x_1^t"]\\
    R\arrow[r,swap,"x_1^a"]&R.
\end{tikzcd}
\]
By the definition of the Koszul complex and the octahedral axiom, we get the following exact triangle in $D(R)$
\[
\begin{tikzcd}
    K_{\mydot}(x_1^{t-a};R)\arrow[r]&K_{\mydot}(x_1^t;R)\arrow[r]&K_{\mydot}(x_1^a;R)\arrow[r]&\hspace{1mm}.
\end{tikzcd}
\]
Tensoring with $K_{\mydot}(x_2,\dots,x_k;\myR \Gamma(\cO_Y))$ we get the following exact triangle
\[
{
\begin{tikzcd}[column sep=small]
    K_{\mydot}(x_1^{t-a},x_2,\dots,x_k;\myR \Gamma(\cO_Y))\arrow[r]&K_{\mydot}(x_1^t,x_2,\dots,x_k;\myR \Gamma(\cO_Y))\arrow[r]&K_{\mydot}(x_1^a,x_2,\dots,x_k;\myR \Gamma(\cO_Y))\arrow[r]&\hspace{1mm}.
\end{tikzcd}}
\]
Then, using the fact that $H_i(x_1^a,x_2,\dots,x_k;\myR \Gamma(\cO_Y))=0$ for all $i\geq1$, we get that the map
\[
\begin{tikzcd}
    H_0(x_1^{t-a},x_2,\dots,x_k;\myR \Gamma(\cO_Y))\arrow[r,"x_1^a"]&H_0(x_1^t,x_2,\dots,x_k;\myR \Gamma(\cO_Y))
\end{tikzcd}
\]
is injective and we are done.

The proof that \KosClos satisfies strong colon-capturing, version B follows from the following diagram as in the proof of colon-capturing: 
\[
\begin{tikzcd}
    R\arrow[r,"x_{k+1}"]\arrow[d]&R\arrow[d]\\
    H_0(x_1,\dots,x_k;\myR \Gamma(\cO_Y))\arrow[r,"x_{k+1}"]&H_0(x_1,\dots,x_k;\myR \Gamma(\cO_Y))
\end{tikzcd}
\]
Taking $r\in(x_1,\dots,x_k)^{\Kos}\colon x_{k+1}$ we note that the image of $rx_{k+1}$ is zero in the bottom right of the diagram. The injectivity of the bottom map implies that $r\in(x_1,\dots,x_k)^{\Kos}$. Note that because \KosClos satisfies strong colon-capturing, version B, it also satisfies colon-capturing.
\end{proof}


\subsection{Rational singularities}

The proof of \autoref{thm.cc} suggests the following stronger result:

\begin{thm}
    \label{thm.InjectiveMapOnCohomologyKoszulToLocal}
    Let $(R,\fm)$ be a local ring and $x_1,\dots,x_d$ be a system of parameters for $R$ and let $X$ be a big Cohen-Macaulay $R$-complex. Then the natural map
    \[H_0(x_1,\dots,x_k;X)\cong H^k(x_1,\dots,x_k;X)\to H_{(x_1,\dots,x_k)}^k(X)\]
    is injective for all $k$.
\end{thm}
\begin{proof}
This largely follows from the proof of \autoref{thm.cc}. Note that if $x$ is a regular element, because $X$ is big Cohen-Macaulay, the natural map $H^1(x^k;X)\to H^1(x^{k+1};X)$ is injective by an argument similar to the proof of \autoref{thm.cc}. Since $X$ is big Cohen-Macaulay, $K^{\mydot}(x^\infty;X)\cong H^1_{(x)}(X)$ and the map
\[
H^1(x;X)\to H^1_{(x)}(X)\cong H^1(x^{\infty};X)
\]
is injective (as all the maps in the colimit are injective). More generally, we get that all the maps in the following sequence are injective
\[H^k(x_1^i,\dots,x_k^i;X)\to H^k(x_1^{i+1},\dots,x_k^i;X)\to\cdots\to H^k(x_1^{i+1},\dots,x_k^{i+1};X).\]
Indeed, these maps are the ones we proved were injective in the proof of \autoref{thm.cc} \autoref{thm.cc.versionA}.
Thus the map
\[H^k(x_1,\dots,x_k;X)\to H^k_{(x_1,\dots,x_k)}(X)\]
is injective for all $k$ and we are done.
\end{proof}

{%
We apply this now to the study of rational singularities.  First we get a description of \KosClos for parameter ideals.  Indeed, this is exactly what one might hope for in view of \cite{SmithFRatImpliesRat,HaraRatImpliesFRat}.

\begin{proposition}
    \label{prop.KHClosureOfFullParameterIdeal}
    Let $(R,\fm)$ be an {equidimensional} local ring satisfying \autoref{set.NiceSetup}.  If $x_1, \dots, x_d$ is a full system of parameters, then $(x_1, \dots, x_i)^{\Kos}$ is the kernel of the composition
    \begin{equation}
        \label{eq.prop.KHClosureOfFullParameterIdeal}
        R \xrightarrow{1 \mapsto \big[{1 + (x_1, \dots, x_i)}\big]} H^i_{(x_1, \dots, x_i)}(R) \to H^i_{(x_1, \dots, x_i)}(\myR\Gamma(\cO_Y)).
    \end{equation}
    In particular, $(x_1, \dots, x_d)^{\Kos}$ is the kernel of 
    $R \xrightarrow{1 \mapsto \big[{1 + (x_1, \dots, x_d)}\big]} H^d_{\fm}(\myR\Gamma(\cO_Y))$.
\end{proposition}
\begin{proof}
    Write ${\bf x} = x_1, \dots, x_i$.  Notice we can factor the composition \autoref{eq.prop.KHClosureOfFullParameterIdeal} as 
    \[
        R \to H^0\big(K_{\mydot}({\bf x}; \myR\Gamma(\cO_Y))\big) \xrightarrow{\gamma} \colim_k H^0\big(K_{\mydot}({\bf x}^k; \myR\Gamma(\cO_Y))\big) \cong H^i_{(x_1, \dots, x_i)}(\myR\Gamma(\cO_Y)).
    \]
    As $\myR\Gamma(\cO_Y)$ is Cohen-Macaulay, the colimiting maps are still injective and so $\gamma$ injects.  The result follows.
\end{proof}

It turns out that a single full parameter ideal being \KosClosed is a quite strong condition, just like for tight closure. 

\begin{lemma}[{\cf \cite[Theorem 4.3]{HochsterHunekeFRegularityTestElementsBaseChange}}]
    \label{lem.ParameterIdealObservationsFollowingHochsterHuneke}
    Suppose $(R, \fm)$ is a{n equidimensional} local ring satisfying \autoref{set.NiceSetup}.  Suppose $x_1, \dots, x_d$ is a full system of parameters.  Suppose $(x_1, \dots, x_d) = (x_1, \dots, x_d)^{\KH}$.  Then 
    \begin{enumerate}
        \item $(x_1, \dots, x_i) = (x_1, \dots, x_i)^{\KH}$ for all $1 \leq i \leq d$.  \label{lem.ParameterIdealObservationsFollowingHochsterHuneke.fullImpliesPartial}
        \item $R$ is Cohen-Macaulay. \label{lem.ParameterIdealObservationsFollowingHochsterHuneke.FullImpliesCM}
        \item For any $t > 0$, $(x_1^t, \dots,x_d^t) = (x_1^t, \dots,x_d^t)^{\KH}$. \label{lem.ParameterIdealObservationsFollowingHochsterHuneke.FullImpliesPowers}
        \item Any parameter ideal is \KosClosed.  \label{lem.ParameterIdealObservationsFollowingHochsterHuneke.AnyParamIsClosed}
        \item $R$ has rational singularities and hence all ideals are \KosClosed.\label{lem.ParameterIdealObservationsFollowingHochsterHuneke.RationalSingularities}
    \end{enumerate}
\end{lemma}
Much of the proof of this result is essentially identical to that of a portion of \cite[Theorem 4.3]{HochsterHunekeFRegularityTestElementsBaseChange} by formally replacing tight closure by \KosClos.  The key point is that \KosClos satisfies strong colon capturing and can be computed by maps to local cohomology, just like tight closure.  We include a careful proof for the convenience of the reader.  We thank Kyle Maddox for suggesting this question and for some valuable discussions.
\begin{proof}    
    Suppose $(x_1, \dots, x_{i+1})$ is a partial parameter ideal that is \KosClosed.  We wish to show that $J = (x_1, \dots, x_i)$ is also \KosClosed.  Suppose $u \in (x_1, \dots, x_i)^{\KH} \subseteq (x_1, \dots, x_{i+1})^{\KH} = (x_1, \dots, x_{i+1})$.  Thus we can write $u = v + x_{i+1}r$ for some $v \in (x_1, \dots, x_i)$ and $r \in R$.  Then $u-v \in (x_1, \dots, x_i)^{\KH}$ and so 
    \[ 
        r \in (x_1, \dots, x_i)^{\KH} : x_{i+1} = (x_1, \dots, x_i)^{\KH} = J^{\KH}
    \]    
    by \autoref{thm.cc} \autoref{thm.cc.versionB} (the other containment $\supseteq$ always holds).  Thus $u \in J + x_{i+1}J^{\KH}$ and so $J^{\KH} = J + x_{i+1}J^{\KH}$ which forces $u \in J$ by Nakayama's lemma.  This proves \autoref{lem.ParameterIdealObservationsFollowingHochsterHuneke.fullImpliesPartial} by descending induction.
    
    For \autoref{lem.ParameterIdealObservationsFollowingHochsterHuneke.FullImpliesCM}, if $(x_1, \dots, x_d)$ is \KosClosed, then by \autoref{lem.ParameterIdealObservationsFollowingHochsterHuneke.fullImpliesPartial} and colon capturing, $(x_1, \dots, x_i) : x_{i+1} = (x_1, \dots, x_i)$ for all $i$.  Hence $R$ is Cohen-Macaulay.

    Suppose $(x_1, \dots, x_d)$ is a full parameter ideal that is \KosClosed (and hence $R$ is Cohen-Macaulay).  We first show that $(x_1^t, \dots, x_d^t)$ is also \KosClosed.  Suppose it is not.  Then some element $r \in R$, mapping to the socle of $R/(x_1^t, \dots, x_d^t)$, must also be in $(x_1^t, \dots, x_d^t)^{\KH}$.  As $R$ is Cohen-Macaulay, we can write $r = x_1^{t-1} \dots x_d^{t-1} u$ for some $u \in R$ mapping into the socle of $R/(x_1, \dots, x_d)$.  As $\overline{u} \in R/(x_1, \dots, x_d)$ and $\overline{r} \in R/(x_1^t, \dots, x_d^t)$ map to the same place in $H^d_{\m}(\myR \Gamma(\cO_Y))$, we see that $u \in (x_1, \dots, x_d)^{\KH} = (x_1, \dots, x_d)$.  But then clearly $r \in (x_1^t, \dots, x_d^t)$.  This proves \autoref{lem.ParameterIdealObservationsFollowingHochsterHuneke.FullImpliesPowers}.  

    Next consider the map
    \[
        H^d_{\m}(R) = \colim_k H_0({\bf x}^k; R) \to  \colim_k H_0({\bf x}^k; \myR \Gamma(\cO_Y)) = H^d_{\m}(\myR\Gamma(\cO_Y)).
    \]
    Since each $(x_1^k, \dots, x_d^k)$ is \KosClosed, the individual maps $R/(x_1^k, \dots, x_d^k) \to H^d_{\m}(\myR\Gamma(\cO_Y))$ inject, and so we see that \begin{equation}
    \label{eq.InjectionOnLocalCohomologyAssumingSOPClosed}
        H^d_{\m}(R) \to H^d_{\m}(\myR\Gamma(\cO_Y))
    \end{equation}
    injects as well.  
    
    To prove \autoref{lem.ParameterIdealObservationsFollowingHochsterHuneke.AnyParamIsClosed}, suppose $(y_1, \dots, y_d)$ is another full parameter ideal.  As $R$ is Cohen-Macaulay $R/(y_1, \dots, y_d) \to H^d_{\m}(R)$ injects.  It follows that 
    \[
        R/(y_1, \dots, y_d) \to H^d_{\m}(\myR\Gamma(\cO_Y))
    \]
    injects.  Hence $(y_1, \dots, y_d)$ is \KosClosed.
    
    For the final statement, as $R$ is Cohen-Macaulay, we see that  \autoref{eq.InjectionOnLocalCohomologyAssumingSOPClosed}
    means that $R$ is pseudo-rational (see \cite{LipmanTeissierPseudoRational}) and hence $R$ has rational singularities.
\end{proof}

We immediately obtain the following corollary.
\begin{corollary}
    \label{cor.RatSingsVsKoszulClosedParameterIdeal}
    Let $(R,\fm)$ be a{n equidimensional} local ring satisfying \autoref{set.NiceSetup}. Then the following are equivalent.
    \begin{enumerate}
        \item $R$ has rational singularities.
        \item Every ideal of $R$ is \KosClosed.  \label{cor.RatSingsVsKoszulClosedParameterIdeal.allClosed}
        \item Some ideal generated by a full system of parameters of $R$ is \KosClosed.\label{cor.RatSingsVsKoszulClosedParameterIdeal.singleClosed}  
    \end{enumerate}  
\end{corollary}
\begin{proof}
    If $R$ has rational singularities, $R \to \myR \Gamma(\cO_Y)$ is an isomorphism and so it easily follows that all ideals are \KosClosed.  That clearly implies that one full parameter ideal is \KosClosed.  Finally, the fact that \autoref{cor.RatSingsVsKoszulClosedParameterIdeal.singleClosed} implies rational singularities is simply \autoref{lem.ParameterIdealObservationsFollowingHochsterHuneke}.
\end{proof}
}
Combining \autoref{prop.FiniteExtensionComputationKHClosure} with \autoref{cor.RatSingsVsKoszulClosedParameterIdeal} yields the following result.

\begin{corollary}
    Suppose $R$ as in \autoref{set.NiceSetup} has normalization $R^{\mathrm{N}}$ with rational singularities (for example, if the normalization is nonsingular, which is automatic if $R$ is 1-dimensional).  Then for any ideal $I \subseteq R$ we have that $I^{\Kos} = I R^{\mathrm{N}} \cap R$.
\end{corollary}

In fact, we obtain the following stronger variant of \autoref{cor.RatSingsVsKoszulClosedParameterIdeal}.  

\begin{proposition}
    \label{prop.TestIdealForKosClosure}
    The \KosTestIdeal $\tau_{\Kos}(R) = \bigcap_{J \subseteq R} (J : J^{\Kos})$ contains the {\traceIdealComplex} of $\myR\Gamma(\cO_Y)$ in $R$ for any resolution of singularities or regular alteration $\pi : Y \to \Spec R$.  Furthermore, if $R$ is Cohen-Macaulay, then these two ideals are equal and also coincide with $\Ann_R(\omega_R / \Gamma( \omega_Y))$ where $\pi : Y \to \Spec R$ is a resolution of singularities.
\end{proposition}
\begin{proof}
Pick $g$ in the \traceIdealComplex of $\myR \Gamma(\cO_Y)$.
    By \autoref{lem.TraceOfComplexDGAVersion}, this implies that there is a map $\gamma : \myR\Gamma(\cO_Y) \to R$ such that the composition $R \to \myR\Gamma(\cO_Y) \to R$ is multiplication by $g$. 
    We will prove that $g J^{\Kos} \subseteq J$ for any ideal $J$ of $R$.  Fix $J = ({\bf f}) = (f_1, \dots, f_n)$.  Consider the factorization:
    \[
        R \to K_{\mydot}({\bf f}; R) \to K_{\mydot}({\bf f}; \myR\Gamma(\cO_Y)) \xrightarrow{\gamma} K_{\mydot}({\bf f}; R).
    \]
    Taking $H^0$, the kernel of the composition becomes $J : g$ (as $H^0(K_{\mydot}({\bf f}; R)) = R/J$).  Hence we see that $J^{\Kos} \subseteq J : g$, or in other words that $g \in J : J^{\Kos}$ as desired.

    For the second statement, suppose $R$ is Cohen-Macaulay,  $g \in \tau_{\Kos}(R)$, and $x_1, \dots, x_d \in R$ is a full system of parameters with $J_k = (x_1^k, \dots, x_d^k)$.  Then $g J_k^{\Kos} \subseteq J_k$ for all $k$.  Consider some $\eta = [h + (x_1^k, \dots, x_d^k)] \in H^d_{\fm}(R)$ mapping to zero in $H^d_{\fm}(\myR\Gamma(\cO_Y))$ so that $h \in J_k^{\Kos}$ by \autoref{cor.RatSingsVsKoszulClosedParameterIdeal}. 
    Therefore, $gh \in (x_1^k, \dots, x_d^k)$ which implies that $g\eta = 0$.  Thus we have 
    \[
        0 = g \ker\big(H^d_{\fm}(R) \to H^d_{\fm}(\myR\Gamma(\cO_Y))\big).
    \]
    By Matlis and Grothendieck duality, we obtain that $g (\omega_R / \Gamma(\omega_Y)) = 0$, or in other words that $g \omega_R \subseteq \Gamma(\omega_Y)$.  But this means that we can factor the multiplication-by-$g$-map on $\omega_R$ as 
    \[
        \omega_R \to \Gamma(\omega_Y) \to \omega_R.
    \]
    Taking Grothendieck duality and using that $R$ is Cohen-Macaulay, we obtain a composition
    \[
        R \leftarrow \myR\Gamma(\cO_Y) \leftarrow R 
    \]
    where the composition is again multiplication by $g$.  This proves that $g$ is in the {\traceIdealComplex} of $\myR\Gamma(\cO_Y)$ via \autoref{lem.TraceOfComplexDGAVersion} as desired.
    
    For the final statement, the argument above shows that $\Ann_R(\omega_R/\Gamma(\omega_Y))$ is contained in the {\traceIdealComplex} of $\myR\Gamma(\cO_Y)$. Given an element of the {\traceIdealComplex} of $\myR\Gamma(\cO_Y)$, running the argument in reverse shows that the reverse containment holds. 
\end{proof}

This immediately tells us what the $\Kos$-test ideal reduces to modulo $p \gg 0$ in a Cohen-Macaulay ring.

\begin{corollary}
    \label{cor.KHTestIdealReducesToParameterTestIdeal}
    If $R$ is Cohen-Macaulay and essentially of finite type over a field of characteristic zero, then the \KosTestIdeal reduces modulo $p \gg 0$ to the parameter test ideal.
\end{corollary}
\begin{proof}
 We may reduce to the finite type case.  The result follows from \autoref{prop.JOmegaReducesToTauOmega}.
\end{proof}

We can also say something more precise about how the closure operations themselves reduce modulo $p \gg 0$.

\begin{proposition}
\label{prop.KHContainedInPlusModP}
    Suppose $k$ is a field of characteristic zero and $R$ is a finitely generated $k$-algebra.  Fix $J \subseteq R$ an ideal.  Let $R_t$ and $J_t$, $(J^{\Kos})_t$, for $t \in \mathrm{m}-\Spec A$ denote a family of reduction-to-characteristic $p > 0$ models of $R$, $J$ and $(J^{\Kos})_t$ respectively.  Then 
    \[
        (J^{\Kos})_t \subseteq (J_t)^+
    \]
    for a Zariski-dense open set of $t$ in $\mathrm{m}\text{-}\Spec A$. 
\end{proposition}
\begin{proof}
    Enlarging $A$ if necessary, we may assume that a resolution of singularities $\pi : Y \to \Spec R$ is also reduced to characteristic $p$ to become $\pi_t : Y_t \to \Spec R_t$.  Thanks to \cite{BhattDerivedDirectSummandCharP}, $R_t \to (R_t)^+$ factors through $R_t \to \myR \Gamma( \cO_{Y_t})$.  Now, fix $J = (f_1, \dots, f_n)$.  As $H^0 K_{\mydot}({\bf f}; R_t^+) = R_t^+/({\bf f})$, we have that 
    \[ 
        \ker \Big( H^0 K_{\mydot}({\bf f}; R_t) \to H^0 K_{\mydot}({\bf f}; R_t^+) \Big) = J_t^+/J_t.
    \]
    Furthermore, the datum of $K_{\mydot}({\bf f}; R) \to K_{\mydot}({\bf f}; \myR\Gamma(\cO_Y))$ can easily be seen to be reduced to characteristic $p > 0$ and the result follows.
\end{proof}

We will see later in \autoref{sec.ComputationsInM2} that in fact $J^{\KH}$ is \emph{strictly} tighter than tight closure in {equal} characteristic zero or plus closure.

\begin{remark}
\label{rem.KoszulClosureInCharp}
    In any characteristic, and for any fixed alteration (regular or not) $\pi : Y \to \Spec R$, one can define a closure operation for $I = ({\bf{f}}) \subseteq R$ by 
    \[
        I^{\mathrm{K}\pi} = \ker \Big(R \to H_0 \big(K_{\mydot}({\bf f}; \myR\Gamma(\cO_Y))\big) \Big)
    \]
    It need not satisfy properties like colon capturing of course as $\myR\Gamma(\cO_Y)$ need not be locally Cohen-Macaulay in positive or mixed characteristic.  Regardless, in characteristic $p > 0$, the argument of \autoref{prop.KHContainedInPlusModP} proves $ I^{\mathrm{K}\pi} \subseteq I^+$.  As any finite extension is itself an alteration, we immediately see that 
    \[
        I^+ = \sum_{\pi} I^{\mathrm{K}\pi}.
    \]
\end{remark}


\begin{remark}
    Schoutens proved in \cite[Theorem 10.4]{Schoutens.NonStandardTightClosureForAffineCAlgebras} that (equational) tight closure in equal characteristic zero is contained in all of his ultraproduct closures, in particular generic and non-standard tight closure. As a consequence, since \KosClos\ is contained in the tight closure in {equal} characteristic zero, it is also contained in generic and non-standard tight closure.

    Similarly, \KosClos\ is contained in parasolid closure \cite[Corollary 9.9]{BrennerHowToRescueSolidClosure}.    
\end{remark}

{%
\begin{remark}
    Another interesting closure can be constructed as follows.  For simplicity, suppose that $R$ is a domain finite type over $\bQ$.  Spread $R$ out to $R_{\bZ}$, a finitely generated $\bZ$-algebra domain such that $R_{\bZ} \otimes_{\bZ} \bQ = R$ as if one is beginning the reduction-to-characteristic-$p$ process.  Now fix a prime $p$ and consider $R_{\bZ_{(p)}} := R_{\bZ} \otimes_{\bZ} \bZ_{(p)}$ where $\bZ_{(p)}$ is the localization at the prime ideal $(p)$ (ie, a DVR).  We know that the $p$-adic completion $\widehat{R^+_{(p)}}$ is a balanced big Cohen-Macaulay $R$-algebra by \cite{BhattAbsoluteIntegralClosure}, and so, after inverting $p$, it can be used to create an interesting closure operation on $R$, $J \mapsto J\widehat{R^+_{(p)}}[1/p] \cap R$.  This is closely related to, but in principle slightly smaller than, doing Heitmann's full extended plus closure in mixed characteristic and then inverting $p > 0$, \cite{heitmannepf}.

    Fix $X_{(p)} \to \Spec R_{\bZ_{(p)}}$ a blowup providing a resolution of singularities after inverting $p$ (ie, inducing a resolution $X \to \Spec R$).  We have a factorization $R_{(p)} \to \myR\Gamma(\cO_{X_{(p)}}) \to \myR\Gamma(X^+, \cO_{X^+_{(p)}})$.  
    By \cite[Theorem 3.12]{BhattAbsoluteIntegralClosure}, we see $\widehat{R_{(p)}^+}/p \cong \myR\Gamma(X^+, \cO_{X^+_{(p)}})/p$  (the modulo $p$ on the right means in the derived sense, ie, we are tensoring with the Koszul complex on $p$).   
    But then by derived Nakayama we see that $\widehat{R_{(p)}^+}$ and $\myR\Gamma(X^+, \cO_{X^+_{(p)}})$ agree up to derived completion.  
    Hence $R_{(p)} \to \widehat{R_{(p)}^+}$ factors through $\myR \Gamma(\cO_{X_{(p)}})$.  
    Inverting $p$, we have a factorization 
    \[
        R \to \myR \Gamma(\cO_X) \to \widehat{R^+}[1/p].
    \]
    Hence $J^{\KH} \subseteq J\widehat{R^+_{(p)}}[1/p] \cap R$.  As a consequence, we easily see that if $J_{(p)} \subseteq R_{(p)}$ is a model for $J$ in mixed characteristic, then $J^{\KH} \subseteq J_{(p)}^{\mathrm{epf}} \otimes_{\bZ} \bQ$ where $\mathrm{epf}$ denotes full extended plus closure, see \cite{heitmannepf}.
\end{remark}
}


\subsection{The Brian\c{c}on-Skoda property}

In \cite{MurayamaUniformBoundsOnSymbolicPowers}, Murayama formalized the Brian\c{c}on-Skoda property for closure operations.  Namely, we say that $\cl$ has the \emph{Brian\c{c}on-Skoda property} if for every $n$-generated ideal $J \subseteq R$ and for every integer $k \geq 0$ we have that 
\[
    \overline{J^{n+k-1}} \subseteq (J^{k})^{\cl}.
\]

{
When $k = 1$, it follows from \cite{MaMcDonaldRGSchwede.BSForPseudorRational} that this holds for \KosClos.  
We give an alternate proof of this (which preceded {\it loc. cit.} historically).   }
The proof strategy mimics parts of \cite{LipmanTeissierPseudoRational}, \cite{HochsterHunekeApplicationsofBigCM} and \cite{RodriguezVillalobosSchwede.BrianconSkodaProperty}.  First we handle the case of a system of parameters.

\begin{theorem}
    \label{thm.SkodaForParameterKos}
    Suppose $R$ is a domain satisfying \autoref{set.NiceSetup} and $J \subseteq R$ is an ideal generated by a partial system of parameters, $J = (f_1, \dots, f_n)$.  Then 
    \[
        \overline{J^{n}} \subseteq J^{\Kos}.
    \]
\end{theorem}
{

\begin{proof}
Pick $h \in \overline{J^n}$.  Let $\mu : W \to \Spec R$ denote the normalized blowup of $J$.  By \cite[Lemma 3.1]{RodriguezVillalobosSchwede.BrianconSkodaProperty} (which was first shown in a proof from  \cite{LipmanTeissierPseudoRational}), we see that the \Cech class 
\[ 
    [h/(f_1 \dots f_n)] \in \ker \big( H^n_J(R) \to H^n_{\mu^{-1}(V(J))}(W, \cO_W)\big).
\]
Hence, if $\pi : Y \to W \to \Spec R$ is a resolution, we also see that 
\[
    [h/(f_1 \dots f_n)] \mapsto 0 \in H^n_{\pi^{-1}V(J)}(Y, \cO_Y) \cong H^n_J( \myR \Gamma(\cO_Y)).
\]
Now, $h$ maps to the Koszul class $[h /(f_1 \dots f_n)] = 0$ in the composition
\[
    R \to H_0(\underline{f}, \myR \Gamma(\cO_Y)) \to H^n_J( \myR \Gamma(\cO_Y)).
\]
But $H_0(\underline{f}, \myR \Gamma(\cO_Y)) \hookrightarrow H^n_J( \myR \Gamma(\cO_Y))$ injects by \autoref{thm.InjectiveMapOnCohomologyKoszulToLocal}.  Hence $h \mapsto 0$ in 
$R \to H_0(\underline{f}, \myR \Gamma(\cO_Y))$ and so $h \in J^{\Kos}$ as desired.
\end{proof}}

\begin{theorem}
    \label{thm.SkodaForPowerKos}
    Suppose $R$ is a domain satisfying \autoref{set.NiceSetup} and $J \subseteq R$ is an ideal generated $n$ elements, $J = (f_1, \dots, f_n)$.  Then 
    \[
        \overline{J^{n}} \subseteq J^{\Kos}.
    \]  
\end{theorem}
\begin{proof}
    Without loss of generality, we may assume that $R$ is local with maximal ideal $\m$.
    Suppose $h \in \overline{J^n}$.
    Following the proof of \cite[Theorem 7.1]{HochsterHunekeApplicationsofBigCM}, see also the proof of \cite[Theorem 4.1]{RodriguezVillalobosSchwede.BrianconSkodaProperty}, we can find a local map $\psi : S \to R$ of excellent local domains satisfying the following properties:
    \begin{enumerate}
        \item There exist a partial system of parameters $x_1, \dots, x_n \in S$ with $\psi(x_i) = f_i$.
        \item There exist $z \in \overline{(x_1, \dots, x_n)^n}$ with $\psi(z) = h$.
    \end{enumerate}
    It follows that $z \in (x_1, \dots, x_n)^{\Kos}$ by \autoref{thm.SkodaForParameterKos}.  By persistence (\autoref{prop.PersistenceOfKos}), it follows that $\psi(z)=h \in J^{\Kos}$ as desired.
\end{proof}

\begin{corollary}
    Suppose $R$ is a domain satisfying \autoref{set.NiceSetup} and $I = (f)$ is a principal ideal.  Then $\overline{I} = I^{\Kos}$.
\end{corollary}
\begin{proof}
    We have that $\overline{I} \subseteq I^{\Kos}$ by \autoref{thm.SkodaForPowerKos}.  The reverse containment is \autoref{prop.KHClosureInIntegralClosure}.  
\end{proof}

Based upon \cite{BrianconSkoda, LipmanTeissierPseudoRational,  LipmanSathayeJacobianIdealsAndATheoremOfBS, HHmain, heitmannepf}, it is natural to expect that 
\[
    \overline{J^{n+k-1}} \subseteq (J^k)^{\Kos}.
\]
Unfortunately, this is false, as explained in the example below.  This is not completely surprising however, as the Koszul complex does not play as nicely as one might want with powers of ideals.  {However, in \cite{MaMcDonaldRGSchwede.BSForPseudorRational}, it is shown that if one uses the Buchsbaum-Eisenbud complex or equivalently the associated Eagon-Northcott complex instead of the Koszul complex {to define a closure}, then the containment does hold.  Perhaps this should be interpreted as saying that while the Koszul complex gives a derived interpretation of $R/J$, the Eagon-Northcott or Buchsbaum-Eisenbud complex gives a derived description of $R/J^k$.}

\begin{example}
    \label{exam.KosClosureNoGeneralizedBS}
    Using the Macaulay2 implementation as discussed below in \autoref{sec.ComputationsInM2} one can verify the following.  In $R = {\bQ}[x,y,z]/(x^5+y^5+z^5)$, if one sets $I = (x,y)$, a parameter ideal, then 
    \[ 
        \overline{I^3} \not\subseteq (I^2)^{\KH}.
    \]
    For the explicit computation in Macaulay2, see \autoref{rem.M2HironakaPreClosure} below.
\end{example}

\section{Computations and examples for \KosClos}
\label{sec.ComputationsInM2}

An interesting property of the \KosClos is that the hard part in computing it is computing the resolution of singularities $\pi : Y \to \Spec R$ and computing $\myR \Gamma(\cO_Y)$.  Compare this to tight, or even plus closure, which can be extremely subtle to compute (although tight closure of $0$ in local cohomology is understood thanks to \cite{KatzmanParameterTestIdeals}).  Indeed, whenever one can compute a resolution of singularities (for instance for a cone singularity) or if one knows $\Gamma(\omega_Y)$ for some regular alteration, it is not difficult to implement \KosClos in Macaulay2 \cite{M2}.  We have done that and the code is available here:
\begin{center}
    \url{https://www.math.utah.edu/~schwede/M2/KHClosure.m2}
\end{center}
and as an ancillary file in the arXiv submission.  There are two key ingredients in this implementation, the {\tt BGG} package \cite{BGGSource} to compute\footnote{If $\Gamma(\omega_Y)$ is known for a regular alteration, then {\tt BGG} is not needed as $\myR \Gamma(\cO_Y)$ can be computed via duality.} $\myR\Gamma(\cO_Y)$ and the {\tt Complexes} package \cite{ComplexesSource} to work with complexes sufficiently functorially.

\begin{example}
\label{exam.KHClosureExampleComputations}
    We show how to compute this closure in Macaulay2, when we know the resolution of singularities.
One of the most studied examples in tight closure theory is the cone over an elliptic curve and more generally diagonal hypersurfaces, see for instance \cite{McDermottTightClosurePlusClosureCubicalCones}, \cite{Singh.TightClosureDiagonalHypersurfaces} and \cite[Example 9.5]{Brenner.SlopesOfVBOnProjectiveCurvesAndApplicationsToTC}.
{
    \small
    \color{blue}
\begin{verbatim}
    i1 : loadPackage "KHClosure";
    
    i2 : A = QQ[x,y,z]; J = ideal(x^3+y^3+z^3); m = ideal(x,y,z); R = A/J;
    
    i6 : koszulHironakaClosure(ideal(x,y)*R,m) --KH closure of a parameter ideal - not closed 
    
                       2
    o6 = ideal (y, x, z )
    
    o6 : Ideal of R
    
    i7 : diagonal2 = koszulHironakaClosure(ideal(x^2,y^2,z^2)*R,m)
    
                 2   2   2
    o7 = ideal (z , y , x )
    
    o7 : Ideal of R
    
    i8 : member(x*y*z,diagonal2) -- it is in the tight closure 
    
    o8 = false
    
    i9 : diagonal3 = koszulHironakaClosure(ideal(x^3,y^3,z^3)*R,m)
    
                 3   3     3    3   2 2 2
    o9 = ideal (z , y , - y  - z , x y z )
    
    o9 : Ideal of R
    
    i10 : member(x^2*y^2*z^2, diagonal3) --in both the tight and KH closures
    
    o10 = true
    
    i11 : BrennerComparison = koszulHironakaClosure(ideal(x^4,x*y,y^2)*R,m)
    
                  2         3     3
    o11 = ideal (y , x*y, -y , x*z )
    
    o11 : Ideal of R
    
    i12 : member(y*z^2, BrennerComparison) --this is in tight closure, but not KH closure
    
    o12 = false
\end{verbatim}
}
In the above examples, the term {\color{blue}{\tt m}} is the ideal whose blowup computes the resolution of singularities.  Alternately, you can pass it a module isomorphic to $\Gamma(\omega_Y)$ for any regular alteration.  In fact, in more complicated examples that tends to be much faster as it bypasses the {\tt BGG} package, see \autoref{rem.SpeedConsiderations} and \autoref{subsubsec.ComputationsWithoutBlowups} for more discussion.

Additionally, for higher degree diagonal hypersurfaces, \KosClos does not always contain the elements in tight closure.  Compare the next example with \cite{Singh.TightClosureDiagonalHypersurfaces}.
{
    \small
    \color{blue}
\begin{verbatim}
    i2 : A = QQ[x,y,z,w]; J = ideal(x^4+y^4+z^4+w^4);  m = ideal(x,y,z,w);  R = A/J;
    
    i6 : diagonal3 = koszulHironakaClosure(ideal(x^3,y^3,z^3,w^3), m) --it's already closed
    
                 3   3   3   3
    o6 = ideal (w , z , y , x )
    
    o6 : Ideal of R
    
    i7 : member(x^2*y^2*z^2*w^2, diagonal3)
    
    o7 = false
\end{verbatim}
}

Note, in \cite[Remark 4.5]{BrennerKatzman.ArithmeticOfTightClosure} it is shown that $(x,y,z)^7 \subseteq (x^4,y^4,z^4)^*$ in $k[x,y,z]/(x^7+y^7+z^7)$ in almost all characteristics $p$.  We do not have that for \KosClos, although we do have $(x,y,z)^8$ contained in $(x^4,y^4,z^4)^{\Kos}$.
{
    \small
    \color{blue}
\begin{verbatim}
    i2 : A = QQ[x,y,z]; J = ideal(x^7+y^7+z^7);  m = ideal(x,y,z);  R = A/J;
    
    i6 : BrennerKatzmanExample45=koszulHironakaClosure(ideal(x^4,y^4,z^4), m)
    
                 4   4   4   2 3 3   3 2 3   3 3 2
    o6 = ideal (z , y , x , x y z , x y z , x y z )
    
    o6 : Ideal of R
    
    i7 : isSubset(m^7*R, BrennerKatzmanExample45)
    
    o7 = false
    
    i8 : isSubset(m^8*R, BrennerKatzmanExample45)
    
    o8 = true
\end{verbatim}
}
\end{example}

\begin{remark}
    Even in a fixed positive characteristic, for any blowup $Y \to \Spec R$, thanks to \autoref{rem.KoszulClosureInCharp}, the {\tt koszulHironakaClosure} Macaulay2 method will produce an ideal contained in the plus closure of $I$ (and hence also in the tight closure of $I$).
\end{remark}

\begin{remark}
\label{rem.SpeedConsiderations}
    The {\tt koszulHironakaClosure} function can be quite fast, frequently much faster than computing the integral closure in Macaulay2 for instance.  In higher (co)dimensions however, the typical bottleneck is computing the derived pushforward $\myR \Gamma(\cO_Y)$ using the {\tt BGG} package.  It is beyond our computers' capabilities for $Y$ a resolution of a cone over an Abelian surface in $\mathbb{P}^8$.  Instead, we recommend calling {\tt koszulHironakaClosure} with $\Gamma(Y, \omega_Y)$ instead of an ideal, as the computation is much faster, as done in the example below.
    {
    \small
    \color{blue}
    \begin{verbatim}
    i2 : A = QQ[xr,yr,zr,xs,ys,zs,xt,yt,zt];
    
    i3 : n = 3; B = (QQ[x,y,z]/(ideal(x^n+y^n+z^n)))**(QQ[r,s,t]/(ideal(r^n+s^n+t^n)));
    
    i5 : phi = map(B, A, {x*r,y*r,z*r,x*s,y*s,z*s,x*t,y*t,z*t});
    
    i6 : J = ker phi; -- make the Segre product, not Cohen-Macaulay
    
    i7 : R = A/J; -- a cone over an Abelian variety
    
    i8 : mR = ideal(xr,yr,zr,xs,ys,zs,xt,yt,zt); --maximal ideal, actually the multiplier ideal
    
    i9 : partParm = ideal(xs,yt);
    
    i10 : time trim koszulHironakaClosure(partParm, mR*R^1)
     -- used 1.6527s (cpu); 1.15284s (thread); 0s (gc)
    
    o10 = ideal (yt, xs, zs*zt, xr*yr)
    
    o10 : Ideal of R
\end{verbatim}
}
The above computation of the \KosClos of a parameter ideal is beyond the capabilities of our computers if we use the blowup strategy.  However, as it is a quasi-Gorenstein isolated log canonical singularity, we know that $\Gamma(\omega_Y) = \fm \omega_R$, and so we can use that to compute the \KosClos.
\end{remark}

\subsection{Our Macaulay2 implementation}

Suppose $S \to R$ is a surjection from a polynomial ring to $R$, $({\bf f}) \subseteq S$ maps onto $J$ in $R$, and $\pi : Y \to \Spec R$ is a resolution of singularities.  The {\tt BGG} package \cite{BGGSource} lets one compute the complex $\myR\Gamma(\cO_Y)$ as a complex of $S$-modules.  We then consider the map 
\[
    \myR\Gamma(\cO_Y) \to K_{\mydot}({\bf f}; S) \otimes^{\myL}_S \myR\Gamma(\cO_Y)
\]
which we construct via the {\tt Complexes} package \cite{ComplexesSource}.  If $R$ is normal, then the $0$th cohomology on the left is exactly $R$, otherwise it is the normalization $R^{\mathrm{N}}$.  We can then take $0$th homology of that map and consider the image $M$ of 
\begin{equation}
\label{eq.CohomologyMapForM2}
    H^0\big(\myR\Gamma(\cO_Y)\big) \to H^0\big(K_{\mydot}({\bf f}; S) \otimes^{\myL}_S \myR\Gamma(\cO_Y)\big).
\end{equation}

Note that 
\[
    K_{\mydot}({\bf f}; S) \otimes^{\myL}_S \myR\Gamma(\cO_Y) \cong K_{\mydot}({\bf f}; S) \otimes^{\myL}_S R \otimes_R^{\myL} \myR\Gamma(\cO_Y) =K_{\mydot}({\bf f}; R) \otimes_R^{\myL} \myR\Gamma(\cO_Y) 
\]
and so there is no dependence on the choice of $S$.  

If $R$ is reduced but not normal with normalization $R^{\mathrm{N}}$, then the image $M$ of the map \autoref{eq.CohomologyMapForM2}
is 
\[
    M = R^{\mathrm N}/ (({\bf f})R^{\mathrm N})^{\Kos}
\]
The $R$-annihilator\footnote{In fact, we take the $S$-annihilator in Macaulay2, but as $R$ is a quotient of $S$, this is harmless.} of that is 
\[
    \{ x \in R \;|\; x \in (J R^{\mathrm N})^{\Kos} \} = R \cap (J R^{\mathrm N})^{\Kos}
\]
which agrees with $J^{\Kos}$ thanks to \autoref{prop.FiniteExtensionComputationKHClosure}.

\subsubsection{Computations without blowups}
\label{subsubsec.ComputationsWithoutBlowups}
Instead of specifying an ideal to compute a resolution of singularities, we can specify $\Gamma(\omega_Y)$.  Then we can observe, via repeated usage of Grothendieck duality, that 
\[
    \myR \Gamma(\cO_Y) \cong \myR \Hom_R( \myR \Gamma(\omega_Y^{\mydot}), \omega_R^{\mydot}) \cong \myR \Hom_A( \myR \Gamma(\omega_Y^{\mydot}), \omega_A^{\mydot}) = \myR \Hom_A( \myR \Gamma(\omega_Y), \omega_A)[\dim A - \dim R],
\]
the last of which Macaulay2 can compute.  In our experience, this implementation is much faster than the implementation where we compute the blowup.

\subsection{Ideal powers -- \KosClos is not a semi-prime or star operation}
\label{subsec.NotASemiprimeOrStarOperation}

It is natural to ask if \KosClos is a semi-prime or star operation as described in \cite[Section 4.1]{nme-guide2}, see also \cite{Petro.SomeResultsOnTheAsymptoticCompletion,Gilmer.MultiplicativeIdealTheory}.

Recall that a closure operation is \emph{semi-prime} if we have $I\cdot J^{\cl} \subseteq (IJ)^{\cl}$.  It is a \emph{star operation} if $(xJ)^{\cl} = x(J^{\cl})$.  Using our {\tt KHClosure} package, we can easily verify that \KosClos satisfies neither property.

{    
    \color{blue}
    \small{
\begin{verbatim}
    i2 : A = QQ[x,y,z];
    
    i3 : J = ideal(x^5+y^5+z^5) ;
    
    i4 : mA = ideal(x,y,z);
    
    i5 : R = A/J;
    
    i6 : I = ideal(x,y);
    
    i7 : IKH = koszulHironakaClosure(I, mA)
    
                       2
    o7 = ideal (y, x, z )
    
    o7 : Ideal of R
    
    i8 : I2KH = koszulHironakaClosure(I^2, mA)
    
                 2        2   4     3     3
    o8 = ideal (y , x*y, x , z , y*z , x*z )
    
    o8 : Ideal of R
    
    i9 : isSubset(IKH*IKH, I2KH)
    
    o9 = false
    
    i10 : IxKH = koszulHironakaClosure(I*x,mA)
    
                       2     3   5    5
    o10 = ideal (x*y, x , x*z , y  + z )
    
    o10 : Ideal of R
    
    i11 : IxKH == x*IKH
    
    o11 = false
\end{verbatim}}
}


\section{Alternate alteration-based (pre)closures}
The \KosClos introduced above was not the first  operation we considered.  Indeed, we developed it when trying to study the colon capturing property of other operations.  We mention these  operations below.

\subsection{Hironaka preclosure and Hironaka closure}
\label{subsec.HironakaPreclosure}
We begin with what we now call the Hironaka preclosure and observe that it has a natural generalization for modules.

\begin{defn}
\label{def.Hironaka}
    Suppose $R$ satisfies \autoref{set.NiceSetup}. 
    Let $L\subseteq M$ be $R$-modules and let $\pi\colon Y\to\Spec R$ be a regular alteration. Define the \emph{Hironaka preclosure} of $L$ in $M$ to be
        \[L^{\mathrm{Hir}}_M\colon=\ker(M\to H_0    (M/L\otimes_R^{\myL} \myR \Gamma(\cO_Y))).\]    
\end{defn}

\begin{prop}
    The Hironaka preclosure of $L$ in $M$ is a well-defined preclosure operation. Furthermore, it is functorial, residual, and, if $R$ is local, is faithful.
\end{prop}
\begin{proof}
    This operation is independent of the choice of alteration by a similar argument to the one in the proof of \autoref{prop.kosIndependence}. Clearly the Hironaka preclosure is extensive, as $L\subseteq\ker (M\to H_0(M/L\otimes_R^\bL\myR \Gamma(\cO_Y)))$. Similarly, it is order preserving.    

    To see that it is functorial, consider $f\colon M\to N$ and the following commutative diagram
    \[\begin{tikzcd}
        M\arrow[r,"f"]\arrow[d]&N\arrow[d]\\
        H_0(M/L\otimes_R^\bL \myR \Gamma(\cO_Y))\arrow[r]&H_0(N/f(L)\otimes_R^\bL \myR \Gamma(\cO_Y))
    \end{tikzcd}\]
    which implies $f(L^{\mathrm{Hir}}_M)\subseteq f(L)^{\mathrm{Hir}}_N$. The Hironaka preclosure must also be residual because $M\to H_0(M/L\otimes_R^\bL \myR \Gamma(\cO_Y))$ factors through $M/L$. Finally, suppose $(R,\fm,k)$ is local. Then because $\myR \Gamma(\cO_Y)$ is locally Cohen-Macaulay the natural map $R\to \myR \Gamma(\cO_Y)\to H_0(k\otimes_R^\bL \myR \Gamma(\cO_Y))$ is nonzero and so Hironaka preclosure is faithful.
\end{proof}

{
It was shown in \cite{MaMcDonaldRGSchwede.BSForPseudorRational} that if $R$ is a ring satisfying \autoref{set.NiceSetup} (or more generally, in any characteristic if $R$ has a resolution of singularities and we use that resolution of singularities to define the preclosure), and if $J$ is $n$-generated, then we have that $\overline{J^{n+k-1}} \subseteq (J^k)^{\Hir}$.  Hence, Hironaka preclosure fixes at least one issue with Koszul-Hironaka closure.  
}

\begin{remark}
    \label{rem.M2HironakaPreClosure}

     The Macaulay2 package {\tt KHClosure} provides a command {\tt hironakaClosure} which computes the Hironaka preclosure when we know the multiplier module $\Gamma(\omega_Y) \subseteq \omega_R$ and when the ambient ring $R$ is Cohen-Macaulay.  In this case, we compute $\myR\Gamma(\cO_Y)$ over $R$ and not $A$.  Because of this, and because we take a (partial) resolution of $R/J$ instead of simply forming the Koszul complex, {\tt hironakaClosure} tends to be \emph{much} slower than {\tt koszulHironakaClosure}.
    {
        \color{blue}
        \small
        \begin{verbatim}
i2 : A = QQ[x,y,z];  J = ideal(x^5+y^5+z^5);  mA = ideal(x,y,z);

i5 : R = A/J; mR = sub(mA, R);

i7 : I = ideal(x,y); -- a parameter ideal

i8 : I2KH = koszulHironakaClosure(I^2,mA)

             2        2   4     3     3
o8 = ideal (y , x*y, x , z , y*z , x*z )

o8 : Ideal of R

i10 : I2Hir = hironakaClosure(I^2,mR^3*R^1) --needs multiplier ideal/module of R in this case

              2        2   3     2     2
o10 = ideal (y , x*y, x , z , y*z , x*z )

o10 : Ideal of R

i12 : isSubset(I2KH, I2Hir)

o12 = true

i13 : member(x*z^2, I2KH) -- KH closure is strictly smaller than Hironaka

o13 = false
        \end{verbatim}    
    }    
\end{remark}

\begin{example}
    \label{exam.HironakaPreclosureNotIdempotent}
    Hironaka preclosure is not idempotent.  The following example was identified by ChatGPT 5.6 Sol theoretically, and we provide the verification with Macaulay2 here.  In $R = \bQ[x,y,z]/(x^4+y^4+z^4)$ the ideal $J = (x^3+y^3+z^3,xz^2,yz^2)$ has the property that $J^{\Hir} \neq (J^{\Hir})^{\Hir}$.  
    {\small
        \color{blue}
        \begin{verbatim}
i2 : R = QQ[x,y,z]/ideal(x^4+y^4+z^4); m = ideal gens R;

i4 : J = ideal(x^3+y^3+z^3,x*z^2,y*z^2)

             3    3    3     2     2
o4 = ideal (x  + y  + z , x*z , y*z )

o4 : Ideal of R

i5 : JHir = hironakaClosure(J, m^2*R^1) --m^2 is multiplier ideal                                     

               2     2   3    3    3   4   5
o5 = ideal (y*z , x*z , x  + y  + z , z , y z)

i6 : JHirHir = hironakaClosure(JHir, m^2*R^1)

               2     2   3    3    3   4   4    2 2    5
o6 = ideal (y*z , x*z , x  + y  + z , z , y z, x y z, y )

i7 : isSubset(JHirHir,JHir)

o7 = false
        \end{verbatim}
    }
\end{example}

In view of the above example, it is reasonable to take the idempotent hull of Hironaka closure (see \cite[Construction 3.1.5]{nme-guide2} or \cite[Section 4.6]{DikTho-closure}).  We note that $(-)^{\Hir}_M$ is idempotent on submodules $L \subseteq M$ such that $M/L$ satisfies a characteristic-free analog of having finite phantom projective dimension, see \autoref{cor.HirIdempotentForPPD}.

\begin{definition}
\label{def.HironakaIterated}
    Suppose $R$ satisfies \autoref{set.NiceSetup} and that $L \subseteq M$ are $R$-modules.  We define the $\text{Hironaka}^\infty$ closure of $L$ in $M$ to be:
    \[
        L_M^{\Hir^{\infty}} = \bigcup_{n = 1}^{\infty} L^{\Hir^n}_M
    \]
    where $(-)^{\Hir^n}$ is the result of iteratively applying $\Hir$ $n$-times.  Explicitly, $L_M^{\Hir^1} = L_M^{\Hir}$ and $L_M^{\Hir^{n+1}} = (L_M^{\Hir^n})^{\Hir}_M$ for all $n > 0$.

    Note that if $M$ is Noetherian, then $L_M^{\Hir^{\infty}} = L_M^{\Hir^n}$ for some $n \gg 0$.  
\end{definition}

Hironaka preclosure and Hironaka$^{\infty}$ closure behave better than Koszul-Hironaka closure with respect to products of ideals and modules.

{
\begin{prop}
\label{prop.HirClosureIdealProducts}
    For any ideal $I$ of $R$ and any $R$-submodule inclusion $L \subseteq M$, we have $I\cdot L^{\Hir}_M \subseteq (IL)^{\Hir}_M$.  In particular, $\Hir$ is a \emph{semi-prime operation} (see \cite[Lemma 4.1.2(1)]{nme-guide2} and its proof, noting that idempotence was not necessary), in the sense that for any $x\in R$ and ideal $J$, we have $x \cdot J^{\Hir} \subseteq (xJ)^{\Hir}$.  As a consequence, for any ideals $I, J \subseteq R$
    \[
        I^{\Hir^\infty} J^{\Hir^\infty} \subseteq (IJ)^{\Hir^\infty}.
    \]
\end{prop}

\begin{proof}
    Let $P_{\mydot}$ be a complex of projective modules (with homological grading) that is quasi-isomorphic to $\myR \Gamma(\cO_Y)$. Since the latter is an $R$-algebra, there is some $\eta \in P_0$ that represents the multiplicative identity; then for $z\in M$, we have $z \in L^{\Hir}_M$ if and only if the cohomology class of $\bar z \otimes \eta$ is $0$ in $H_0(M/L \otimes_R P)$.  Note that by right-exactness of tensor product, we have for any $R$-module $N$ that $\img d_1^{N \otimes_R P} = \img (N \otimes_R \img d_1^P \rightarrow N \otimes_R P_0)$ induced by the inclusion $\img d_1^P \hookrightarrow P_0$.  Thus, we have \[
    L^{\Hir}_M = \{z \in M \mid z \otimes \eta \in \img(M \otimes_R \img d_1^P) + \img (L \otimes_R P_0)\},
    \]
    where both images are in $M \otimes_R P_0$, induced by the inclusions $L \hookrightarrow M$ and $\img d_1^P \hookrightarrow P_0$.
    
    Now, let $z \in L^{\Hir}_M$ and $a\in I$.  Then $z \otimes \eta = \sum_{i=1}^n m_i \otimes d_1^P(p_i) + \sum_{j=1}^t \ell_j \otimes q_j$ in $M \otimes_R P_0$, where $m_i \in M$, $\ell_j \in L$, $p_i \in P_1$, and $q_j \in P_0$.  But then $az \otimes \eta = \sum_{i=1}^n am_i \otimes d_1^P(p_i) + \sum_{j=1}^t a\ell_j \otimes q_j$.  Since $a\ell_j \in L$ for all $j$, it follows from the displayed equation above that $az \in (IL)^{\Hir}_M$.

    For the final statement, it follows from an immediate induction argument that $I \cdot J^{\Hir^n} \subseteq (IJ)^{\Hir^n}$ for all $n>0$, and hence that $I \cdot J^{\Hir^\infty} \subseteq (IJ)^{\Hir^\infty}$. Therefore we have the chain of containments
    \[
        I^{\Hir^{\infty}} J^{\Hir^{\infty}} \subseteq (I^{\Hir^{\infty}} J)^{\Hir^{\infty}} = (J I^{\Hir^{\infty}})^{\Hir^{\infty}} \subseteq ((JI)^{\Hir^{\infty}})^{\Hir^{\infty}} = (JI)^{\Hir^{\infty}} = (IJ)^{\Hir^{\infty}}
    \]
    which proves the proposition.
\end{proof}
}

If $M = R$ and $L = I = (f_1, \dots, f_n)$ is an ideal, then it is easy to see that 
\begin{equation}
    \label{eq.KosInHironaka}
I^{\Kos} \subseteq I^{\mathrm{Hir}}
\end{equation}  
since we have a factorization $R \to K_{\mydot}({\bf f}; R)\otimes^{\myL} \myR\Gamma(\cO_Y) \to R/I \otimes^{\myL} \myR\Gamma(\cO_Y)$.
Furthermore, if $R$ is Cohen-Macaulay and $I = (f_1, \dots, f_n)$ is a parameter ideal then $I^{\Kos} =I^{\mathrm{Hir}}$ (as then $R/I \cong K_{\mydot}({\bf f}, R)$ in the derived category).  In fact, we will see that $I^{\KH} = I^{\Hir}$ for parameter ideals even without the Cohen-Macaulay condition in  \autoref{cor.ClosuresCoincideInChar0ForParameterIdeals}. 

\begin{remark}
The proof of \autoref{prop.KHClosureInIntegralClosure} runs without change and so we obtain that $I^{\mathrm{Hir}} \subseteq \overline{I}$.  
\end{remark}
We also obtain the following comparison with tight closure via reduction modulo $p \gg 0$. 

 \begin{lemma}
 \label{lem.HirIsInPlus}
     Suppose $R$ is a reduced ring of finite type over a field of characteristic zero with finitely generated $R$-modules $L\subseteq M$.  Then $L_M^{\mathrm{Hir}} \subseteq L_M^*$ where the right side denotes reduction modulo $p \gg 0$ tight closure.  In fact, when reducing modulo $p \gg 0$ as in \autoref{prop.KHContainedInPlusModP}, for an open dense subset $U \subseteq \mathfrak{m}\text{-}\Spec A$, $(L^{\mathrm{Hir}}_M)_t$ is contained in $(L_t)^+_{M_t}$. 
 \end{lemma}
 \begin{proof}
      First, note that in order to compute $H_0(M/L\otimes_R^\bL\myR\Gamma(\cO_Y))$ we can truncate a projective resolution of $M/L$ in degree $d+1$ where $d=\dim(R)$. To see this, note that $\myR\Gamma(\cO_Y)$ is equivalent to a bounded complex concentrated in cohomological degree $[0,d]$. If $P\to M/L$ is a projective resolution, then $M/L\otimes_R^\bL\myR\Gamma(\cO_Y)\cong P\otimes_R\myR\Gamma(\cO_Y)$. To compute $H_0(M/L\otimes_R^\bL\myR\Gamma(\cO_Y))$, it suffices to consider 
     \[(P\otimes_R\myR\Gamma(\cO_Y))_1\to (P\otimes_R\myR\Gamma(\cO_Y))_0\to(P\otimes_R\myR\Gamma(\cO_Y))_{-1}\]
     which is equivalent to
     \[\bigoplus_{i=1}^{d+1}P_i\otimes_R\myR\Gamma(\cO_Y)^{i-1}\to \bigoplus_{i=0}^{d}P_i\otimes_R\myR\Gamma(\cO_Y)^i\to\bigoplus_{i=-1}^{d-1}P_i\otimes_R\myR\Gamma(\cO_Y)^{i+1}.\] 
     This implies that $H_0(M/L\otimes_R^\bL\myR\Gamma(\cO_Y))=H_0(P_{\leq d+1}\otimes_R\myR\Gamma(\cO_Y))$.

     We now modify the argument of \autoref{prop.KHContainedInPlusModP}. As before, let $M_t$, $L_t$, and $(L_M^{\Hir})_t$, for $t\in\fm\text{-}\Spec A$ denote a family of reduction-to-characteristic $p>0$ models of $M$, $L$, and $L_M^{\Hir}$ respectively. By enlarging $A$ if necessary, we can assume that a resolution of singularities $\pi\colon Y\to\Spec R$ reduces to a resolution $\pi_t\colon Y_t\to\Spec R_t$ as well. The fact that $H_0(M/L\otimes_R^\bL\myR\Gamma(\cO_Y))$ can be computed with bounded complexes ensures that we can also enlarge $A$ so that $(L_M^{\Hir})_t=\ker\big(M_t \to H_0(M_t/L_t\otimes_{R_t}^\bL \myR\Gamma(\cO_{Y_t}))\big)$.  The proof of \autoref{prop.KHContainedInPlusModP} now works essentially unchanged as we have a factorization $R_t \to \myR\Gamma(\cO_{Y_t}) \to R_t^+$.
 \end{proof}

   Many of the elements from the reduction-mod-$p$ tight closure that we saw were not in the \KosClos in \autoref{exam.KHClosureExampleComputations} can easily be seen to be in the Hironaka closure using the Macaulay2 package.  In particular, the {\tt hironakaClosure} function gives another way to check that various elements are in the tight closure for all $p \gg 0$ thanks to \autoref{lem.HirIsInPlus} above.

%
%
%
%
%
%

\subsection{Canonical alteration closure}

    Our next closure operation is inspired by \cite{Mcd-multklt} in the form of \autoref{thm.Mcd-multklt}.  Indeed, as the collection $\Gamma( \omega_Y)$ for various alterations $Y \to \Spec R$ computes the multiplier ideal $\mJ(R)$ in the sense of de Fernex-Hacon \cite{DeFernexHaconSingsOnNormal}, it is natural to simply use that family of modules to construct a closure whose test ideal is the multiplier ideal.  We expect this operation to be closer to tight closure in equal characteristic zero, see also \autoref{thm.TightClosureEqualsModuleClosureParameterTestModules} below.


\begin{defn}
\label{def.CanonicalAlterationClosure}
    Suppose $R$ is a normal domain satisfying \autoref{set.NiceSetup} and 
    let $L\subseteq M$ be $R$-modules and $\pi:Y \to \Spec(R)$ be a regular alteration. We set $\cl_\pi$ to be the closure operation
    \[L_M^{\cl_\pi}=L_M^{\cl_{\canon}},\]
    where we are viewing $\canon$ as an $R$-module and this closure as the module closure it defines in the sense of \autoref{def.ModuleClosure}. We define the \emph{canonical alteration closure} of $L$ in $M$, denoted $L_M^\alt$, to be the  operation
    \[L_M^\alt\colon=\bigcap_{\pitoR} L_M^{\cl_\pi}.\]
    In other words,
    \begin{align*}L_M^\alt&=\bigcap_{\pitoR} \{z \in M \mid s \otimes z \in \im(\canon \otimes_R L \to \canon \otimes_R M) ~\forall~ s \in \canon\}\\
    &=\bigcap_{\pitoR} \{z \in M \mid s \otimes z \in \ker(\canon \otimes_R M \to \canon \otimes_R M/L) ~\forall~ s \in \canon\}.
    \end{align*}

    Recall (see \cite[Proposition 6.2]{EpsteinRGVassilev.ACommonFrameworkForTestIdealsClosure}) that the \emph{meet} $\bigwedge_{\alpha \in \Lambda} \cl_\alpha$ of a collection $\{\cl_\alpha\}_{\alpha \in \Lambda}$ of closure operations is defined for every nested pair of $R$-modules $L \subseteq M$ for which the $\cl_\alpha$ are defined, via $L^{(\bigwedge_\alpha \cl_\alpha)}_M := \bigcap_\alpha L^{\cl_\alpha}_M$.  Hence, $\alt = \displaystyle \bigwedge_{\pi: Y \ra \Spec(R)} \cl_\pi$.
\end{defn}

The choice of the word ``canonical'' in \emph{canonical alteration closure} is meant to emphasize the role of the canonical module.  

\begin{prop}\label{pr:altisaclosure}
    The operation $\alt$ is a well-defined closure operation. Further, $\alt$ is residual, functorial, and, if $R$ is local, faithful.
\end{prop}

\begin{proof}
    The fact that module closures are functorial and residual is contained in \cite[Proposition 7.4]{EpsteinRGVassilev.ACommonFrameworkForTestIdealsClosure}.
    As the modules we are using are finitely generated, the individual $\cl_{\pi}$ operations are faithful for local rings by Nakayama's Lemma.  By \cite[Proposition 6.4]{EpsteinRGVassilev.ACommonFrameworkForTestIdealsClosure}, the meet of residual and functorial closures is a residual and functorial closure operation.  By \cite[Theorem 4.3]{dietzclosureprops}, when $R$ is local, the meet of faithful closure operations is faithful. This finishes the proof. \qedhere
        \end{proof}

Another useful fact is that, at least on Artinian modules, $\alt$ is a module closure (though of what module depends on the pair $L\subseteq M$).

\begin{lemma}
\label{lem:usejustonealteration}
Given $L \subseteq M$ Artinian, there is some $\pitoR$ a regular alteration such that $L_M^\alt=L_M^{\cl_\pi}.$
\end{lemma}

\begin{proof}
    This follows from the fact that any two regular alterations, $\pi_i : X_i \to \Spec R$ for $i=1,2$, can be dominated by a third alteration $\pi: Y \to \Spec R$.  
    \[
        \xymatrix{
            Y \ar[d] \ar[dr]^{\pi} \ar[r] & X_1 \ar[d]^-{\pi_1} \\
            X_2 \ar[r]_-{\pi_2} & \Spec R
        }
    \]

We then claim that $L_M^{\cl_\tau}\subseteq L_M^{\cl_{\pi_1}}\cap L_M^{\cl_{\pi_2}}$. To see this, we note that both $X_1$ and $X_2$ are smooth and thus are derived splinters in the sense of \cite{BhattDerivedDirectSummandCharP} and so $\Gamma(\omega_{X_1})$ and $\Gamma(\omega_{X_2})$ are both summands of $\Gamma(\omega_Y)$. Thus, if $z\in M$ is such that $z\otimes c$ is in the image of $L\otimes_R\Gamma(\omega_Y)\to M\otimes_R\Gamma(\omega_Y)$ for all $c\in\Gamma(\omega_Y)$ we see that this must also hold for all $c\in\Gamma(\omega_{X_1})$ and also for all $c\in\Gamma(\omega_{X_2})$ and we get the containment claimed. This proves the result, as $M$ is Artinian.
\end{proof}

\begin{prop}\label{contain}
    For a domain $R$ satisfying \autoref{set.NiceSetup}, let $I$ be an ideal of $R$. Then $I^{\Kos}\subseteq I^{\mathrm{Hir}}\subseteq I^{\alt}$. If $L\subseteq M$ are $R$-modules then $L_M^{\mathrm{Hir}}\subseteq L_M^\alt$.
\end{prop}
\begin{proof}
    We already observed that $I^{\KH} \subseteq I^{\Hir}$ in equation \autoref{eq.KosInHironaka}.

    Now consider the pair of $R$-modules $L\subseteq M$ and let $\pi:Y\to\Spec R$ be a regular alteration. For each $z\in\Gamma(\omega_Y)$, consider the map $\psi_z\colon \myR \Gamma(\cO_Y)\to\myR \Gamma(\omega_Y)$ induced by $\cO_Y \xrightarrow{1 \mapsto z} \omega_Y$.
    Then consider
    \[\begin{tikzcd}
        M\arrow[r]&H_0(M/L\otimes_R^\bL \myR \Gamma(\cO_Y))\arrow[rr,"M/L\otimes\psi_z"]&& M/L\otimes_R \Gamma(\omega_Y)
    \end{tikzcd}
    \]
    which sends $m\mapsto\bar{m}\otimes z$. Thus we see that $L^{\mathrm{Hir}}_M=\ker(M\to H_0(M/L\otimes_R^\bL \myR \Gamma(\cO_Y))\subseteq\bigcap_{z\in\Gamma(\omega_Y)}\ker(M\to M/L\otimes_R \Gamma(\omega_Y))=L^{\cl_\pi}_M$. Thus, $L^{\mathrm{Hir}}_M\subseteq L^\alt_M$ and we are done.
\end{proof}

\begin{remark}
    As \KosClos satisfies colon capturing, and strong colon capturing version A, so do the Hironaka and canonical alteration closures.  
Indeed, by \autoref{thm.cc} and \autoref{contain}, we obtain colon-capturing:
\[
(x_1,\dots,x_k):x_{k+1}\subseteq(x_1,\dots,x_k)^{\Kos}\subseteq(x_1,\dots,x_k)^{\mathrm{Hir}}\subseteq(x_1,\dots,x_k)^\alt.
\]
The proof that they satisfy strong colon-capturing version A is similar.  We also immediately see that they satisfy a version of the Brian\c{c}on-Skoda property as we have
\[
    \overline{J^m} \subseteq J^{\Kos} \subseteq J^{\Hir} \subseteq J^{\alt}
\]
if $J$ is $m$-generated.
\end{remark} 

We prove the next result by the method of \cite[Theorem 3.12]{perezrg}.

\begin{thm}
\label{thm:alttestidealisJ}
Let $R$ be normal, local, and satisfying \autoref{set.NiceSetup}. Then the multiplier ideal of de Fernex-Hacon satisfies $\mJ(R)=\tau_{\alt}(R)$.  As a consequence, $R$ has KLT-type if and only if all modules are $\alt$-closed.
\end{thm}

\begin{proof}
    First, we know by Theorem \ref{thm:testistrace} that $\tau_\alt(R)=\ann_R 0_E^\alt$. So it suffices to show that
    \[\ann_R 0_E^\alt=\sum_{\pitoR} \tr_{\canon}(R).\]

    First we discuss alternate ways to write $0_E^\alt$. It follows from the definition that
    \[0_E^\alt= \bigcap_{\pitoR} \bigcap_{s \in \canon} \ker(E \to \canon \otimes_R E), \]
    where the maps $E \to \canon \otimes_R E$ send $z \mapsto s \otimes z$.
    Since $E$ is Artinian, there must exist finitely many regular alterations $\pitoR$, say $\pi_1,\ldots,\pi_t$ and for each $i$, finitely many $s \in \Gamma(\omega_{Y_i})$, say $s_{i1},\ldots,s_{ik_i}$ such that
    \[0_E^\alt=\bigcap_{i=1}^t \bigcap_{j=1}^{k_i} \ker(E \to \Gamma(\omega_{Y_i}) \otimes_R E),  \]
    where the $i,j$th map sends $z \mapsto s_{ij} \otimes z$.
    Define 
    \[\phi:E \to \bigoplus_i \bigoplus_j \Gamma(\omega_{Y_i}) \otimes_R E\]
    to be the map sending
    \[z \mapsto \bigoplus_i \bigoplus_j s_{ij} \otimes z.\]
    Then by the above, $0_E^\alt=\ker \phi$.

    Now we prove that $\ann_R 0_E^\alt=\sum_{\pitoR} \tr_{\canon}(R)$. First, let $c \in \ann_R 0_E^\alt$, so that 
    \[0=c 0_E^\alt=c \ker \phi.\]
    Then $0_E^\alt \subseteq \ann_E c$.
    Taking the Matlis dual of this inclusion, we get a surjection
    \[\hat{R}/c\hat{R}=\Hom_{\hat{R}}(\ann_E c,E) \twoheadrightarrow \Hom_{\hat{R}}(0_E^\alt,E).\]
    We also have an exact sequence
    \[0 \to \ker \phi \to E \xrightarrow{\phi} \bigoplus_j \bigoplus_i \Gamma(\omega_{Y_i}) \otimes_R E.\]
    Taking the Matlis dual of this exact sequence, we get
    \[\Hom_{\hat{R}}(0_E^\alt,E)=\Hom_{\hat{R}}(\ker \phi,E)={\coker} \phi^\vee.\]
    We rewrite the latter as 
    \[\frac{\hat{R}}{\sum_i \sum_j \im\left(\Hom_{\hat{R}}(\Gamma(\omega_{Y_i}) \otimes_R \hat{R},\hat{R}) \to \hat{R} \right) },\]
    where the $i,j$th map sends $\psi \mapsto \psi(s_{ij})$.
    As we have a surjection 
    \[\hat{R}/c\hat{R} \twoheadrightarrow \Hom_{\hat{R}}(0_E^\alt,E),\]
    this implies that
    \[c\hat{R} \subseteq \sum_i \sum_j \im\left(\Hom_{\hat{R}}(\Gamma(\omega_{Y_i}) \otimes_R \hat{R},\hat{R}) \to \hat{R} \right).\]
    Since the $\Gamma(\omega_{Y_i})$ are \fg, Hom here commutes with flat base change, so the above is equal to
    \[\left(\sum_i \sum_j \im\left(\Hom_{R}(\Gamma(\omega_{Y_i}),R) \to R\right)\right) \otimes_R \hat{R}.\]
    Since completion is faithfully flat (and hence pure), this implies that
    \[c \in \sum_i \sum_j \im\left(\Hom_R(\Gamma(\omega_{Y_i}),R) \to R \right),\]
    which is certainly contained in
    \[\sum_\pi \tr_{\canon}(R).\]

    For the other inclusion, suppose $c \in \sum_{\pitoR} \tr_{\canon}(R)$. Then there exist regular alterations $\pi_1',\ldots,\pi_{\ell}'$ and $s'_{i1},\ldots,s'_{im_i}$ such that
    \[c \in \sum_i \sum_j \im(\Hom_R(\Gamma(\omega_{Y'_i}),R) \to R),\]
    where the $i,j$-th map sends $\psi \mapsto \psi(s'_{ij})$. Enlarge the sets of $s_{ij},\pi_i$ from the first part of the proof to include these. Then
    \[c \in \sum_i \sum_j \im(\Hom_R(\Gamma(\omega_{Y_i}),R) \to R).\]
    This implies that there is a surjection
    \[R/cR \twoheadrightarrow \frac{R}{\sum_i \sum_j \im(\Hom_R(\Gamma(\omega_{Y_i}),R) \to R)}.\]
    Taking Matlis duals, we get
    \[\Hom_R\left(\frac{R}{\sum_i \sum_j \im(\Hom_R(\Gamma(\omega_{Y_i}),R) \to R)},E \right) \hookrightarrow \Hom_R(R/cR,E)=\ann_E c.\]
    Further, the left hand side is equal to $0_E^\alt$. Hence $c \in \ann_R 0_E^\alt$.

    For the final statement, note that $R$ has KLT-type if and only if $\mJ(R) = R$.
\end{proof}

We list some further properties of the multiplier ideal/submodule for a particular alteration:

\begin{lemma}
    For a given alteration $\pitoR$, the following holds:
    \begin{align*}
        \mJ_\pi(R) &=\tr_{\canon}(R)=\tau_{\cl_\pi}(R).
    \end{align*}
\end{lemma}

\begin{proof}
    Definition \ref{def:multsubmod} gives us the first equality. The other equality follows from Theorem \ref{thm:testistrace}. 
\end{proof}

\begin{cor}
    If $\tau : Z \to \Spec R$ is a regular alteration further along the inverse limit system than $\pi : Y \to \Spec R$, then $\mJ_\pi(M) \subseteq \mJ_\tau(M)$ for any $R$-module $M$.
\end{cor}

\begin{proof}
    Let $\kappa:Z\to Y$ be such that $\tau=\pi\circ\kappa$. Since $Y$ is smooth and thus a derived splinter by \cite[Theorem 2.12]{BhattDerivedDirectSummandCharP}, we get that $\omega_Y\to\kappa_*\omega_Z$ splits. Taking global sections implies we have a surjection $\Gamma(\omega_Z)\twoheadrightarrow\Gamma(\omega_Y)$. From \cite[Proposition 2.8.1]{lindo} or \cite[Proposition 2.27.5]{perezrg}, we see that $\tr_{\canon}(M) \subseteq \tr_{\Gamma(\omega_Z)}(M)$ for any $R$-module $M$.
\end{proof}

{%
Using the fact that $(H^{-i} \omega_R^{\mydot})^{\vee} \cong H^i_{\m}(R)$, we obtain a pairing $\omega_R \times H^d_{\m}(R) \to E$, see for instance \cite[Theorem 6.7]{HartshorneLocalCohomology} or \cite[Section 2]{SmithFRatImpliesRat}.  This pairing is perfect if $R$ is complete.  Via this pairing, we can define $\Ann_{\omega_R} N$ for any submodule $N \subseteq H^d_{\m}(R)$ as in \cite[Section 2]{SmithFRatImpliesRat}, see also \cite{jiangrg}.
\begin{theorem}
    \label{thm.JOmegaCanonicalAltDual}
    Suppose $R$ is an excellent domain with a dualizing complex.  Suppose $\m \subseteq R$ is maximal and $M = H^{d}_{\m}(R)$ where $d = \dim R_{\m}$.  Then 
    \[
        (\mJ(\omega_{R})_{\m})^{\vee} \cong M \big/ 0_M^{\alt}
    \]
    where $\vee$ denotes Matlis duality.  In particular, $\Ann_{\omega_R} 0_M^{\alt} = \mJ(\omega_R)$.  Furthermore, $0_M^{\alt} = 0_M^{\cl_\pi}$ for any regular alteration $\pi : Y \to \Spec R$.  
\end{theorem}
\begin{proof}
    Without loss of generality we may assume that $R$ is local.

    For any regular alteration $\pi : Y \to \Spec R$, set $S$ to be the ring of the associated Stein factorization so that $R \subseteq S$ is finite and $S$ is normal. 

    We have that $\Hom_R(\omega_S, \omega_R) \cong S$ since $S$ is normal and hence S2 (see for instance \cite[Section 1]{HartshonreGeneralizedDivisorsAndBiliaison} and observe that $\Hom_R(\Hom_R(S, \omega_R), \omega_R)$ is the S2-ification of $S$ and that $\Hom_R(S, \omega_R) \cong \omega_S$). Note that $\Hom_R(\omega_S, \omega_R)$ is generated as an $S$-module by the trace map $\mathrm{Tr} : \omega_S \to \omega_R$.  Furthermore, as $S$ is normal, $\Gamma(\omega_Y) = \mJ(\omega_S) \subseteq \omega_S$ agrees with $S$ in codimension 1, and so $\Hom_R(\Gamma(\omega_Y), \omega_R) \cong S$ as well.  Now $\mJ(\omega_R) = \mathrm{Tr}(\Gamma(\omega_Y))$ for any regular alteration $Y$ by \cite[Theorem 8.1]{BlickleSchwedeTuckerTestAlterations} (set $\Delta_X = -K_X$ which forces $\Delta_Y = -K_Y$).  It immediately follows that $\mJ(\omega_R) = \sum_{\phi} \Image(\Gamma(\omega_Y) \xrightarrow{\phi} \omega_R)$ as $\mathrm{Tr}$ is in that sum, and all other maps $\phi \in \Hom_R(\omega_S, \omega_R)$ are multiples of trace. 

    We see that 
    \[
        \begin{array}{rcl}
            \mJ(\omega_R) & = & \sum_{\phi} \Image\big(\Gamma(\omega_Y) \xrightarrow{\phi} \omega_R\big)\\
            & = & \Image\big( \Hom_R(\Gamma(\omega_Y), \omega_R) \otimes_R \Gamma(\omega_Y) \xrightarrow{\phi \otimes y \mapsto \phi(y)} \omega_R \big)\\
            & = & \sum_{y} \Image\big( \Hom_R(\Gamma(\omega_Y), \omega_R) \xrightarrow{\phi \mapsto \phi(y)} \omega_R\big)\\
            & = & \Image\big( \bigoplus_{y} \Hom_R(\Gamma(\omega_Y), \omega_R) \xrightarrow{\phi \mapsto \phi(y)} \omega_R\big)
        \end{array}
    \]
    where $\phi$ in the first sum runs over elements of $\Hom_R(\Gamma(\omega_Y), \omega_R)$ and the $y$ in the later sums runs over $y \in \Gamma(\omega_Y)$.  

    As $\omega_R$ is the bottom cohomology of $\omega_{R}^{\mydot}$, we see that 
    \[ 
        \Hom_R(\Gamma(\omega_Y), \omega_R)^{\vee} \cong (H^{-d} \myR \Hom_R(\Gamma(\omega_Y), \omega_R^{\mydot}))^{\vee} \cong H^d_{\m}(\Gamma(\omega_Y))
    \]
    where the second isomorphism is local duality.
    Therefore, we see that 
    \[
        \begin{array}{rcl}
        \mJ(\omega_R)^{\vee} & = & \Image\big( \bigoplus_{y} \Hom_R(\Gamma(\omega_Y), \omega_R)  \to \omega_R \big)^{\vee} \\
        & \cong & \Image\Big( H^d_{\m}(R) \to \prod_{y} H^d_{\m}(\Gamma(\omega_Y)) \Big).
        \end{array}
    \]
    We know $\mJ(\omega_R)^{\vee}$ independent of the choice of $Y$ and so the kernel of the map is as well.  That is $\ker(H^d_{\m}(R) \to \prod_y H^d_{\m}(\Gamma(\omega_Y)))$ is independent of $Y$, and so also coincides with 
    \[ 
        \ker\big(H^d_{\m}(R) \to \prod_Y \prod_y H^d_{\m}(\Gamma(\omega_Y))\big)
    \] 
    where the outer product runs over regular alterations $Y \to \Spec R$.

    Now, as $H^d_{\m}(\Gamma(\omega_Y)) \cong H^d_{\m}(R) \otimes_R \Gamma(\omega_Y)$, we see that $0_M^{\alt}$ is simply 
    \[
        \ker \Big( H^d_{\m}(R) \to \prod_{Y} \prod_{y} H^d_{\m}(R) \otimes_R \Gamma(\omega_Y) \Big) = \ker \Big( H^d_{\m}(R) \to \prod_{Y} \prod_{y} H^d_{\m}(\Gamma(\omega_Y)) \Big) = \ker \Big( H^d_{\m}(R) \to \prod_{y} H^d_{\m}(\Gamma(\omega_Y)) \Big)
    \]
    where $Y$ runs over alterations and $y$ over elements of $\Gamma(\omega_Y)$.
    This proves the first statement and the fact that $0_M^{\alt}  = 0_M^{\cl_\pi}$.

    For the statement about annihilators, we observe our work so far implies that $(\omega_R / \mJ(\omega_R))^{\vee} = 0_M^{\alt}$.  
    Then by \cite[2.1 Lemma (i)]{SmithTestIdealsInLocalRings} (we do not need Cohen-Macaulay or S2 for this statement) or \cite[Proposition 4.2]{jiangrg} (see also \cite[Remark 5.3]{Hara.GeometricInterpretationOfTightClosureAndTestIdeals}), we have that $(0_M^{\alt})^{\vee} = \omega_{\widehat{R}} / \Ann_{\omega_{\widehat{R}}} 0_M^{\alt}$.  Hence, as we have a natural isomorphism $N^{\vee\vee} \cong N \otimes_R \widehat{R}$ for $N$ finitely generated, we see that
    \[
        (\omega_R / \mJ(\omega_R))^{\vee\vee} \cong \omega_{\widehat{R}} / \Ann_{\omega_{\widehat{R}}} 0_M^{\alt} \;\;\; \text{ and so } \;\;\; \mJ(\omega_R) \otimes_R \widehat{R} = \Ann_{\omega_{\widehat{R}}} 0_M^{\alt}
    \]
    where we view $\mJ(\omega_R) \otimes_R \widehat{R} \subseteq \omega_{R} \otimes_R \widehat{R} = \omega_{\widehat{R}}$.  It immediately follows that $\mJ(\omega_R) \subset \Ann_{\omega_R} 0_M^{\alt}$.  But $R \to \widehat{R}$ is faithfully flat and hence pure so $\ker (\omega_R \to {\omega_{R} \otimes_R \widehat{R} \over \mJ(\omega_R) \otimes_R \widehat{R}}) = \mJ(\omega_R)$ and therefore $\Ann_{\omega_R} 0^{\alt}_M = \mJ(\omega_R)$ as desired.
\end{proof}

\begin{corollary}
    \label{cor.ClosuresAreAllEqualForFullParameterIdeals}
    Suppose $(R, \m)$ is a $d$-dimensional Noetherian local domain satisfying \autoref{set.NiceSetup} and $J = (x_1, \dots, x_d) \subseteq \m$ is a full parameter ideal.  Then $J^{\Kos} = J^{\Hir} = J^{\alt} = J^{\cl_\pi}$ where $\pi : Y \to \Spec R$ is any regular alteration.
\end{corollary}
\begin{proof}
    The containments $J^{\Kos}  \subseteq J^{\Hir} \subseteq  J^{\alt}\subseteq  J^{\cl_\pi}$ are found in \autoref{contain} or follow by definition, and so it suffices to show that $J^{\cl_{\pi}} \subseteq J^{\Kos}$.  From the statement and proof of \autoref{thm.JOmegaCanonicalAltDual}, we have that 
    \[
        H^d_{\m}(R)/0^{\cl_{\pi}}_{H^d_{\m}(R)} = \mJ(\omega_R)^{\vee} = \Image( H^d_{\m}(R) \to \prod_y H^d_{\m}(\Gamma(\omega_Y))
    \] 
    where $\pi : Y \to \Spec R$ is a regular alteration and $y$ runs over elements of $\Gamma(\omega_Y)$.  However, we also have that 
    \[ 
        \mJ(\omega_R)^{\vee} = \Big(\Image\big( \Gamma(\omega_Y) \to \omega_R \big)\Big)^{\vee} = \Image\big( H^d_{\m}(R) \to H^d_{\m}(\myR \Gamma(\cO_Y)) \big).
    \]
    Thus $0^{\cl_{\pi}}_{H^d_{\m}(R)} = \ker\big( H^d_{\m}(R) \to H^d_{\m}(\myR \Gamma(\cO_Y)) \big)$.
    
    Now, consider the map $R \to R/J \xrightarrow{1 \mapsto [1/(x_1 \cdots x_d)]} H^d_{\m}(R)$.  As module closures are residual and functorial, we know we have a map
    \[ 
        J^{\cl_\pi}/J = 0^{\cl_{\pi}}_{R/J} \to 0^{\cl_\pi}_{H^d_{\m}(R)} = \ker\big( H^d_{\m}(R) \to H^d_{\m}(\myR \Gamma(\cO_Y)) \big).
    \]
    But by \autoref{prop.KHClosureOfFullParameterIdeal}, we have that the kernel of $R \to H^d_{\m}(\myR \Gamma(\cO_Y))$ is $J^{\Kos}$ so that $J^{\cl_{\pi}} \subseteq J^{\Kos}$ as desired.
\end{proof}
}

{%
In fact, this will also imply the result for partial parameter ideals.  

\begin{corollary}
\label{cor.ClosuresCoincideInChar0ForParameterIdeals}
    Suppose $(R, \m)$ is local and satisfies \autoref{set.NiceSetup}, $x_1, \dots, x_d$ is a full system of parameters and $J = (x_1, \dots, x_i)$ for some $1 \leq i \leq d$.  Then $J^{\Kos} = J^{\Hir} = J^{\alt} = J^{\cl_\pi}$ for $\pi$ any regular alteration.
\end{corollary}
\begin{proof}
    We have the containments $\subseteq$ so it suffices to show that $J^{\cl_\pi} \subseteq J^{\Kos}$.
    We know 
    \[
        J^{\cl_\pi} \subseteq \bigcap_{n>0}(x_1, \dots, x_i, x_{i+1}^n, \dots, x_{d}^n)^{\cl_{\pi}} = \bigcap_{n>0}(x_1, \dots, x_i, x_{i+1}^n, \dots, x_{d}^n)^{\Kos} = J^{\Kos}
    \]
    where we have used \autoref{cor.ClosuresAreAllEqualForFullParameterIdeals} for the second equality and \autoref{prop.IntersectionContainments} for the last equality.  The result follows.
\end{proof}

As $J^{\cl_{\pi}}$ and $J^{\KH}$ commute with localization, this result generalizes outside the local case to any ideal $J = (f_1, \dots, f_i)$ such that $f_1, \dots, f_i$ is part of a system of parameters in every localization $R_Q$ with $J \subseteq Q \in \Spec R$.
}

\section{Connections to positive characteristic}

Suppose $R$ is as in \autoref{set.NiceSetup}.  As one varies over all regular alterations $\pi : Y \to \Spec R$, one obtains the multiplier modules / Grauert-Riemenschneider modules $\Gamma(\omega_Y)$ for all finite integral extensions $R \subseteq S$.  These modules reduce modulo-$p$ to parameter test modules $\tau(\omega_S)$ of \cite{SmithTestIdealsInLocalRings} as we observed in the proof of \autoref{cor.KHTestIdealReducesToParameterTestIdeal}.  

It is thus natural to ask what happens if we construct a characteristic $p > 0$ closure operation using the modules $\tau(\omega_S)$ as one varies over all finite domain extensions $R \subseteq S$ (in analogy to the canonical alteration closure of \autoref{def.CanonicalAlterationClosure}).  In fact, below in \autoref{thm.TightClosureEqualsModuleClosureParameterTestModules}, we prove that this closure coincides with tight closure under mild hypotheses.

We briefly recall the notion of tight interior from \cite{tightinterior}.
\begin{definition}
    Suppose $R$ is a ring of characteristic $p > 0$ and $M$ is an $R$-module.  We define the tight interior $M_*$ of $M$ to be 
    \[
        \bigcap_{c \in R^{\circ}} \bigcap_{e_0 \geq 0} \sum_{e \geq e_0} \Image\Big( \Hom_R(F^e_* R, M) \xrightarrow{g \mapsto g(F^e_* c)} M\Big).
    \]    
\end{definition}
If $R$ is $F$-finite, reduced and one chooses $c \in R^{\circ}$ a big test element, then in fact we have 
    \[
        M_* = \sum_{e \geq e_0} \Image\Big( \Hom_R(F^e_* R, M) \xrightarrow{g \mapsto g(F^e_* c)} M\Big)
    \]
by \cite[Theorem 2.5]{tightinterior}.

We introduce one other useful definition.
\begin{defn}
    An $R$-module $W$ is called a \emph{durable test module} if for all $R$-module inclusions $L \subseteq M$, we have $\img (1_W \otimes_R (L \into M)) = \img(1_W \otimes_R (L^*_M \into M))$.  Here $1_W$ denotes the identity map on $W$ and so the identification of images just means that $\im(W \otimes L \to W \otimes M) = \im(W \otimes L_M^* \to W \otimes M)$ under the canonical maps.
\end{defn}

\begin{rem}
When $R$ is a prime characteristic Noetherian ring, Huneke \cite{Hu-strongtest} defined a \emph{strong test ideal} to be an ideal $T$ such that for all ideals $I$ of $R$, we have \begin{equation}\label{eq:strongtest}
    I^*T=IT.
\end{equation}  Clearly any such $T$ is contained in the finitistic test ideal of $R$, and Huneke showed that one can often find such an ideal that is not in any minimal prime of $R$.  Enescu \cite{En-strongtest} defined the more general notion of a \emph{strong test module} to be an $R$-module $T$ that satisfies \eqref{eq:strongtest}, and showed that the parameter test submodule is frequently a strong test module.

Note that the notion of durable test modules is a generalization of Enescu's strong test modules, since for any $R$-module $W$ and ideal $I$, we have $\img(1_W \otimes (I \hookrightarrow R)) = IW$.
\end{rem}

\begin{prop}\label{pr:tintdurable}
Let $R$ be an F-finite ring. Let $U, L, M$ be $R$-modules.  Let $j: L \into M$ and $j': L^*_M \into M$ be inclusion maps.  Then $\img(1_{U_*} \otimes j) = \img(1_{U_*} \otimes j')$.

That is, the tight interior of any $R$-module is a durable test module for $R$.
\end{prop}

\begin{proof}
Let $u \in U_*$ and $z \in L^*_M$.  It is enough to show that $u \otimes z$ is in the image of $1_{U_*} \otimes j$.  Let $c\in R^\circ$ and $e_0 \in \N$ such that $cz^q \in L^{[q]}_M$ for all $e \geq e_0$, where $q=p^e$.  Then by definition of $U_*$, there exist $e_1, \ldots, e_n \in [e_0, \infty) \cap \N$ and $g_1, \ldots, g_n$ such that for each $1\leq i \leq n$, $g_i \in \Hom_R(F^{e_i}_*(R), U)$ and $\sum_{i=1}^n g_i(F^{e_i}_*(c)) =u$.  Then it is enough to show that for each $i$, $g_i(F^{e_i}_*(c)) \otimes z \in \img(1_{U_*} \otimes j)$, so from now on we simplify notation by fixing $i$ and setting $g:=g_i$ and $e:=e_i$.

Consider the following commutative diagram: \[\xymatrix{
F^e_*(R) \otimes_R L \ar[d]^\alpha \ar[r]^{f_1}& F^e_*(R) \otimes_R L^*_M \ar[d]^\beta \ar[r]^{f_2}& F^e_*(R) \otimes_R M \ar[d]^\gamma\\
U_* \otimes_R L \ar[r]^{h_1}&U_* \otimes_R L^*_M \ar[r]^{h_2}&U_* \otimes_R M.
}
\]
We define the maps above by setting $\alpha := g \otimes_R 1_L$, $\beta := g \otimes_R 1_{L^*_M}$, $\gamma := g \otimes_R 1_M$, and each of the horizontal maps is the tensor product of an identity map with an inclusion map.

Then $y := g(F^e_*(c)) \otimes z$ (in $U_* \otimes L^*_M$) equals $\beta(t)$ (where $t:=F^e_*(c) \otimes z \in F^e_*(R) \otimes L^*_M$), but $f_2(t) = F^e_*(c) \otimes z$ (in $F^e_*(R) \otimes M$) $\in \img(1_{F^e_*(R)} \otimes j) = \img(f_2 \circ f_1)$.  That is, there are some $d_j \in R$ and $\ell_j \in L$, $1 \leq j \leq m$, such that if we set $a:= \sum_{j=1}^m F^e_*(d_j) \otimes \ell_j \in F^e_*(R) \otimes_R L$, we have $f_2(t) = F^e_*(c) \otimes z = \sum_{j=1}^m F^e_*(d_j) \otimes \ell_j =  f_2(f_1(a))$.  Thus, in $U_* \otimes M$, we have \begin{align*}
  g(F^e_*(c)) \otimes z &= h_2(y) = h_2(\beta(t)) = \gamma(f_2(t)) = \gamma(f_2(f_1(a)))\\
  &= h_2(h_1(\alpha(a))) \in \img (h_2 \circ h_1) = \img (1_{U_*} \otimes j),
\end{align*}
as was to be shown.
\end{proof}

We recover the following corollary -- a variant of Enescu's result.

\begin{cor}[{\cf \cite[Corollary 2.6]{En-strongtest}}]\label{cor:taudurable}
    For any F-finite reduced ring $R$, the big test ideal is a durable test module.  If $R$ is additionally locally equidimensional, then the parameter test submodule is also a durable test module.
\end{cor}
\begin{proof}
    The big test ideal is the tight interior of $R$ by \cite[Proposition 2.3]{tightinterior}.  For the parameter test submodule, see  \autoref{lem.InteriorOfOmegaIsTauOmega} below.
\end{proof}

 We assume it is well known to experts that the parameter test submodule $\tau(\omega_R) \subseteq \omega_R$ is the tight interior of $\omega_R$, but we do not know of a reference so we provide a short proof.

\begin{lemma}
\label{lem.InteriorOfOmegaIsTauOmega}
        Suppose $R$ is an $F$-finite reduced and locally equidimensional Noetherian ring, then $\tau(\omega_R) = (\omega_R)_*$ as submodules of $\omega_R$.
    \end{lemma}
    \begin{proof}
        By \cite[Chapter 2, Corollary 5.8]{SchwedeSmith.FBook} and \cite[Corollary 2.11]{tightinterior}, the formation of $\tau(\omega_R)$ and $(\omega_R)_*$ commute with localization and completion.  Hence we may assume that $R$ is complete and local (since two submodules $L_1, L_2 \subseteq \omega_R$ coincide if and only if they coincide after completion at each maximal ideal).  But now, both submodules Matlis dualize to $H^d_{\fm}(R)/0^*_{H^d_{\fm}(R)}$, where $d=\dim R$; see \cite[Proposition 3.5]{tightinterior} for $(\omega_R)_*$. 
        The lemma follows.
    \end{proof}

\begin{thm}
    \label{thm.TightClosureEqualsModuleClosureParameterTestModules}
    Let $R$ be an F-finite domain.  Then for any $R$-submodule inclusion $L \into M$, we have \[
    L^*_M = \bigcap_S L^{\cl_{\tau(\omega_S)}}_M =: L_M^{\cl},
    \]
    where the intersection is taken over all module-finite domain extensions $R \ra S$.
\end{thm}

\begin{proof}
    First let $z\in L^\cl_M$.  Let $T$ be an ideal of $R$ such that $\tau(\omega_R) \cong T$ as $R$-modules.  
    Then for each $q=p^e$, there is an $R^{1/q}$-module isomorphism $\phi: T^{1/q} \ra \tau(\omega_{R^{1/q}})$.  Choose a nonzero $c\in T$.  Then since $R \into R^{1/q}$ is module-finite by assumption, letting $a=\phi(c^{1/q})$ we have \begin{align*}
    a \otimes z &\in \img(\tau(\omega_{R^{1/q}}) \otimes_R L \ra \tau(\omega_{R^{1/q}}) \otimes_R M) 
    \end{align*}
    by definition of $\cl$.  Then applying the maps $\phi^{-1} \otimes 1$, we have $c^{1/q} \otimes z \in \img(T^{1/q} \otimes_R L \ra T^{1/q} \otimes_R M)$, so that using the injection $T^{1/q} \into R^{1/q}$, we also get $c^{1/q} \otimes z \in \img(R^{1/q} \otimes_R L \ra R^{1/q} \otimes_R M)$. That is, $cz^q \in L^{[q]}_M$, so by definition $z\in L^*_M$. 

    Conversely let $z \in L^*_M$.  Let $R \into S$ be a module-finite domain extension and let $a \in \tau(\omega_S) =: U$.  Then in the module $U \otimes_S (S \otimes_R M)$, we have the following, where the first equality is because $U$ is a durable test module for $S$ (see \autoref{cor:taudurable}).  We are using the convention that if $C$ is a submodule of $M$, then $SC = \img(S \otimes_R C \ra S \otimes_R M)$ induced by the inclusion map $C \into M$: \begin{align*}
        a \otimes (1 \otimes z) &\in \img (U \otimes_S (SL)^*_{S \otimes M} \ra U \otimes_S (S \otimes_R M)) \\
        &= \img(U \otimes_S (SL) \ra U \otimes_S (S \otimes_R M)) \\
        &= \img (U \otimes_S (S \otimes_R L) \ra U \otimes_S (S \otimes_R M)).
    \end{align*}
    By isomorphism of the functors $U \otimes_S (S \otimes_R -)$ and $U \otimes_R -$ on ${}_R$Mod, it follows that \[
    a \otimes z \in \img(U \otimes_R L \ra U \otimes_R M).
    \]
    Thus, $z \in L^\cl_M$.
\end{proof}

\begin{remark}
    Let $R$ be an F-finite domain. Under geometric hypotheses, for instance if $R$ is essentially of finite type over a perfect field, for every finite extension $R \subseteq S$, we know by \cite{deJongAlterations,BlickleSchwedeTuckerTestAlterations}, that there exists a regular alteration ${Y_S} \to \Spec S$ such that $\Gamma(\omega_{Y_S}) \to \omega_S$ has image $\tau(\omega_S)$, that is, there is a surjection $\Gamma(\omega_{Y_S})\twoheadrightarrow  \tau(\omega_S)$.  Hence by \cite[Proposition 2.20]{perezrg} we have that 
    \[
        L_M^* = \bigcap_S L_M^{\cl_{\tau(\omega_S)}} \supseteq \bigcap_S L_M^{\cl_{\Gamma(\omega_{Y_S})}}.
    \]
    In particular, if one intersects over all regular alterations $Y \to \Spec R$, one obtains 
    \[
        L_M^* \supseteq \bigcap_Y L_M^{\cl_{\Gamma(\omega_{Y})}}.
    \]
    Here, the right side is a naive definition of $\alt$-closure in characteristic $p > 0$.  Notice, however, that in characteristic $p > 0$ we need not have $\Gamma(\omega_{Y}) = \myR \Gamma(\omega_{Y})$, see for instance \cite[Example 3.11]{HaconKovacsGenericVanishingFailsForSingularVarieties}.  
\end{remark}

\subsection{Applications to canonical alteration closure}

We can now compare canonical alteration closure with tight closure in equal characteristic zero obtained via reduction to characteristic $p > 0$.

\begin{proposition}
    \label{prop.TCIsContainedInCaltChar0}
    Suppose $k$ is a field of characteristic zero and $R$ is a domain of finite type over $k$, and let $L \subseteq M$ be finitely generated $R$-modules.  
    Then $L_M^{*} \subseteq L_M^{\alt}$ where $L_M^*$ denotes reduction modulo $p$ tight closure \cite{HochsterHunekeTightClosureInEqualCharactersticZero}.  {In fact, after reduction modulo any $p \gg 0$, we have that $(L_p)_{M_p}^* \subseteq (L_M^{\cl_{\pi}})_p$ where the $(-)_p$ denotes reduction modulo $p$. }
\end{proposition}

\begin{proof}
    We prove the second statement as it clearly implies the first.  
    Fix $Y \to \Spec R$ an alteration and set $S = \Gamma(\cO_Y)$, a finite extension of $R$.  
    Fix generators $b_1, \dots, b_t \in \Gamma(\omega_Y)$.  Note that if $z \in \ker\big( M \xrightarrow{m \mapsto \bar m \otimes b_i} M/L \otimes_R \Gamma(\omega_Y) \big)$ for all $i = 1, \dots, t$, then for any $b = \sum a_i b_i$ we see that $z \in \ker\big( M \xrightarrow{m \mapsto \bar m \otimes b} M/L \otimes_R \Gamma(\omega_Y) \big)$.  Furthermore, we know that $\Gamma(\omega_Y)$ reduces modulo $p \gg 0$ to $\tau(\omega_{S_p})$ by \autoref{prop.JOmegaReducesToTauOmega}.  It follows that 
    \[  
        (L_p)_{M_p}^{\cl_{\tau(\omega_{S_p})}} = (L_M^{\cl_\pi})_p
    \]
    for $p \gg 0$.  Thanks to \autoref{thm.TightClosureEqualsModuleClosureParameterTestModules}, we know that $(L_p)_{M_p}^* \subseteq (L_p)_{M_p}^{\cl_{\tau(\omega_{S_p})}}$.
    %
    %
    This completes the proof.
\end{proof}

\begin{remark}
\label{rem.ClosureForInfiniteInCalt}
Suppose there exists an infinite set of $p > 0$ such that $x_p \in (L_p)_{M_p}^*$.  The above result then implies we have that $x \in L_M^{\alt}$.  We explain this more precisely using the notation of \autoref{subsec.ReductionModP}: we assume that for a Zariski dense but not necessarily open set $V \subseteq \mathfrak{m}\text{-}\Spec A$ that $x_{\bf t} \in (L_{\bf t})_{M_{\bf t}}^*$ for all ${\bf t} \in V$.  The result above implies that $x \in L_M^{\alt}$.  In particular, thanks to \cite{BrennerKatzman.ArithmeticOfTightClosure}, if $R = \bQ[x,y,z]/(x^7+y^7+z^7)$, we know that $x^3y^3 \in (x^4,y^4,z^4)^{\alt}$ even though $x^3y^3$ is not in the usual  tight closure in equal characteristic zero $(x^4,y^4,z^4)^{*}$.  Thus, $I^\alt$ can contain $I^*$ strictly.
\end{remark}




{Finally, our results give a concrete description of the tight closure of parameter ideals as we vary $p \gg 0$.  In the Cohen-Macaulay standard graded case such that the ring is $F$-rational outside the irrelevant ideal, a different (although related) precise description of the tight closure of a full parameter ideal for $p \gg 0$ was given in \cite[Proposition 6.2(iii)]{Hara.GeometricInterpretationOfTightClosureAndTestIdeals}.  In the Gorenstein case, another related characterization ($J^* = J : \mJ(R)$) easily follows from \cite[Corollary 4.2(2)]{Huneke.TCParamAndGeometry}, or from \cite[Theorem 1]{Vraciu.TightClosureAndLinkageGorenstein} in view of the fact that the multiplier ideal $\mJ(R)$ reduces modulo $p \gg 0$ to the test ideal $\tau(R_p)$, thanks to \cite{SmithMultiplierTestIdeals,Hara.GeometricInterpretationOfTightClosureAndTestIdeals}.  In  \cite[Theorem 5.24]{Yamaguchi.CharacterizationOfMultiplierIdealsViaUltra} it is pointed out that the argument of Huneke mentioned above generalizes to the quasi-Gorenstein case, even for $\mathfrak{a}^t$-tight closure.  We have learned that some experts know other characterizations of the behavior of tight closure of parameter ideals modulo $p \gg 0$, but we are not aware of a reference for the behavior of parameter ideals modulo $p$ in the generality we obtain below.

\begin{theorem}
\label{thm.AllOurClosuresAgreeWithTightClosureForParameter}
    Suppose $R$ is a domain of finite type over a field of characteristic zero.  Suppose that $J = (f_1, \dots, f_h)$ is an ideal such that $f_1, \dots, f_h$ is part of a system of parameters at every localization of $R$ at a prime containing $J$.  Then 
    \[
        J^{\Kos} = J^* = J^{\cl_\pi} = (J \Gamma(\omega_Y)) : \Gamma(\omega_Y)
    \]
    where $\pi : Y \to \Spec R$ is a regular alteration and $J^*$ denotes the reduction modulo $p$ version of tight closure.  Furthermore, for all $p \gg 0$, $(J^{\Kos})_p = (J_p)^* = (J^{\cl_\pi})_p$. Hence $J_p^*$ is constant for $p \gg 0$.
\end{theorem}
\begin{proof}
    The ideal $J$ is called a \emph{parameter ideal} in \cite{HochsterHunekeTightClosureInEqualCharactersticZero}.  As both the formation of $J^{\Kos}$ and $J^{\cl_\pi}$ commute with localization, we see that $J^{\Kos} = J^{\cl_\pi}$ by \autoref{cor.ClosuresCoincideInChar0ForParameterIdeals} as we check the statement after localizing at primes containing $J$.  
    We also obtain that $(J^{\Kos})_p \subseteq (J_p)^*$ for $p \gg 0$ thanks to \autoref{prop.KHContainedInPlusModP} while $(J_p)^* \subseteq (J^{\cl_\pi})_p$ for $p \gg 0$ by \autoref{prop.TCIsContainedInCaltChar0}.  The final equality with the colon is a consequence of \autoref{eq.ModuleClosureForIdeals}.  The result follows.
\end{proof}
In \cite[Corollary 4.1]{schoutenscanonicalbigcmalgebras}, it is shown that $J^*$ also coincides with expansion and contraction to a certain big Cohen-Macaulay algebra for $J$ a parameter ideal.
}

\section{Further questions}

We end the paper with some questions.

\subsection{Some questions on \KosClos}

\begin{question}
    Is there a characteristic $p > 0$ closure that is more closely related to $\KH$-closure?  For instance, consider the smallest closure operation on ideals in characteristic $p > 0$ that contains tight closure\footnote{equivalently plus closure; see  \cite{plusclosure}} of parameter ideals and is persistent.  How does that compare to \KosClos in {equal} characteristic zero?  Does that closure operation commute with localization?  Does it characterize $F$-rational singularities?
\end{question}

One important question is whether or not \KosClos can be extended to a module closure.

\begin{question}
    Does \KosClos extend to a closure { on modules}?   { If so, does that extended closure} satisfy Dietz's generalized colon capturing axiom \cite{dietz} (see also \autoref{subsec.ClosureOperations})?  Does { the} Hironaka {pre}closure satisfy Dietz's generalized colon capturing axiom?
\end{question}

A positive answer to the above question would imply that \KosClos induces a big Cohen-Macaulay module.  One possibility is to use the Buchsbaum-Rim complex \cite{BuchsbaumRimAGeneralizedKoszulComplex1} instead of the Koszul complex.

We could also hope that the characterization of $\tau_{\KH}(R)$ from \autoref{prop.TestIdealForKosClosure} extends beyond the Cohen-Macaulay case.
\begin{question}
    Is the \traceIdealComplex of $\myR \Gamma(\cO_Y)$ equal to $\tau_{\KH}(R)$ in general?
\end{question}

    In \cite{bigcmalgebraaxiom}, the third named author explored a condition on closure operations which guarantees the existence of a big Cohen-Macaulay \emph{algebra} (and which any closure operation induced by a big Cohen-Macaulay algebra satisfies), also see \cite{MurayamaUniformBoundsOnSymbolicPowers} for an alternate approach.  In our current paper, we studied a closure operation on ideals induced by a Cohen-Macaulay complex that is also a cosimplicial algebra / differential graded algebra.

    For any differential graded $R$-algebra $A$, one could form the associated Koszul-type closure on ideals $I = (\underline{f})$ via $\ker\big(R \to H_0(K_{\mydot}(\underline{f}; A))\big)$ (the argument of \autoref{prop.kosIndependence} applies).  We tentatively call this a \emph{Koszul dg algebra closure}, and denote it by $I^{\mathrm{K}_A}$.  { Alternately, one could consider the $A$-preclosure defined by $I^A := \ker\big(R \to H_0(R/I \otimes^{\myL} A)\big)$.}

    It is thus natural to ask the following.

\begin{question}
    What properties of an ideal closure operation are unique to being induced by a differential graded algebra as above?  What if it is also a (locally/big) Cohen-Macaulay complex?
\end{question}

    It is worth noting that the properties we proved about \KosClos and Hironaka preclosure are also common to any context where we have a sufficiently (weakly) functorial association from Noetherian rings $R$ to a (hopefully) locally Cohen-Macaulay (over $R$) differential graded algebra.  
    Besides of course the usual weakly functorial associations to (non-derived) big Cohen-Macaulay algebras, we do not know any other general way to produce such Cohen-Macaulay differential graded algebras except for resolution of singularities in {equal} characteristic zero.

In characteristic $p > 0$, Frobenius closure also appears prominently.  It is natural to ask if there is a corresponding operation in {equal} characteristic zero.  We propose the following.

\begin{definition}
    Suppose $R$ satisfies \autoref{set.NiceSetup} and $I = (f_1, \dots, f_n)$ is an ideal of $R$.  We define the \emph{Koszul-Du Bois closure of $I$}, denoted $I^{\mathrm{KD}}$ to be 
    \[
        \ker\Big( R \to H_0\big(K_{\mydot}({\bf f}; R) \otimes^{\myL} \DuBois{R}\big) \Big)
    \]
    where $\DuBois{R} = \myR \Gamma(\Spec R, \DuBois{\Spec R})$ is the $0$th graded piece of the Deligne-Du Bois complex of $R$ (it can be viewed as a cosimplicial algebra), see \cite{DuBoisMain,GNPP,PetersSteenbrinkMixedHodgeStructures,MurayamaInjectivityAndCubicalDescentForSchemesStacksSpaces}.  When $R$ is Cohen-Macaulay and essentially of finite type over a field of characteristic zero, we know that  $\DuBois{\Spec R}$ is a locally Cohen-Macaulay complex by \cite{KovacsSchwedeDBDeforms}.
\end{definition}

Again, using the argument of \autoref{prop.kosIndependence}, it is not difficult to see it is well defined and indeed a closure operation as $\DuBois{R}$ can be viewed as a cosimplicial algebra.  When $R$ is essentially of finite type over a field, we know $\DuBois{R}$ is at least ``close'' to being locally Cohen-Macaulay and in fact it is locally Cohen-Macaulay if $R$ is, by the Matlis dual version of \cite[Theorem 3.3]{KovacsSchwedeDBDeforms}.  We have not studied this operation deeply however.  

We do point out that the existing Macaulay2 package can be used to compute it if $R$ is normal and Cohen-Macaulay.  In that case, if $\pi : Y \to \Spec R$ is a resolution of singularities with SNC exceptional divisor $E$, then $\myR \Hom_R(\DuBois{R}, \omega_R) = \Gamma(\omega_Y(E))$ thanks to \cite[Theorem 3.8]{KovacsSchwedeSmithLCImpliesDuBois} and \cite[Theorem 3.3]{KovacsSchwedeDBDeforms}.  Thus one may call {\tt koszulHironakaClosure(I, M)} where {\tt M} is the module $\omega_Y(E)$ and this will compute ${\tt I}^{KD}.$

\subsection{Questions on canonical alteration closure}
{%


Some of the good properties of canonical alteration closure, such as colon capturing, are simply deduced from \KosClos.  It would be natural to try to prove them directly.  Since canonical alteration closure is defined for modules, it is also natural to hope that it satisfies Dietz's generalized colon capturing \cite{dietz}, see \autoref{subsec.ClosureOperations} for a precise statement.

\begin{question}
    Does canonical alteration closure satisfy Dietz's generalized colon capturing axiom?  
\end{question}

Perhaps the most natural question, based on the results of the previous section, is the following:

\begin{question}
    Suppose $R_{\bZ}$ is a Noetherian domain finite type over $\bZ$, $J_{\bZ} \subseteq R_{\bZ}$ is an ideal, 
    and we have base changes $R_{\bQ} = R_{\bZ} \otimes_{\bZ} \bQ$ and $R_p = R_{\bZ} \otimes_{\bZ} \mathbb{F}_p$, 
    and likewise with $J_{\bQ}$ and $J_p$.  
    Consider $x_\bZ\in R_\bZ$ such that $x_\bQ \in R_\bQ$ is in $(J_{\bQ})^{\alt}$.  Does there exist an infinite set of primes $p > 0$ such that $x_p \in (J_p)^*$?  More generally, is the mod $p$ reduction of $(J_{\bQ})^{\alt}$ equal to $(J_p)^*$ for infinitely many $p > 0$?
\end{question}
This question could of course be generalized to finite type domains over various $\bZ$-algebras as in the more general reduction modulo $p > 0$ setup of \autoref{subsec.ReductionModP}.  See also \autoref{rem.ClosureForInfiniteInCalt} as well as \cite{BrennerKatzman.ArithmeticOfTightClosure, HochsterHunekeTightClosureInEqualCharactersticZero} for related discussion.

\appendix
\section{Modules of finite phantom projective dimension}\label{appendix}

{
The story of phantom projective dimension begins with the theorem of Buchsbaum and Eisenbud characterizing when a finite complex of finite free modules starting with the zero map is acyclic \cite{BuchsbaumEisenbud.WhatMakesAComplexExact}.  That is, they showed that a complex can be identified as a finite free resolution of a module (necessarily the zeroth homology of the complex) by examining the ranks of the matrices and the grades of their ideals of minors (called the \emph{standard conditions on rank and depth}). 

Some 15 years later, Hochster and Huneke \cite{HHmain} showed that some weaker conditions (called the standard conditions on rank and \emph{height}) on a complex could, under mild hypotheses in the characteristic $p$ setting, characterize when the complex is a \emph{phantom} resolution of a module -- i.e., that the complex and its Frobenius iterates are acyclic up to tight closure.  In that paper and later work, it is shown that modules with `phantom resolutions' have nice properties, analogous to nice properties enjoyed by modules with authentic finite resolutions (i.e., modules of finite projective dimension). For instance, Aberbach showed that tight closure equals plus closure, and thus commutes with localization, for finite modules $L \subseteq M$ when $M/L$ has a phantom resolution, assuming the ring is locally excellent \cite[Theorem 3.1 and Corollary 3.2]{AbFrat}.  He also gave a phantom Auslander-Buchsbaum formula over local rings that are homomorphic images of Gorenstein rings \cite[Theorem 3.2.7]{Aberbach.FinitePhantom}. The role of the \emph{depth} of $R$ is replaced by that of \emph{dimension}, and the role of \emph{projective dimension} is replaced by the length of the shortest \emph{phantom} resolution of the module.

This then led to the more general point of view, which says that if a complex $P$ satisfies the standard conditions on rank and height, then the zeroth homology of the complex should satisfy conditions analogous to being of finite projective dimension.  For a much more recent use of that viewpoint, see Kimura's result  \autoref{prop.SRHimplication} later in this appendix.
}


We begin with several definitions.



\begin{definition}
    Let $R$ be a Noetherian ring and let $\alpha\colon F\to G$ be a map of finite free $R$ modules. The \emph{rank} of $\alpha$ is the largest integer $r$ such that the map $\bigwedge^r\alpha\colon\bigwedge^r F\to\bigwedge^rG$ is nonzero. Let $I_t(\alpha)$ denote the ideal of $t$-by-$t$ minors of $\alpha$. Then $\rank(\alpha)$ equals the largest $t$ such that $I_t(\alpha)\neq0$.
\end{definition}

\begin{definition}[{\cite[9.6]{HHmain}}]
    Let $R$ be a Noetherian ring and let $P$ be a complex of finitely generated free modules concentrated in homological degrees $[n,0]$ with differentials $\partial_i\colon P_i\to P_{i-1}$. Let $b_i=\rank P_i$ and let $r_i=\sum_{j=i}^n(-1)^{j-i}b_j$ for each $1\leq i\leq n$. We say that $P$ satisfies the \emph{standard condition on rank} if $\rank\partial_i=r_i$ for $1\leq i\leq n$. We say that $P$ satisfies the \emph{standard condition on height} if $\height I_{r_i}(\partial_i)\geq i$ for each $1\leq i\leq n$.
\end{definition}

We now use these definitions to describe what it means for a submodule to have a phantom resolution.

\begin{definition}
    Suppose $R$ is a Noetherian ring and $L \subseteq M$ is an inclusion of finitely generated $R$-modules.  We say that $M/L$ has \emph{\CMPhantom} (\emph{finite CM-phantom projective dimension}) if there exists a finite complex $P_{\mydot}$ of finitely generated free modules concentrated in degrees $[n,0]$, such that 
    \begin{enumerate}
        \item $P_{\mydot}$ satisfies the standard conditions on rank and height, and
        \item $H_0(P_{\mydot}) \cong M/L$.  
    \end{enumerate}
    We say that $M/L$ has \emph{\locCMPhantom} if for every $Q \in \Spec R$, $L_Q \subseteq M_Q$ has \CMPhantom.
\end{definition}


Although we will not need it, it is useful to remind the reader of the following.

\begin{prop}[{\cite[Theorem~9.8]{HHmain}}]
    Let $R$ be a reduced local Noetherian ring of positive characteristic that is locally equidimensional and a homomorphic image of a Cohen-Macaulay ring.  Let $J$ be an ideal. Then $J \subseteq R$ has finite phantom projective dimension (in the sense of \cite{Aberbach.FinitePhantom}) if and only if it has \CMPhantom. 
\end{prop}

In this appendix, we show that if $R$ has characteristic zero and if $M/L$ has \CMPhantom, then various characteristic zero closure operations we introduced in this paper all coincide (and also coincide with an analog of $J^{\KH}$).  We will see that this generalizes the Koszul complex results we obtained above.  An important consequence is that, for \CMPhantom ideals and submodules (a class much larger than parameter ideals), tight closure is constant for $p\gg0$.  This is quite unlike the case of general ideals and modules; see \cite{BrennerKatzman.ArithmeticOfTightClosure}. 



For us, the main implication of the standard conditions on rank and height is the following result of Kimura (\autoref{prop.SRHimplication}), which was pointed out to us by Linquan Ma.  Also compare with \cite[Theorem 1.2.3]{Aberbach.FinitePhantom} and \cite[Section 3.4]{Iyengar.DepthForComplexes} in the big Cohen-Macaulay module/algebra case.


First, however, we adopt the following convention on dualizing complexes for the purpose of this appendix, following Kimura's conventions.

\begin{convention}[Shifted dualizing complexes]
    \label{conv.ShiftedDualizingComplex}
    Throughout this appendix when we talk about a dualizing complex $D$ on a local ring $R$, we always mean a complex in $D^{\geq 0}(R) = D_{\leq 0}(R)$ such that $H_0(D) \neq 0$.  In particular, this is normalized in a different way than the dualizing complexes we worked with earlier in the paper (it is shifted by $[-\dim R]$).  Because of this distinction, we use $D$ for the dualizing complex instead of $\omega_R^{\mydot}$ as used earlier in the paper.
\end{convention}

\begin{prop}[{\cite[Proposition~3.6(2)]{kimura2025traceidealsconductorsideals}}]\label{prop.SRHimplication}
    Let $R$ be a Noetherian local ring 
    with a dualizing complex $D$ (as in \autoref{conv.ShiftedDualizingComplex}) 
    let $M$ be a finitely generated $R$ module. If $F$ is a complex of finitely generated free $R$-modules satisfying the standard conditions on rank and height, then $H_i(F\otimes_R^\bL\myR\Hom(M, D))=0$ for $i>0$. In particular, $H_i(F\otimes_R^\bL D)=0$ for $i>0$.
\end{prop}

We now come to our main result of the appendix, which is that all our operations agree for (locally) \CMPhantom submodules.

\begin{prop}
\label{prop.PhantomHirCalt}
    Suppose $R$ is a domain satisfying \autoref{set.NiceSetup} and let $L\subseteq M$ be a pair of finitely generated $R$-modules. 
    Suppose $M/L$ has \locCMPhantom.  
    Then
    \[L^{\Hir}_M=L^\alt_M=L^{\cl_\pi}_M\]
    where $\pi\colon Y\to\Spec R$ is any regular alteration. In fact, if $M/L$ has \CMPhantom represented by a complex $P$, then all of these also agree with 
    \[
        \ker\big(M\to H_0(P\otimes_R^\bL\myR\Gamma(\cO_Y))\big).
    \]
\end{prop}

\begin{proof}
    First, observe we have 
    \[
        L^{\Hir}_M \subseteq L^\alt_M \subseteq L^{\cl_\pi}_M
    \]
    and so it suffices to show the terms on the far left and right agree for the first statement.  As those terms commute with localization, we may assume that $R$ is sufficiently local and hence we may assume that $M/L$ has \CMPhantom with associated perfect complex $P$.  Then as $\ker\big(M\to H_0(P\otimes_R^\bL\myR\Gamma(\cO_Y))\big) \subseteq L^{\Hir}_M$ (thanks to the map $P \to H_0(P) \cong M/L$), it suffices to show that $\ker\big(M\to H_0(P\otimes_R^\bL\myR\Gamma(\cO_Y))\big) =  L^{\cl_\pi}_M$.

    Note that we may assume, without loss of generality, that $M=P_0$ is a finitely generated free module.  This follows from the fact that our (pre)closure operations are residual. 
    Let $P$ be our complex satisfying the standard conditions on rank and height such that $H_0(P)=M/L$. Consider the following
    \[K_P:=\ker(M\to H_0(P\otimes_R^\bL\myR\Gamma(\cO_Y))).\]
    By the map $P\to M/L$ we see that $K_P\subseteq L^{\Hir}_M$ and so it suffices to show that $L^{\cl_\pi}_M\subseteq K_P$.
    
    Observe that, because $M=P_0$ and the map $P_1\to P_0$ surjects onto $L \subseteq P_0$, we have
    \[L^{\cl_\pi}_M= \bigcap_{z \in \Gamma(\omega_Y)} \big\{m\in M~|~m\otimes z\in\Image\big(P_1\otimes_R\Gamma(\omega_Y)\to P_0\otimes_R\Gamma(\omega_Y)\big)\big\}.\]

    Next, let $D$ be a dualizing complex for $R$ as in \autoref{conv.ShiftedDualizingComplex} and represented by a complex in homological degrees $[-n, 0]$ and consider the series of natural isomorphisms
    \[
        P\otimes_R^\bL\myR\Gamma(\cO_Y)\cong P\otimes_R^\bL\myR\Hom_R(\Gamma(\omega_Y),D)\cong \myR\Hom_R(\Gamma(\omega_Y),P\otimes_R^\bL D)
    \]
    where the first isomorphism is simply Grothendieck duality and the second crucially uses that $P$ has finite projective dimension (being a finite complex of free modules). Then 
    \[K_P=\ker\big(M\to\Hom_{D(R)}(\Gamma(\omega_Y),P\otimes_R^\bL D)\big).\]
    If we can understand this map, we will be able to see that $L^{\cl_\pi}_M\subseteq K_P$.

    Consider the sequence of maps in $D(R)$
    \[M\to P\otimes_R^\bL\myR\Hom_R(\Gamma(\omega_Y),D)\xrightarrow{\sim}\myR\Hom_R(\Gamma(\omega_Y),P\otimes_R^\bL D).\]
    Let $\phi_m$ be the image of $m\in M$ inside $H_0(\myR\Hom_R(\Gamma(\omega_Y),P\otimes_R^\bL D))\cong\Hom_{D(R)}(\Gamma(\omega_Y),P\otimes_R^\bL D)$. We would like to confirm that $\phi_m$ is the zero map when $m\in L^{\cl_\pi}_M$. Because $P$ satisfies the standard conditions on rank and height, $H_i(P\otimes_R D)=0$ for $i>0$ by \autoref{prop.SRHimplication} and this, combined with the fact that $\Gamma(\omega_Y)$ is concentrated in degree zero, tells us that $\varphi_m$ is the zero map in $D(R)$ if and only if $H_0(\varphi_m)=0$. Thus we just need to show that
    \[
    \Image((\phi_m)_0)\subseteq\Image(d^{P\otimes D}_1\colon(P\otimes_RD)_1\to(P\otimes_RD)_0).
    \]

    In fact, because $\Image((\phi_m)_0)$ lands in $P_0\otimes_RD_0$, we will show that each element in $\Image((\phi_m)_0)$ can be written as a sum of elements of the form $d^{P\otimes D}_1(m\otimes z)$ for $m\in P_1$ and $z\in\Gamma(\omega_Y)\subseteq\ker(D_0\to D_{-1})$.
    
    Let $F\to\Gamma(\omega_Y)$ be a free resolution with $F_i=0$ for $i<0$. Then we can represent the above by a map of honest complexes as
    \[M\to P\otimes_R\Hom_R(F,D)\to\Hom_R(F,P\otimes_R D).\]
    The first map sends $m\in M=P_0$ to $m\otimes\psi\in P_0\otimes_R\Hom_R(F,D)_0\cong P_0\otimes_R\Hom_R(F_0,D_0)$. Here, $\psi\colon F_0\to D_0$ corresponds, via the natural map $R\to\myR\Hom_R(\Gamma(\omega_Y),D)$, to the image of $1$ in $H^0(Y, \cO_Y) \cong \Hom_{D(R)}(\Gamma(\omega_Y),D)$, that is, the Grothendieck trace. The isomorphism $P_0\otimes_R\Hom_R(F_0,D_0)\to\Hom_R(F_0,P_0\otimes_RD_0)$ sends $m\otimes\psi\mapsto(z \mapsto m\otimes\psi(z))=(\phi_m)_0$.
    
    Consider the following commutative diagram
    \[
    \begin{tikzcd}
        P_1\otimes_R\Gamma(\omega_Y)\arrow[r]\arrow[d]&P_1\otimes_RD_0\arrow[d]\\
        P_0\otimes_R\Gamma(\omega_Y)\arrow[r]&P_0\otimes_RD_0
    \end{tikzcd}
    \]
    Let $z\in F_0$ and consider $m\otimes\psi(z)\in P_0\otimes_R D_0$. Note that the map $F_0\to D_0$ factors through $\Gamma(\omega_Y)$, so in fact there is some $\bar{z}\in\Gamma(\omega_Y)$ mapping to $\psi(z)\in D_0$. Because $m\in L^{\cl_\pi}_M$, $m\otimes \bar{z}\in\Image(P_1\otimes_R\Gamma(\omega_Y)\to P_0\otimes_R\Gamma(\omega_Y))$ for all $\bar{z}\in\Gamma(\omega_Y)$. This immediately implies that $m\otimes\psi(z)\in\Image(d^{P\otimes D}_1\colon(P\otimes_RD)_1\to(P\otimes_RD)_0)$ for all $z\in F_0$, which implies that $H_0(\phi_m)=0$ and thus that $m\in K_P$.
\end{proof}

We conclude with some corollaries.

\begin{corollary}
\label{cor.AppendixReductionModP}
    Suppose $R$ is a domain of finite type over a field of characteristic zero. Suppose that 
    $M/L$ has \locCMPhantom.  
    Then, for all $p\gg0$, $(L^{\Hir}_M)_p=(L_p)_{M_p}^+ = (L_p)_{M_p}^*=(L^{\cl_\pi}_M)_p$. Hence, $(L_p)_{M_p}^*$ is constant for $p\gg0$.
\end{corollary}
\begin{proof}     
    {This follows from \autoref{lem.HirIsInPlus}, \autoref{prop.TCIsContainedInCaltChar0},  and \autoref{prop.PhantomHirCalt} as $(L^{\Hir}_M)_p \subseteq (L_p)_{M_p}^* \subseteq (L^{\cl_\pi}_M)_p$.}    
\end{proof}

\begin{corollary}
\label{cor.HirIdempotentForPPD}
    Suppose $R$ is a Noetherian domain satisfying the conditions of \autoref{set.NiceSetup}. Suppose that $M/L$ has \locCMPhantom.  Then $(L^{\Hir}_M)_M^{\Hir} = L^{\Hir}_M$.  In other words, $\Hir$ is idempotent on $M/L$.
\end{corollary}
\begin{proof}
    Simply observe that $L^{\Hir}_M \subseteq (L^{\Hir}_M)_M^{\Hir} \subseteq (L^{\alt}_M)_M^{\alt} = L_M^{\alt} = L_M^{\Hir}$.
\end{proof}

\subsection{Appendix AI statement}
\label{subsec.app.AIDisclosure}
    The main result \autoref{prop.PhantomHirCalt} of this appendix was found in a discussion with an LLM.  Specifically, we asked OpenAI's ChatGPT 5.6 Sol to show that if $J \subseteq R$ is an ideal in a ring of characteristic zero such that $J \subseteq R$ has \CMPhantom\footnote{We spelled out the conditions precisely.} with associated perfect complex $P_{\mydot}$, then 
    \[
        \ker\big( R \to H_0(P_{\mydot} \otimes^{\myL} \myR\Gamma(\cO_Y)) \big) = J^{\alt}.
    \]
    In our prompt we also pointed out that Kimura's work \cite[Proposition 3.6]{kimura2025traceidealsconductorsideals}, recalled in \autoref{prop.SRHimplication}, was likely relevant.  The LLM then provided a correct proof of the prompted question.  
    
    After understanding the proof in the case of ideals, we generalized the result to the module case without LLM assistance (although we suspect that LLMs would have had no trouble writing out that generalization as well).  We then wrote up the result from our own perspective.  All the writing in this appendix was done by the human authors based upon their own understanding of the results.  ChatGPT 6 Astra and Claude Opus 5.5 were also used for the purpose of identifying typos, errors, and grammatical issues.

\bibliographystyle{skalpha}
\bibliography{alterationsrefs}
\end{document}